\documentclass[11 pt]{amsart}
\usepackage{amsmath,amssymb,amsthm, mathrsfs, mathtools}
\usepackage{amscd,mathtools}
\usepackage{fullpage}
\usepackage{graphicx,epstopdf}
\usepackage{color,xcolor}
\usepackage{enumerate}
\usepackage{xspace}
\usepackage{tipa}
\usepackage{stmaryrd}
  
\usepackage[utf8]{inputenc}
\usepackage[T1]{fontenc}
\usepackage{imakeidx}
\makeindex[columns=3, title=Alphabetical Index, intoc]

\usepackage[shortlabels]{enumitem}
\usepackage{epsfig}
\usepackage{pinlabel}
\usepackage[all]{xy}

\usepackage{tikz}
\usetikzlibrary{shapes.geometric}
\usepackage{caption, subcaption}
\usepackage{tikz-cd}
\usetikzlibrary{calc,decorations.pathreplacing}
\usepackage{tikz}
\usetikzlibrary{calc}
\usetikzlibrary{arrows}
\usetikzlibrary{decorations.pathreplacing}
\usetikzlibrary{intersections}
\usetikzlibrary{decorations.pathmorphing}

\usepackage{mathrsfs}
\usetikzlibrary{matrix}
\usetikzlibrary{positioning}
\DeclareMathAlphabet{\pazocal}{OMS}{zplm}{m}{n}
\tikzset{>=stealth}

\usepackage{hyperref}

  \newcommand{\calB}{\mathcal{B}}
  \newcommand{\calC}{\mathcal{C}}

  \newcommand{\calG}{\mathcal{G}}

  \newcommand{\calL}{\mathcal{L}}
  
  \newcommand{\calN}{\mathcal{N}}
  
  \newcommand{\calP}{\mathcal{P}}

  \newcommand{\calU}{\mathcal{U}}
    \newcommand{\calV}{\mathcal{V}}
  \newcommand{\calW}{\mathcal{W}}

  \newcommand{\calZ}{\mathcal{Z}}

  \newcommand{\NN}{\mathbb{N}}

  \newcommand{\RR}{\mathbb{R}}

  \newcommand{\gothic}{\mathfrak}

  \newcommand{\go}{{\gothic o}}

  \newcommand{\hG}{\widehat{G}}

  \newtheorem{theorem}{Theorem}[section]
  \newtheorem{proposition}[theorem]{Proposition}
  \newtheorem{corollary}[theorem]{Corollary}
  \newtheorem{lemma}[theorem]{Lemma}

  \newtheorem{introthm}{Theorem}

  \theoremstyle{definition}
  \newtheorem{definition}[theorem]{Definition}
  \newtheorem{claim}[theorem]{Claim}
  
  \newtheorem*{claim*}{Claim}
  
  \newtheorem{example}[theorem]{Example}
  
  \newtheorem*{question*}{Question}
  \newtheorem*{answer*}{Answer}
  \newtheorem*{application*}{Application}

  \theoremstyle{remark}
  \newtheorem{remark}[theorem]{Remark}
  \newtheorem*{remark*}{Remark}
  
  \newcommand{\thmref}[1]{Theorem~\ref{#1}}
  
  \newcommand{\lemref}[1]{Lemma~\ref{#1}}
  \newcommand{\propref}[1]{Proposition~\ref{#1}}

  \newcommand{\defref}[1]{Definition~\ref{#1}}
  
  \newcommand{\eqnref}[1]{Equation~\eqref{#1}}

  \DeclareMathOperator{\genus}{g}
  \DeclareMathOperator{\boundary}{b}
  \DeclareMathOperator{\cent}{ctr}

  \newcommand{\pka}{\partial_{\kappa}}

  \newcommand{\Map}{\ensuremath{\operatorname{Map}}\xspace}  
    
  \newcommand{\EL}{\ensuremath{\mathcal{EL}}\xspace} 
  \newcommand{\PML}{\ensuremath{\mathcal{PML}}\xspace}  
  \newcommand{\bfa}{{\textbf{a}}}  
  \newcommand{\bfb}{{\textbf{b}}}  
  \newcommand{\bfc}{{\textbf{c}}}

\DeclareMathOperator{\diam}{diam}

  \newcommand{\param}{{\mathchoice{\mkern1mu\mbox{\raise2.2pt\hbox{$
  \centerdot$}}
  \mkern1mu}{\mkern1mu\mbox{\raise2.2pt\hbox{$\centerdot$}}\mkern1mu}{
  \mkern1.5mu\centerdot\mkern1.5mu}{\mkern1.5mu\centerdot\mkern1.5mu}}}
\DeclarePairedDelimiterX{\norm}[1]{\lvert}{\rvert}{#1}
\DeclarePairedDelimiterX{\Norm}[1]{\lVert}{\rVert}{#1}

  \newcommand{\ST}{\mathbin{\Big|}} 
  \newcommand{\from}{\colon\thinspace}

\newcommand{\CAT}{\ensuremath{\operatorname{CAT}(0)}\xspace}         
    
\newcommand{\ps}{\partial_{s}}

\begin{document}

\title[Sublinearly Morse Boundary II]{Sublinearly Morse Boundary II: Proper geodesic spaces}
  
\author   {Yulan Qing}
\address{Department of Mathematics, Fudan University, China}
\email{yulan.qing@gmail.com}
\author   {Kasra Rafi}
\address{Department of Mathematics, University of Toronto, Toronto, ON, Canada}
\email{rafi@math.toronto.edu}
\author   {Giulio Tiozzo}
\address{Department of Mathematics, University of Toronto, Toronto, ON, Canada}
\email{tiozzo@math.toronto.edu}
 
\date{\today}
  
\maketitle

\begin{abstract}
We build an analogue of the Gromov boundary for any proper geodesic metric space, 
hence for any finitely generated group. More precisely, for any proper geodesic metric space $X$ and any sublinear function $\kappa$, we construct a boundary for $X$, denoted $\pka X$, that is quasi-isometrically invariant and metrizable. 
As an application, we show that when $G$ is the mapping class group of a finite type surface, or a relatively hyperbolic group, then with minimal assumptions the Poisson boundary of $G$ can be realized on the $\kappa$-Morse boundary of $G$ equipped
the word metric associated to any finite generating set.
\end{abstract}

\section{Introduction}
In this paper, we construct an analogue of the Gromov boundary for
a general proper geodesic metric space. That is, a notion of a boundary at infinity that is 
invariant under quasi-isometry, has good topological properties and is as large as possible. 
Our guiding principle is that, moving from the setting of Gromov hyperbolic 
spaces to general metric spaces, most key arguments still go through if we 
replace uniform bounds with sublinear bounds (with respect to distance
to some base point). Examples of this philosophy 
have appeared in the literature before, for example in \cite{Drutu, asympCAT(0), EFW12, EFW13, ACGH, EMR18}. 
In a prequel to this paper \cite{QRT19}, such a boundary was constructed in the setting of 
CAT(0) metric spaces.

\subsection*{Statement of Results}
Let $(X, d_X)$ be a proper, geodesic metric space with a base point $\go$. 
Recall that, when $X$ is Gromov hyperbolic, the Gromov boundary of $X$ is the set 
of equivalence classes of quasi-geodesic rays emanating from $\go$, equipped
with the cone topology. Two quasi-geodesic rays are considered equivalent 
if they stay within bounded distance from each other. In a Gromov hyperbolic
space, every quasi-geodesic ray $\beta$ is \emph{Morse}: 
that is, any other quasi-geodesic segment $\gamma$ 
with endpoints on $\beta$ stays in a bounded neighborhood of $\beta$. 

Similarly, we consider quasi-geodesic rays in $X$ (rays are
always assumed to be emanating from~$\go$). Roughly speaking, 
we say a quasi-geodesic ray $\beta$ is \emph{sublinearly Morse} if any other 
quasi-geodesic segment $\gamma$ 
with endpoints sublinearly close to $\beta$ stays in a sublinear neighborhood of $\beta$
(see \defref{D:k-morse} for the precise definition). We group the set of sublinearly Morse
quasi-geodesic rays into equivalence classes by setting quasi-geodesic rays
$\alpha$ and $\beta$ to be equivalent if they stay sublinearly close to each other. 
We call the set of equivalence classes of sublinearly Morse quasi-geodesic 
rays, equipped with a coarse version of the cone topology, the 
\emph{sublinearly Morse boundary} of $X$. 

In fact, the above construction works for any given \emph{sublinear function}
$\kappa \from [0,\infty) \to [1, \infty)$, where $\kappa$ is a concave, 
increasing function with
\[
\lim_{t \to \infty} \frac{\kappa(t)}t = 0. 
\] 
Then, we define the $\kappa$\emph{-Morse boundary} $\pka X$ to be the space of equivalence classes of $\kappa$-Morse
quasi-geodesic rays equipped with the coarse cone topology (see Definition~\ref{D:nbd}). 
We obtain a possibly large family of boundaries for $X$,
each associated to a different sublinear function $\kappa$. 

We show that $\pka X$ is metrizable and invariant under quasi-isometries; moreover, 
$\kappa$-boundaries 
associated to different sublinear functions are topological subspaces of each other.

\begin{introthm}\label{introthm-main}
Let $X$ be a proper, geodesic metric space, and let $\kappa$ be a sublinear function. 
Then we construct a topological space $\pka X$ with the following properties: 
\begin{enumerate}
\item \textup{(Metrizability)} 
The spaces $\pka X$ and $X \cup \pka X$ are metrizable, and $X \cup \pka X$ is a bordification of $X$; 
\item \textup{(QI-invariance)} Every $(k, K)$-quasi-isometry $\Phi \from X \to Y$ between proper geodesic metric spaces 
induces a homeomorphism $\Phi^\star \from \pka X \to \pka Y$;
\item \textup{(Compatibility)} For sublinear functions $\kappa$ and $\kappa'$ where
$\kappa \leq c \cdot \kappa'$ for some $c>0$, we have 
$\pka X \subset \partial_{\kappa'} X$ where the topology of $\pka X$ is the subspace 
topology. Further, letting $\partial X := \bigcup_\kappa \pka X$, 
we obtain a quasi-isometrically invariant topological space that contains all $\pka X$ 
as topological subspaces. We call $\partial X$ the \emph{sublinearly Morse boundary} 
of $X$. 
\end{enumerate}
\end{introthm} 

Note that from QI-invariance it follows that $\pka X$ and $\partial X$ do not depend on 
the base point $\go$. Moreover, it also implies that the $\kappa$-Morse boundary of 
a finitely generated group $G$ is independent of the generating set. Thus $\pka G$
and $\partial G$ are well defined. 

We now argue that, in different settings, the $\kappa$-Morse boundary $\pka X$ is large for an appropriate
choice of $\kappa$.
Recall that the \emph{Poisson boundary} is the maximal boundary from the measurable point of view (see Section \ref{S:Poisson}). 
In this paper, we let $G$ be either a mapping class group or a relatively hyperbolic group and $X$ be a Cayley graph of $G$ and show that $\pka X$ 
is a topological model for the Poisson boundary of $(G, \mu)$ associated to any non-elementary finitely supported measure $\mu$. 

In fact, we show the following general criterion: if almost every sample path of the random walk driven by $\mu$ sublinearly tracks a $\kappa$-Morse geodesic, then the $\kappa$-Morse boundary can be identified with the Poisson boundary. 
The following result was obtained in collaboration with Ilya Gekhtman.

\begin{introthm} \label{T:poiss-general}
Let $G$ be a finitely generated group, and let $(X, d_X)$ be a Cayley graph of $G$. 
Let $\mu$ be a probability measure on $G$ with finite first moment with respect to $d_X$, such that the semigroup generated 
by the support of $\mu$ is a non-amenable group. 
Let $\kappa$ be a sublinear function, and suppose that
for almost every sample path $\omega = (w_n)$, there exists a $\kappa$-Morse geodesic ray $\gamma_\omega$ such that 
\begin{equation} \label{E:sub-track}
\lim_{n \to \infty} \frac{d_X(w_n, \gamma_\omega)}{n} = 0.
\end{equation}
Then almost every sample path converges to a point in $\pka X$, and moreover the space $(\pka X, \nu)$, where $\nu$ is the hitting measure for the random walk, is a model for the Poisson boundary of $(G, \mu)$. 
\end{introthm}

For comparison, recall that the \emph{visual boundary} of CAT(0) spaces is not invariant by quasi-isometries \cite{CK00}.
The \emph{Gromov boundary} \cite{Gro87}, on the other hand, is QI-invariant, but only defined if the group is hyperbolic; a natural generalization 
is the \emph{Morse boundary} \cite{Morse}, which is always well-defined and QI-invariant, but very often it is too small: in particular, it has measure zero 
with respect to the hitting measure for most random walks on relatively hyperbolic groups \cite{CDG20}. This is related to the fact that a 
typical sample path is expected to have unbounded excursions in the peripherals. 
Finally, the \emph{Floyd boundary} \cite{Floyd} is well-behaved for relatively hyperbolic groups, but trivial for mapping class groups \cite{KN}.
 
\subsection*{Mapping class groups}
Let $S$ be a surface of finite hyperbolic type and $\Map(S)$ be the mapping class
group of $S$. Let $d_w$ be the word metric on $\Map(S)$ with respect
to some finite generating set. Then letting $(X, d_X) = (\Map(S), d_w)$ we
can consider the $\kappa$-Morse boundary $\pka \Map(S)$ of the 
mapping class group. We show the following characterization.

\newpage
\begin{introthm} \label{T:PB-intro}
Let $\mu$ be a finitely supported, non-elementary probability measure on $\Map(S)$, and let $p := 3\genus(S) - 3 + \boundary(S)$ be the \emph{complexity} of $S$. Then for 
$\kappa(t) = \log^p(t)$, we have:
\begin{enumerate}
\item Almost every sample path $(w_n)$ converges to a point in 
$\partial_\kappa \Map(S)$; 
\item The $\kappa$-Morse boundary $(\partial_\kappa \Map(S), \nu)$ 
is a model for the Poisson boundary of $(\Map(S), \mu)$ where $\nu$ is the 
hitting measure associated to the random walk driven by $\mu$. 
\end{enumerate}
\end{introthm}

The proof uses the machinery of curve complexes introduced by Masur-Minsky 
\cite{MM00}, as well the study of random walks in mapping class groups carried 
by Maher \cite{Maher}, \cite{MaherExp}, Sisto \cite{sisto-track}, Maher-Tiozzo \cite{MaherTiozzo} and Sisto-Taylor \cite{ST}.  
In particular, by \cite{ST}, a typical sample path makes logarithmic progress in each subsurface, which explains the function $\log(t)$, 
while $p$ is related to the ``depth" of the hierarchy paths in $\Map(S)$. 

Moreover, we also obtain the following tracking result between geodesics and sample paths 
in the mapping class group. 

\begin{introthm} \label{T:tracking-intro}
Let $\mu$ be a finitely supported, non-elementary probability measure on $\Map(S)$, and let $p$ as above.
Then, for almost every sample path there exists a $\kappa$-Morse geodesic ray $\gamma_\omega$ in $\Map(S)$ such 
that  
$$\limsup_{n \to \infty} \frac{d_w(w_n, \gamma_\omega)}{\log^{p+1}(n)} < +\infty.$$
\end{introthm}

This result improves the tracking result of Sisto \cite{sisto-track}, where the tracking function is $\sqrt{n \log n}$. 
Sublinear tracking for random walks with respect to the Teichm\"uller metric was obtained in \cite{Ti15}.

\subsection*{Relatively hyperbolic groups}
Now consider a finitely generated group $G$ equipped with a word
metric $d_w$ associated to a finite generating set. Recall that $G$
is relatively hyperbolic with respect to a family of subgroups $H_1, \dots, H_k$ if, 
after contracting the Cayley graph of $G$ along $H_i$-cosets, the resulting graph 
equipped with the usual graph metric is Gromov hyperbolic. Further, $H_i$-cosets
have to satisfy the technical condition of \emph{bounded coset penetration}. 
We say a relatively hyperbolic group $G$ is \emph{non-elementary} if it is infinite, not virtually cyclic, and each $H_i$ is infinite and with infinite index in $G$.
Similarly to above, we show:

\begin{introthm} \label{T:RH-intro}
Let $G$ be a non-elementary relatively hyperbolic group, and let $\mu$ be a
probability measure whose support is finite and generates $G$ as a semigroup. 
Then for $\kappa(t) = \log(t)$, we have:
\begin{enumerate}
\item Almost every sample path $(w_n)$ converges to a point in 
$\partial_\kappa G$; 
\item The $\kappa$-Morse boundary $(\partial_\kappa G, \nu)$ 
is a model for the Poisson boundary of $(G, \mu)$ where $\nu$ is the 
hitting measure associated to the random walk driven by $\mu$. 
\end{enumerate}
\end{introthm}

\begin{remark}
The proof of Theorem \ref{T:PB-intro} also works in the setting of 
hierarchically hyperbolic spaces \cite{HHG}. Hierarchically hyperbolic spaces 
are a family of axiomatically defined spaces with properties that are 
modelled after the mapping class groups. Hence, all the tools we use, such as 
subsurface projections, distance formulas and the Bounded Geodesic Image 
Theorem also exist in the setting of hierarchically hyperbolic spaces. 
\end{remark}

\subsection*{Sublinearly Morse vs. Sublinearly Contracting}

In the construction of the $\kappa$-Morse boundary $\pka X$, 
several different definitions are possible for the notion of a $\kappa$-Morse quasi-geodesic. 
The goal is always to emulate the behavior of quasi-geodesics in a Gromov
hyperbolic space but with sublinear errors (instead of uniform additive errors).  

The definition of $\kappa$-Morse given in this paper (Definition \ref{D:k-morse}) is equivalent to the definition of 
\emph{strongly Morse} in \cite{QRT19}. 
Another natural condition is to require a quasi-geodesic ray $\beta$ 
to be $\kappa$-\emph{weakly contracting} (Definition \ref{Def:generalContracting}): that is, that the projection of a ball disjoint from 
$\beta$ to $\beta$ has a diameter that is bounded by a sublinear function of 
the distance of the center of the ball to the origin.  

In the setting of CAT(0) spaces, these two notions are equivalent (\cite[Theorem 3.8]{QRT19}), 
but this is no longer true for general metric spaces.  
In the Appendix, we prove that $\kappa$-weakly contracting 
quasi-geodesics are always $\kappa$-Morse. The converse is known
not to be the case in general: for example, when $\kappa$ is the constant function, 
a geodesic axis of a pseudo-Anosov element in the mapping class group is always 
Morse \cite{MM00}, but not always strongly contracting \cite{RV}. However, 
$\kappa$-Morse quasi-geodesics are always $\kappa'$-weakly contracting for a larger 
sublinear function $\kappa'$:

\begin{introthm}
Let $X$ be a proper geodesic metric space, let $\kappa$ be a sublinear function, and let 
$\beta$ be a quasi-geodesic ray in $X$. Then:
\begin{enumerate}
\item If $\beta$ is $\kappa$-weakly contracting, it is $\kappa$-Morse;
\item There is a sublinear function 
$\kappa'$ such that if $\beta$ is $\kappa$-Morse then it is $\kappa'$-weakly contracting.
\end{enumerate}
\end{introthm}

The definition of $\kappa$-Morse we use in this paper  is the one that matches our philosophical approach the best, 
is the most flexible and makes the arguments simplest.  Hence, we 
think it is the \emph{correct} definition with minimal assumptions. 

In some places, the same arguments as in \cite{QRT19} apply directly, while in others new ideas are needed; 
for the sake of brevity, we shall skip the proofs when the arguments are exactly the same.

\subsection*{History}
The notion of a \emph{Morse geodesic} is classical \cite{Mor24}, and much progress has been made in recent years using Morse geodesics to define boundaries of groups.
In \cite{Morse}, Cordes, inspired by the  \emph{sublinear boundary} for CAT(0) spaces of Charney and Sultan \cite{CS15}, constructed the 
\emph{Morse boundary} for all proper geodesic spaces, where a quasi-geodesic  $\gamma$ is \emph{Morse} if there are uniform neighborhoods, with size depending on $(q, Q)$, in which all $(q, Q)$-quasi-geodesic segments with endpoints on $\gamma$ lie.  The Morse boundaries are equipped with a \emph{direct limit topology} and  are invariant under 
quasi-isometries. However, this space does not have good topological properties; for example, 
it is not first countable. Cashen-Mackay \cite{cashenmackay}, 
following the work of Arzhantseva-Cashen-Gruber-Hume \cite{ACGH}, 
defined a different topology on the Morse boundary.
They showed that it is Hausdorff and when there is a geometric action by a countable group, it is also metrizable; note that \thmref{introthm-main} does not assume any geometric action. 
Another notion, more closely related to our definition of $\kappa$-Morse boundary, is defined by Kar \cite{asympCAT(0)}: a geodesic space is \emph{asymptotically \CAT} if balls of radius $r$ are coarsely \CAT with an error of $f(r)$, where the function $f(r)$ is sublinear. That is to say, a space is asymptotically \CAT if it has the same $\kappa$-Morse boundary as a \CAT space, where $\kappa= f(r)$.

The definition of a \emph{contracting} geodesic originates from \cite{Mor24} and has been brought back to attention by \cite{Gro87}. 
In \cite{MM00} Masur-Minsky proved that axes of pseudo-Anosov elements are contracting, and since then various versions of this condition have been discussed, in particular the notions of \emph{strongly contracting} and \emph{weakly contracting}, see e.g. \cite{Be06}, \cite{BF09}, \cite{AK11}, \cite{ACT}, \cite{Si18}, \cite{Ya20}, and \cite{RV}. Our definition of $\kappa$-weakly contracting is even weaker, and for that reason it is expected to be generic with respect to many notions of genericity. For random walks, genericity of sublinearly contracting geodesics in relatively hyperbolic
groups follows from \cite{sisto-track}, in hierarchically hyperbolic groups it follows 
from \cite{ST} and in CAT(0) groups it follows from \cite{GQR}. For the counting 
measure, genericity of 
log-weakly contracting geodesics in RAAGs has been shown in \cite{QT19}.

The Poisson boundary of $(G, \mu)$ is trivial for all non-degenerate measures $\mu$ on abelian groups \cite{CD60} and nilpotent groups \cite{DM60}. Discrete subgroups of $SL(d, \mathbb{R})$ are treated in \cite{Furstenberg}, where the Poisson boundary is related to the space of flags. For random walks on Lie groups, the study and the description of the Poisson boundary  was extensively developed in the 70's and the 80's by many authors, most notably Furstenberg. The Poisson boundary of some Fuchsian groups has also been described by Series \cite{Series} as being the limit set of the group. Kaimanovich \cite{Kai1} identified the Poisson boundary of hyperbolic groups with their Gromov boundaries with the associated hitting measures. Karlsson and Margulis proved that visual boundaries of nonpositively curved spaces serve as models for their Poisson boundaries \cite{KarlMarg}. Kaimanovich and Masur \cite{KM2, KM} proved that the Poisson boundary of the mapping class group is the boundary of Thurston's compactification of Teichm\"uller space. Their description also runs for the Poisson boundary of the braid group; see \cite{braids}.

\subsection*{Acknowledgement}
Y.Q.  thanks the University of Toronto for its hospitality in Fall 2020.
K.R. is partially supported by NSERC. 
G.T. is partially supported by NSERC and the Alfred P. Sloan foundation. 
This material is based upon work supported by the National Science Foundation under Grant No. DMS-1928930 while G.T. participated in a program hosted by the Mathematical Sciences Research Institute in Berkeley, California, during the Fall 2020 semester.  
We thank Ilya Gekhtman for several useful conversations.

\section{Preliminaries}

Let $(X, d_X)$ be a proper, geodesic metric space, and let $\go \in X$ be a base point. 
Given a point $p \in X$, we denote $\Vert p \Vert := d_X(\go, p)$.
Let $\alpha$ be a quasi-geodesic ray starting at $\go$. For $r>0$, let $t_r$ be the first time where $\Norm{\alpha(t_r)}=r$ and denote the point  $\alpha(t_r)$ with
\[
\alpha_r : = \alpha(t_r),\]
while the segment from $\alpha(0)$ to $\alpha(t_r)$ will be denoted as
\[\alpha|_{r} := \alpha([0,t_r]).\]

We collect here the two basic geometric properties of the space that we need:
\begin{lemma}
Let $(X, d_X)$ be a proper, geodesic metric space. Then:
\begin{itemize}
\item For any closed set $Z \subset X$ and any point $x \in X$, there is a closed
set $\pi_Z(x)$ of nearest points in $Z$ to $x$. We refer to any point in $\pi_Z(x)$ 
as a nearest point projection from $x$ to $Z$.
\item Any sequence of geodesics $\beta_{n} \from [0, n] \to X$ with $\beta_{0} =\go$ has a subsequence that converges uniformly on compact sets to a geodesic ray $\beta_{0} \from [0, \infty) \to X$. 
\end{itemize} 
\end{lemma}

The following lemma about nearest point projections will be used several times. 

\begin{lemma} \label{L:twice}
Let $\alpha, \beta$ be two quasi-geodesic rays starting at the base point $\go \in X$. Let $x$ be a point on $\alpha$, 
and let $p$ be a nearest point projection of $x$ onto $\beta$. Then we have 
$$\Vert p \Vert \leq 2 \Vert x \Vert.$$
\end{lemma}

\begin{proof}
Note that, since $p$ is a nearest point projection of $x$ onto $\beta$ and $\go$ belongs to $\beta$, we have 
$$d_X(x, p) \leq d_X(\go, x).$$
Hence, by the triangle inequality, 
\begin{equation*}\Vert p \Vert = d_X(\go, p) \leq d_X(\go, x) + d_X(x, p) \leq 2 d_X(\go, x) = 2 \Vert x \Vert.
\qedhere
\end{equation*}
\end{proof}

\subsection{Sublinear functions}

In this paper, a \emph{sublinear function} will be a function $\kappa \from [0, \infty) \to [1, \infty)$ such that 
\[
\lim_{t \to \infty} \frac{\kappa(t)}{t}  = 0.
\]
Moreover, we say that $\kappa \from [0, \infty) \to [1, \infty)$ is a \emph{concave sublinear function} if it is sublinear and moreover it is increasing and concave. 

\begin{remark}
The assumptions that $\kappa$ is increasing and concave make certain arguments
cleaner, otherwise they are not really needed. One can always replace any 
sublinear function $\kappa$ with another sublinear function $\overline \kappa$
so that $\kappa(t) \leq \overline \kappa(t) \leq C \, \kappa(t)$ for some constant $C$ 
and $\overline \kappa$ is monotone increasing and concave. For example, define 
\[
\overline \kappa(t) := \sup \Big\{ \lambda \kappa(u) + (1-\lambda)\cdot \kappa(v) \ST 
\ 0 \leq \lambda \leq 1, \ u,v>0, \ \text{and}\ \lambda u + (1-\lambda)v =t \Big\}.
\]
\end{remark}

\begin{lemma}
If $\kappa \from [0, \infty) \to [1, \infty)$ is a concave sublinear function and $\lambda > 1$, then 
$$\kappa(\lambda t) \leq \lambda \kappa(t)$$
for any $t \geq 0$.
\end{lemma}

\begin{proof}
By concavity, 
$$\kappa(t) = \kappa \left( \frac{1}{\lambda} \lambda t + (1 - \lambda^{-1}) \cdot 0 \right) \geq \frac{1}{\lambda} \kappa(\lambda t) + (1 - \lambda^{-1})\cdot \kappa(0) \geq \frac{1}{\lambda} \kappa(\lambda t) $$
from which the claim follows.
\end{proof}

\section{The $\kappa$-Morse boundary}

We now introduce the definition of $\kappa$-Morse quasi-geodesic, which will be fundamental for our construction. 
To set the notation, we say a quantity $D$ \emph{is small compared to} a radius $r >0$ if 
\begin{equation} \label{Eq:Small} 
D \leq \frac{r}{2\kappa(r)}. 
\end{equation} 

We will fix once and for all a base point $\go \in X$, and all quasi-geodesic rays we consider will be based at $\go$. 
Given a quasi-geodesic ray $\alpha$ and a constant $m$, we define 
$$\mathcal{N}_{\kappa}(\alpha, m) := \Big\{ x \in X \ : \ d_X(x, \alpha) \leq m \cdot \kappa(\Vert x \Vert) \Big\}.$$

The following observation will be useful. 

\begin{lemma}\label{L:symmetry}
Let $\beta$ be a quasi-geodesic ray and $\alpha$ be a geodesic ray, both based at $\go \in X$. Suppose 
that 
$$\beta \subseteq \mathcal{N}_{\kappa}(\alpha, m)$$
for some function $\kappa$ and some constant $m$. Then we also have
$$\alpha \subseteq \mathcal{N}_{\kappa}(\beta , 2 m).$$
\end{lemma}

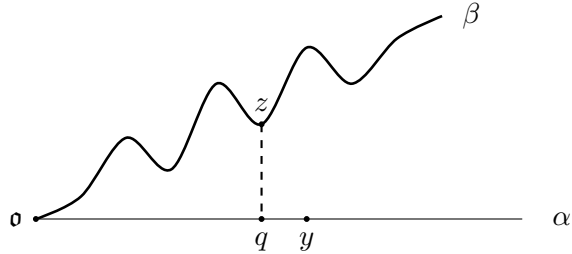
\begin{figure}[h!]
\begin{tikzpicture}[scale=0.6]
 \tikzstyle{vertex} =[circle,draw,fill=black,thick, inner sep=0pt,minimum size=.5 mm]
[thick, 
    scale=1,
    vertex/.style={circle,draw,fill=black,thick,
                   inner sep=0pt,minimum size= .5 mm},
                  
      trans/.style={thick,->, shorten >=6pt,shorten <=6pt,>=stealth},
   ]

  \node[vertex] (o) at (0,0)[label=left:$\go$] {}; 
  \node(a) at (11,0)[label=right:$\alpha$] {}; 
  
  \draw (o)--(a){};
  \node at (9,4.5) [label=right:$\beta$] {};
   
  \draw[thick, dashed] (5, 2.1) to(5, 0) {}; 
  \node[vertex] at(6, 0)[label=below:$y$] {}; 
  \node[vertex] at(5, 0)[label=below:$q$] {}; 
  \node[vertex] at(5, 2.1)[label=above:$z$] {}; 

  \pgfsetlinewidth{1pt}
  \pgfsetplottension{.55}
  \pgfplothandlercurveto
  \pgfplotstreamstart
  \pgfplotstreampoint{\pgfpoint{0cm}{0cm}}
  \pgfplotstreampoint{\pgfpoint{1cm}{0.5cm}}
  \pgfplotstreampoint{\pgfpoint{2cm}{1.8cm}}
  \pgfplotstreampoint{\pgfpoint{3cm}{1.1cm}}
  \pgfplotstreampoint{\pgfpoint{4cm}{3cm}}
  \pgfplotstreampoint{\pgfpoint{5cm}{2.1cm}}
  \pgfplotstreampoint{\pgfpoint{6cm}{3.8cm}}
  \pgfplotstreampoint{\pgfpoint{7cm}{3cm}}
  \pgfplotstreampoint{\pgfpoint{8cm}{4cm}}
  \pgfplotstreampoint{\pgfpoint{9cm}{4.5cm}}
  \pgfplotstreamend
  \pgfusepath{stroke} 
\end{tikzpicture}
\caption{$\Norm y = \Norm z$ and $q \in \pi_{\alpha} (z)$ as in the proof of Lemma~\ref{L:symmetry}. }
\end{figure}

\begin{proof}

Let $y \in \alpha$ be a point and let $r : = \Norm {y}$. Let $z \in \beta$ be a point such that $\Vert z \Vert = r$
and let $q$ be a nearest point projection of $z$ to $\alpha$. 
By assumption, 
\[
d_X(z, q) \leq m \cdot \kappa(r).
\] 
On the other hand,
\begin{align*}
d_X(y, q) & = | \Vert y \Vert- \Norm q | & \text{since $\alpha$ is geodesic} \\
 	   & = | \Vert z  \Vert - \Norm q | & \\
           & \leq d_X(z, q) & \text{by triangle inequality}.
\end{align*}
Therefore we have
\begin{align*}
d_X(y, \beta) \leq d_X(y, z) &\leq d_X(y, q) + d_X(q, z) \\
                                              & \leq 2 d_X(z, q) \\ 
                                              & \leq 2 m \cdot \kappa(r)
\end{align*}
which completes the proof.
\end{proof}

\begin{definition} \label{D:k-morse}
Let $Z \subseteq X$ be a closed set, and let $\kappa$ be a concave sublinear function. 
We say that $Z$ is \emph{$\kappa$-Morse} if there exists a proper function 
$m_Z : \mathbb{R}^2 \to \mathbb{R}$ such that for any sublinear function $\kappa'$ 
and for any $r > 0$, there exists $R$ such that for any $(q, Q)$-quasi-geodesic ray $\beta$
with $m_Z(q, Q)$ small compared to $r$, if 
$$d_X(\beta_R, Z) \leq \kappa'(R)
\qquad\text{then}\qquad
\beta|_r \subset \calN\big(Z, m_Z(q, Q)\big)$$
The function $m_Z$ will be called a $\emph{Morse gauge}$ of $Z$. 
\end{definition}

Note that we can always assume without loss of generality that $\max \{ q, Q \} \leq m_Z(q, Q)$, and we will assume this in the following.

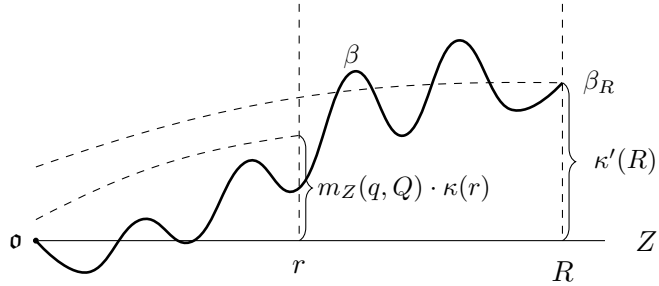
\begin{figure}[h!]
\begin{tikzpicture}[scale=0.7]
 \tikzstyle{vertex} =[circle,draw,fill=black,thick, inner sep=0pt,minimum size=.5 mm]
[thick, 
    scale=1,
    vertex/.style={circle,draw,fill=black,thick,
                   inner sep=0pt,minimum size= .5 mm},
                  
      trans/.style={thick,->, shorten >=6pt,shorten <=6pt,>=stealth},
   ]

  \node[vertex] (o) at (0,0)[label=left:$\go$] {}; 
  \node(a) at (11,0)[label=right:$Z$] {}; 
  
  \draw (o)--(a){};
  \draw [dashed] (0, 1.4) to [bend left = 10] (10,3){};
  \draw [dashed] (0, 0.4) to [bend left = 10] (5,2){};
  
   \draw [decorate,decoration={brace,amplitude=5pt},xshift=0pt,yshift=0pt]
  (5, 2) -- (5,0)  node [thick, black,midway,xshift=0pt,yshift=0pt] {};   
  \node at (7,1){\small $m_{Z}(q, Q) \cdot \kappa(r)$};
     
 \draw [dashed] (10, 4.5) to (10,0) {};
 \node at (10,0)[label=below:$R$] {};
 \draw [dashed] (5, 4.5) to (5,0){};
 \node at (5,0)[label=below:$r$] {};
        
  \draw [decorate,decoration={brace,amplitude=5pt},xshift=0pt,yshift=0pt]
  (10, 3) -- (10,0)  node [thick, black,midway,xshift=0pt,yshift=0pt] {};       
  \node at (11.2, 1.5) { \small $ \kappa'(R)$};
   \node at (10.7, 3) {\small  $ \beta_{R}$};
    \node at (6, 3.5) {\small  $\beta$};
        
  \pgfsetlinewidth{1pt}
  \pgfsetplottension{.75}
  \pgfplothandlercurveto
  \pgfplotstreamstart
  \pgfplotstreampoint{\pgfpoint{0cm}{0cm}}  
  \pgfplotstreampoint{\pgfpoint{1cm}{-.6cm}}   
  \pgfplotstreampoint{\pgfpoint{2cm}{0.4cm}}
  \pgfplotstreampoint{\pgfpoint{3cm}{0cm}}
  \pgfplotstreampoint{\pgfpoint{4cm}{1.5cm}}
  \pgfplotstreampoint{\pgfpoint{5cm}{1cm}}
  \pgfplotstreampoint{\pgfpoint{6cm}{3.2cm}}
  \pgfplotstreampoint{\pgfpoint{7cm}{2cm}}
  \pgfplotstreampoint{\pgfpoint{8cm}{3.8cm}}
  \pgfplotstreampoint{\pgfpoint{9cm}{2.5cm}}
  \pgfplotstreampoint{\pgfpoint{10cm}{3cm}}
  \pgfplotstreamend
  \pgfusepath{stroke} 
\end{tikzpicture}
\caption{Definition of $\kappa$-Morse set $Z$: Every quasi-geodesic ray $\beta$ has the property that there exists $R(Z, r, q, Q, \kappa')$, such that if $\beta_{R}$ is distance $\kappa'(R)$ from $Z$, then $\beta|_{r}$ is in the neighborhood $\calN_{\kappa}(Z, m_{Z}(q, Q))$. }
\label{Fig:Strong1} 
\end{figure}

\subsection{Equivalence classes}

\begin{definition}
Given two quasi-geodesic rays $\alpha$, $\beta$ based at $\go$, we say that $\beta \sim \alpha$ if they sublinearly track each other: 
i.e. if 
$$\lim_{r \to \infty} \frac{d_X(\alpha_r, \beta_r)}{r} = 0.$$
\end{definition}

By the triangle inequality, $\sim$ is an equivalence relation on the space of quasi-geodesic rays based at $\go$, hence also 
on the space of $\kappa$-Morse quasi-geodesic rays.

\begin{lemma} \label{L:equivalence}
Let $\alpha$ be a $\kappa$-Morse quasi-geodesic rays with Morse gauge $m_\alpha$, and let $\beta \sim \alpha$ be a $(q, Q)$-quasi-geodesic ray. 
Then:
\begin{enumerate}[(i)]
\item $\beta$ is $\kappa$-Morse, and moreover
$$\beta \subseteq \mathcal{N}_{\kappa}(\alpha, m_\alpha(q, Q));$$
\item if in addition $\alpha$ is a geodesic ray, then 
\[
\alpha \subseteq \mathcal{N}_{\kappa}(\beta, 2m_\alpha(q, Q)).
\]
\end{enumerate}
\end{lemma}

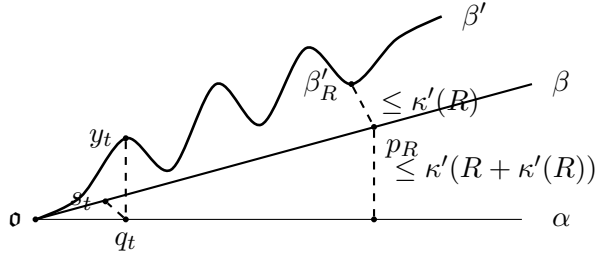
\begin{figure}[h!]
\begin{tikzpicture}[scale=0.6]
 \tikzstyle{vertex} =[circle,draw,fill=black,thick, inner sep=0pt,minimum size=.5 mm]
[thick, 
    scale=1,
    vertex/.style={circle,draw,fill=black,thick,
                   inner sep=0pt,minimum size= .5 mm},
                  
      trans/.style={thick,->, shorten >=6pt,shorten <=6pt,>=stealth},
   ]

  \node[vertex] (o) at (0,0)[label=left:$\go$] {}; 
  \node(a) at (11,0)[label=right:$\alpha$] {}; 
  
  \draw (o)--(a){};
  \draw [-, thick] (o) to (11, 3){};
  \node at (11,3) [label=right:$\beta$] {}; 
  \node at (9,4.5) [label=right:$\beta'$] {};
   
  \node at (7.2, 2.6) [label=right:$\leq \kappa'(R)$] {}; 
  \node at (7.5, 1.1) [label=right:$\leq \kappa'(R + \kappa'(R))$] {}; 
   
  \node[vertex] (a1) at (2, 1.8)[label=left:$y_t$] {}; 
  \node[vertex] (a2) at(2, 0)[label=below:$q_{t}$] {}; 
  \node[vertex] (a3) at(1.55, 0.4)[label=left:$s_{t}$] {}; 
  \node[vertex] (b1) at(7, 3)[label=left:$\beta'_R$] {}; 
  \node[vertex](b2) at (7.5, 2.05)[label=below right:$p_{R}$] {}; 
  \node[vertex] (b3) at(7.5, 0)[label=left:$$] {}; 
   
  \draw[thick, dashed] (a1)to(a2) {}; 
  \draw[thick, dashed] (a2)to(a3) {}; 
   
  \draw[thick, dashed] (b1) to (b2) {}; 
  \draw[thick, dashed] (b2) to (b3) {}; 

  \pgfsetlinewidth{1pt}
  \pgfsetplottension{.55}
  \pgfplothandlercurveto
  \pgfplotstreamstart
  \pgfplotstreampoint{\pgfpoint{0cm}{0cm}}
  \pgfplotstreampoint{\pgfpoint{1cm}{0.5cm}}
  \pgfplotstreampoint{\pgfpoint{2cm}{1.8cm}}
  \pgfplotstreampoint{\pgfpoint{3cm}{1.1cm}}
  \pgfplotstreampoint{\pgfpoint{4cm}{3cm}}
  \pgfplotstreampoint{\pgfpoint{5cm}{2.1cm}}
  \pgfplotstreampoint{\pgfpoint{6cm}{3.8cm}}
  \pgfplotstreampoint{\pgfpoint{7cm}{3cm}}
  \pgfplotstreampoint{\pgfpoint{8cm}{4cm}}
  \pgfplotstreampoint{\pgfpoint{9cm}{4.5cm}}
  \pgfplotstreamend
  \pgfusepath{stroke} 
  \end{tikzpicture}
  \caption{The setup in the proof of Lemma \ref{L:equivalence} (i).}
\end{figure}

\begin{proof}
(i) Define $\kappa'(r) := d_X(\alpha_r, \beta_r)$. By definition of $\sim$, the function $\kappa'$ is sublinear, 
and moreover, 
for any $R > 0$,
$$d_X(\beta_R, \alpha) \leq \kappa'(R).$$
Hence, since $\alpha$ is $\kappa$-Morse, for any $r$ we have 
$$\beta |_r \leq m_\alpha(q, Q)\cdot \kappa(r)$$
thus
\begin{equation} \label{E:Morse1}
\beta \subseteq \mathcal{N}_{\kappa}(\alpha, m_\alpha(q, Q)),
\end{equation}
which proves the second part of (i).
Let us now prove that $\beta$ is $\kappa$-Morse. Let $\kappa'$ be a sublinear function, let $r > 0$ 
and let $\beta'$ be a $(q', Q')$-quasi-geodesic ray such that 
$$d_X(\beta'_R, \beta) \leq \kappa'(R)$$
for some sufficiently large $R$. 
Let $p_R$ be a nearest point projection 
of $\beta'_R$ to $\beta$; by Lemma \ref{L:twice} we have 
$$\Vert p_R \Vert \leq 2 \Vert \beta_R' \Vert = 2 R.$$
Then, by the triangle inequality and equation \eqref{E:Morse1}, 
\begin{align*}
d_X(\beta'_R, \alpha) & \leq d_X(\beta'_R, p_R) + d_X(p_R, \alpha) \\
& \leq  \kappa'(R) + m_\alpha(q, Q)\cdot \kappa (\Vert p_R \Vert) \\
& \leq \kappa'(R) + m_\alpha(q, Q)\cdot \kappa(2 R).
\end{align*}
Since $\kappa''(R) := \kappa'(R) + m(q, Q)\cdot \kappa(2 R)$ is also a sublinear function, and since $\alpha$ is $\kappa$-Morse, 
this implies that 
$$\beta|_{r} \subseteq \mathcal{N}_{\kappa}(\alpha, m_\alpha(q', Q')). $$
Let $y_t$ be any point on $\beta'$ with $\Vert y_t \Vert = t \leq r$. 
By Lemma \ref{L:twice}, if $q_t$ is a nearest point projection of $y_t$ to $\alpha$, we have 
 $$\Vert q_t \Vert \leq 2 \Vert y_t \Vert = 2 t. $$
Now, if $q$ is a point on $\alpha$ and $s$ is a nearest point projection of $q$ to $\beta$,  by the triangle inequality and the Morse property,
$$\Vert q \Vert \geq \Vert s \Vert - d_X(s, q) \geq  \Vert s \Vert - m_\alpha(q, Q)\cdot \kappa(\Vert s \Vert).$$
Moreover, by Lemma \ref{L:twice}, $\Vert s \Vert \leq 2 \Vert q \Vert$,
hence by concavity
$$d_X(s, q) \leq m_\alpha(q, Q)\cdot \kappa(\Vert s \Vert) \leq  2 m_\alpha(q, Q)\cdot \kappa (\Vert q \Vert).$$
Now, if $s_t$ is a nearest point projection of $q_t$ to $\beta$, the above estimate yields
\begin{align*}
d_X(q_t, s_t) & \leq 2 m_\alpha(q, Q)\cdot \kappa (\Vert q_t \Vert) \\
& \leq 4 m_\alpha(q, Q)\cdot \kappa (t)
\end{align*}
hence, putting everything together, 
\begin{align*}
d_X(y_t, \beta) & \leq d_X(y_t, q_t) + d_X(q_t, s_t) \\
& \leq m_\alpha(q', Q')\cdot \kappa(t) +  4 m_\alpha(q, Q)\cdot \kappa (t)
\end{align*}
which, by setting $m_\beta(q', Q') := m_\alpha(q', Q')+ 4 m_\alpha(q, Q)$, 
proves the claim.

(ii) It follows immediately from (i) and Lemma \ref{L:symmetry}. 
\end{proof}

\begin{corollary} \label{Cor:m_beta}
If $\beta$ is a $(q, Q)$-quasi-geodesic ray, and $\beta_{0}$ is a $\kappa$-Morse geodesic ray such that $\beta \sim \beta_{0}$, then the function
\[
m_\beta(\param, \param) := m_{\beta_{0}}(\param, \param) + 4 m_{\beta_{0}}(q, Q)
\]
is a Morse gauge for $\beta$. In particular, the Morse gauge depends only on
$m_{\beta_{0}}$, $q$ and $Q$ and not on the particular quasi-geodesic $\beta$. 
\end{corollary}

\begin{figure}[h!]
\begin{tikzpicture}[scale=0.5]
 \tikzstyle{vertex} =[circle,draw,fill=black,thick, inner sep=0pt,minimum size=.5 mm]
[thick, 
    scale=1,
    vertex/.style={circle,draw,fill=black,thick,
                   inner sep=0pt,minimum size= .5 mm},
                  
      trans/.style={thick,->, shorten >=6pt,shorten <=6pt,>=stealth},
   ]

  \node[vertex] (o) at (0,0)[label=left:$\go$] {}; 
  \node(a) at (11,0)[label=right:$\beta_{0}$] {}; 
  
  \draw (o)--(a){};
  \node at (9,4.5) [label=right:$\beta'$] {};
  \node at (9,-3) [label=right:$\beta$] {};
   
  \draw[thick, dashed] (5, 2.1) to(5, 0) {}; 
  \draw[thick] (5, 0.2) to(5.2, 0.2) {}; 
  \draw[thick] (5.2, 0.2) to(5.2, 0) {};
       
  \draw[thick, dashed] (5, 0) to(5.3,-1) {}; 
  \draw[thick] (5.25, -0.75) to(5.4, -0.7) {}; 
  \draw[thick] (5.4, -0.7) to(5.5, -1) {};
    
  \node[vertex] at(5, 0)[label=below left:$p$] {}; 
  \node[vertex] at(5.3, -1)[label=below:$q$] {}; 
  \node[vertex] at(5, 2.1)[label=above:$z'_{r}$] {}; 

  \pgfsetlinewidth{1pt}
  \pgfsetplottension{.55}
  \pgfplothandlercurveto
  \pgfplotstreamstart
  \pgfplotstreampoint{\pgfpoint{0cm}{0cm}}
  \pgfplotstreampoint{\pgfpoint{1cm}{-0.5cm}}
  \pgfplotstreampoint{\pgfpoint{2cm}{-1.8cm}}
  \pgfplotstreampoint{\pgfpoint{3cm}{-1.1cm}}
  \pgfplotstreampoint{\pgfpoint{4cm}{-2cm}}
  \pgfplotstreampoint{\pgfpoint{5.5cm}{-1cm}}
  \pgfplotstreampoint{\pgfpoint{7cm}{-3cm}}
  \pgfplotstreampoint{\pgfpoint{8cm}{-2cm}}
  \pgfplotstreampoint{\pgfpoint{9cm}{-3cm}}
  \pgfplotstreamend
  \pgfusepath{stroke} 
    
  \pgfsetlinewidth{1pt}
  \pgfsetplottension{.55}
  \pgfplothandlercurveto
  \pgfplotstreamstart
  \pgfplotstreampoint{\pgfpoint{0cm}{0cm}}
  \pgfplotstreampoint{\pgfpoint{1cm}{0.5cm}}
  \pgfplotstreampoint{\pgfpoint{2cm}{1.8cm}}
  \pgfplotstreampoint{\pgfpoint{3cm}{1.1cm}}
  \pgfplotstreampoint{\pgfpoint{4cm}{3cm}}
  \pgfplotstreampoint{\pgfpoint{5cm}{2.1cm}}
  \pgfplotstreampoint{\pgfpoint{6cm}{3.8cm}}
  \pgfplotstreampoint{\pgfpoint{7cm}{3cm}}
  \pgfplotstreampoint{\pgfpoint{8cm}{4cm}}
  \pgfplotstreampoint{\pgfpoint{9cm}{4.5cm}}
  \pgfplotstreamend
  \pgfusepath{stroke} 
  \end{tikzpicture}
\caption{Corollary \ref{Cor:m_beta}: $\Norm {z'_{r}} = r$ and $p = \pi_{\beta_{0}}(z'_{r})$ and $q = \pi_{\beta}(p)$. }
\end{figure}
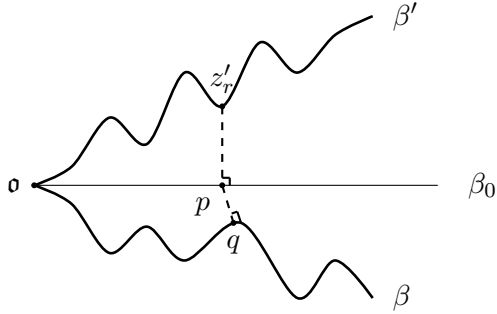

\begin{proof}
Let $\beta' \sim \beta$ be a $(q', Q')$-quasi-geodesic ray. Let $z'_r$ be 
a point along $\beta'$ with norm $r$, let $p \in \pi_{\beta_{0}}(z'_r)$ and let $q$ be a 
nearest point in $\beta$ to $p$. Note that, by Lemma \ref{L:twice},
$\Norm p \leq 2r$. Hence, 
\begin{align*} 
d_X(z'_r, \beta) &\leq d_X(z'_r, p) + d_X(p, q) \\
\intertext{and, by Lemma \ref{L:equivalence} (i) and (ii),}
  &\leq m_{\beta_{0}}(q', Q') \cdot \kappa(r) + 2 m_{\beta_{0}}(q, Q) \cdot \kappa(\Vert p \Vert) \\
  & \leq m_{\beta_{0}}(q', Q') \cdot \kappa(r) + 4 m_{\beta_{0}}(q, Q) \cdot \kappa(r)\\
  &\leq \big( m_{\beta_{0}}(q', Q')+ 4 m_{\beta_{0}}(q, Q)\big) \cdot \kappa(r).
\end{align*} 
This finishes the proof. 
\end{proof} 

We also need the following surgery statement from \cite{QRT19}.
The statement is slightly altered but the proof is identical. 

\begin{lemma}[Surgery Lemma,  Lemma 4.3 \cite{QRT19}]
\label{Lem:surgery}
For every $q, Q, r>0$ there exists $R>0$ such that the following holds. 
Let $\gamma$ be a geodesic ray of length at least $R$ and $\alpha$ be 
a $(q, Q)$-quasi-geodesic ray.  Assume that $d_X(\gamma_r, \alpha)\leq r/2$. 
Then, there exists a $(9 q, Q)$-quasi-geodesic ray $\gamma'$ so that 
\[
\alpha|_{r/2} = \gamma'|_{r/2}
\]
and the portion of $\gamma'$ outside of $B(\go, R)$ is the same as $\gamma$. 
\end{lemma}

\section{The topology on the $\kappa$-Morse boundary}

We denote as $\pka X$ the set of all equivalence classes of $\kappa$-Morse quasi-geodesic rays. In this section we will define
a topology on $\pka X$ to make it into a topological space. Even more, we will construct a bordification on $X \cup \pka X$ and we will show that both $X$ and 
$X \cup \pka X$ have good topological properties. 

\subsection{Bordification}
Let $\beta$ be a $\kappa$-Morse quasi-geodesic ray based at $\go$ and let $m_\beta$ be a Morse gauge for $\beta$. 

\begin{definition} \label{D:nbd}
We define the set $\calU(\beta, r) \subseteq X \cup \partial_\kappa X$ as follows.

\begin{itemize}
\item  
An equivalence class $\textbf{a} \in \partial_\kappa X$ belongs to $\calU(\beta, r)$ if,
for any $(q, Q)$-quasi-geodesic ray $\alpha \in \textbf{a}$, where  $m_\beta(q, Q)$ is 
small compared to $r$ (in the sense of Equation~\ref{Eq:Small}), we have the inclusion 
\[ 
\alpha\vert_r \subseteq \mathcal{N}_{\kappa}\big(\beta, m_\beta(q, Q)\big).
\]
\item 
A point $p \in X$ belongs to $\calU(\beta, r)$ if $d_X(\go, p) \geq r$ and, 
for every $(q, Q)$-quasi-geodesic $\alpha$ between $\go$ and $p$ 
where  $m_\beta(q, Q)$ is small compared to $r$, we have
\[
\alpha|_{r} \subseteq \mathcal{N}_{\kappa}\big(\beta, m_{\beta}(q, Q)\big).
\]
\end{itemize}
We denote $\calU(\beta, r) \cap \pka X$ by $\partial \calU(\beta, r)$. 
\end{definition}

We now verify the following basic properties of $\calU(\beta, r)$. 
\begin{lemma} \label{L:Ubeta}
Let $\beta$ be a quasi-geodesic ray which belongs to the class $\textbf{b}$. 
Then: 
\begin{enumerate}
\item There exists a (not necessarily unique) geodesic ray in the class $\bfb$;
\item The class $\textbf{b}$ belongs to $\calU(\beta, r)$ for any $r > 0$; 
\item $\bigcap_{r > 0} \calU(\beta, r) = \{ \textbf{b} \}$;
\item If $\beta_1 \sim \beta_2$ are two $\kappa$-Morse quasi-geodesic rays, 
for any $r_ 1 > 0$, there exists $r_2 > 0$ such that 
$$\calU(\beta_1, r_1) \supseteq \calU(\beta_2, r_2).$$
\end{enumerate}
\end{lemma}

\begin{proof}
(1) Consider the geodesic segment $\beta_n$ connecting $\go$ and 
$\beta(n)$ for $n =1, 2, 3...$. The sequence of geodesic segments has a convergent subsequence converging to a geodesic ray $\beta'$ by Arzel\'a-Ascoli. Since $\beta$ is 
$\kappa$-Morse, $\beta_n \subset \calN_{\kappa}\big(\beta, m_{\beta}(1, 0)\big)$.
Therefore the same is true for $\beta'$ and hence $\beta'$ belongs to the class $\bfb$.

(2) If $\beta'$ is a $(q, Q)$-quasi-geodesic ray which belongs to $\textbf{b}$, then by Lemma \ref{L:equivalence} 
$$\beta' \subseteq \mathcal{N}_{\kappa}(\beta, m_\beta(q, Q))$$
hence also 
$$\beta'\vert_r \subseteq \mathcal{N}_{\kappa}(\beta, m_\beta(q, Q))$$
for any $r > 0$, as needed.

(3) Let $\textbf{c} \in \partial_\kappa X$ be a class which belongs to $\calU(\beta, r)$ for any $r$, let $(q', Q')$ be the quasi-geodesic 
constants of $\beta$, and let $\gamma \in \textbf{c}$ be a $(q, Q)$-quasi-geodesic ray. Let 
$\beta'$ be a geodesic ray based at $\go$ with $\beta' \sim \beta$.

Take $y \in \gamma$ such that $\Vert y \Vert = r$. Then, by assumption, 
$$\gamma\vert_r \subseteq \calN_{\kappa}\big(\beta, m_\beta(q, Q)\big),$$
for any $r$. 
Hence, if $p$ is a nearest point projection of $y$ onto $\beta$, we have 
$$d_X(y, \beta) = d_X(y, p) \leq m_\beta(q, Q) \cdot \kappa(r).$$
Moreover,  by Lemma \ref{L:twice}, we have $\Vert p \Vert \leq 2 \Vert y \Vert  = 2 r$
and, by Lemma \ref{L:equivalence} (i), we have 
\begin{align*}
d_X(y, \beta') & \leq d_X(y, p) + d_X(p, \beta') \\
& \leq m_\beta(q, Q)\cdot \kappa(r) + 2 m_{\beta'}(q', Q') \cdot \kappa(r).
\end{align*}
Setting 
\[
\widetilde{m}_{\beta'}(q, Q) := m_\beta(q, Q) + 2 m_{\beta'}(q', Q'),
\] 
we get,
$$\gamma\vert_r \subseteq \calN_{\kappa}\big(\beta', \widetilde{m}_{\beta'}(q, Q)\big),$$
for any $r > 0$. 

Let now $p_r$ be a nearest point projection of $\gamma_r$ to $\beta'$. Then 
$$| d_X(\go, p_r) - r |  = | d_X(\go, p_r) - d_X(\go, \gamma_r) | \leq d_X(\gamma_r, p_r) \leq \widetilde{m}_{\beta'}(q, Q) \cdot \kappa(r).$$
Since $\beta'$ is geodesic,  we have 
$$d_X(p_r, \beta'_r) =  | d_X(\go, p_r) - r | \leq m_{\beta'}(q, Q) \cdot \kappa(r)$$
and 
\[
d_X(\gamma_r, \beta'_r) \leq d_X(\gamma_r, p_r) + d_X(p_r, \beta'_r) 
  \leq 2 \widetilde{m}_{\beta'}(q, Q) \cdot \kappa(r).
\]
But $\kappa(r)$ is sublinear, therefore, 
$$\lim_{r \to \infty} \frac{d_X(\gamma_r, \beta'_r)}{r} = 0$$
which implies $\gamma \in \textbf{b}$ and $\textbf{c} = \textbf{b}$. Finally, since 
$p \in \calU(\beta, r)$ implies $d_X(\go, p) \geq r$, the intersection 
$\bigcap_{r > 0} \calU(\beta, r)$ does not contain any point of $X$.

(4) Let $\beta_1$ be a $(q_1, Q_1)$-quasi-geodesic ray, $\beta_2$ a $(q_2, Q_2)$-quasi-geodesic ray with $\beta_1 \sim \beta_2$, and let $r_1 > 0$. 
For $r > 0$, let $\bfa \in \calU(\beta_2, r)$, and pick $\alpha \in \bfa$ be a $(q, Q)$-quasi-geodesic ray such that $m_{\beta_2}(q, Q)$ is small compared to $r$. By definition of $\calU(\beta_2, r)$ we have 
$$d_X(\alpha_r, \beta_2) \leq  m_{\beta_2}(q, Q) \cdot \kappa(r).$$
Let $p_r$ be a nearest point projection of $\alpha_r$ to $\beta_2$. By Lemma \ref{L:twice}, 
$$\Vert p_r \Vert \leq 2 r.$$
Moreover, by Lemma \ref{L:equivalence} (i), 
$$d_X(p_r, \beta_1) \leq \kappa(\Vert p_r \Vert ) m_{\beta_1}(q_2, Q_2) 
\leq 2 m_{\beta_1}(q_2, Q_2) \cdot  \kappa(r)$$
hence 
\begin{equation} \label{E:k-p}
d_X(\alpha_r, \beta_1)  \leq m_{\beta_2}(q, Q) \cdot  \kappa(r) 
  +  2 m_{\beta_1}(q_2, Q_2) \cdot  \kappa(r).
\end{equation}
Now, if we take 
\[
\kappa'(r) := m_{\beta_2}(q, Q) \cdot \kappa(r) 
  + 2 m_{\beta_1}(q_2, Q_2)\cdot  \kappa(r),
\] by Definition \ref{D:k-morse}, there exists $r_2$ 
such that \eqref{E:k-p} for $r = r_2$ implies
$$\alpha \vert_{r_1} \subseteq \mathcal{N}_{\kappa}\big(\beta_1,  m_{\beta_1}(q, Q)\big)$$
hence $\bfa \in \calU(\beta_1, r_1)$, as required.
\end{proof}

We now verify that a sequence of points of $X$ that sublinearly tracks
a quasi-geodesic ray $\gamma$, converges to the class of $\gamma$ in $\pka X$. 

\begin{lemma} \label{L:point-conv}
Let $\gamma \in \bfc$ be a $\kappa$-Morse quasi-geodesic ray based at $\go \in X$ and let $(x_n) \subseteq X$ be a sequence of points with $\Vert x_n \Vert \to \infty$. Moreover, suppose that there exists a constant $C > 0$ such that 
\begin{equation} \label{E:k-track}
d_X(x_n, \gamma) \leq C \cdot \kappa(\Vert x_n \Vert)
\end{equation}
for all $n$. Then the sequence $(x_n)$ converges to $\bfc$ in the topology of $X \cup \pka X$. 
\end{lemma}

\begin{proof}
In order to show the claim, we need to prove that for any quasi-geodesic ray $\beta \in \bfc$ and any $r > 0$ there exists $n_0$ such that 
for all $n \geq n_0$ we have 
$$x_n \in \mathcal{U}(\beta, r).$$ 
Equivalently, we need to show that, for any $r$ and any $(q, Q)$-quasi-geodesic segment $\alpha$ joining $\go$ and $x_n$ 
with $m_\beta(q, Q)$ small with respect to $r$, we have 
$$\alpha \vert_r \subseteq \mathcal{N}_\kappa(\beta, m_\beta(q, Q)).$$
Let $p_n$ be a nearest point projection of $x_n$ onto $\gamma$;  by Lemma \ref{L:twice}, $\Vert p_n \Vert \leq 2 \Vert x_n \Vert$. 
Now, by \eqref{E:k-track} and Lemma \ref{L:equivalence},  
\begin{align*}
d_X(x_n, \beta) & \leq d_X(x_n, p_n) + d_X(p_n, \beta) \\
& \leq C \cdot \kappa( \Vert x_n \Vert) + 2 m_\beta(q', Q') \cdot \kappa( \Vert x_n \Vert)
\end{align*}
where $(q', Q')$ are the quasi-geodesic constants of $\gamma$.
Note moreover that $\beta$ is $\kappa$-Morse by Lemma \ref{L:equivalence}. Hence, consider the sublinear function 
\[
\kappa'(r) := \big(C + 2 m_\beta(q', Q')\big) \cdot \kappa( r)
\] 
and apply the definition of $\kappa$-Morse, obtaining $R$ such that 
if $d_X(x_n, \beta) \leq \kappa'(R)$ then $\alpha \vert_r \subseteq \mathcal{N}_\kappa(\beta, m_\beta(q, Q))$. 
Thus, if we choose $n_0$ so that $\Vert x_n \Vert \geq R$ for all $n \geq n_0$, the definition of $\kappa$-Morse 
implies the following as needed, 
\begin{equation*}
\alpha \vert_r \subseteq \mathcal{N}_\kappa(\beta, m_\beta(q, Q)).
\qedhere
\end{equation*}

\end{proof}

We now show that the sets $\calU(\beta, r)$ and  $\partial \calU(\beta, r)$ as a 
neighborhood basis to define topologies on $X \cup \pka X$ and $\pka X$. 

\begin{definition}\label{nbhd-of-an-element}
Given an equivalence class $\textbf{b}$, we define the set $\mathcal{B}(\textbf{b})$
as the set of subsets $\mathcal{V} \subseteq X \cup \partial_{\kappa}X$ such that there 
exists $\beta \in \textbf{b}$ and $r > 0$ for which $\calU(\beta, r) \subseteq \mathcal{V}$.
Let $\partial \calB(\bfb)$ be the set of subsets of $\pka X$ of the form 
$\calV \cap \pka X$ where $\calV \in \mathcal{B}(\textbf{b})$, equivalently 
a set in $\partial \calB(\bfb)$ contains a set of the form $\partial \calU(\beta, r)$. 
Also, for $x \in X$, define $\calB(x)$ to be the set of subsets $\calV$ of $X$ such that 
$\calV$ contains a ball $B(x,r)$ of radius $r$ centered at $x$. 
\end{definition}

\begin{lemma}\label{topologyconditions}
For every $\bfb \in \partial_\kappa X$, the set $\calB(\bfb)$ satisfies the following properties:
\begin{enumerate}[{\rm (i)}]
\item Every subset of $X \cup \pka X$ which contains a set belonging to $\calB(\bfb)$ 
itself belongs to $\calB(\bfb)$;
\item Every finite intersection of sets of $\calB(\bfb)$ belongs to $\calB(\bfb)$;
\item The element $\bfb$ is in every set of $\calB(\bfb)$;
\item If $\calV \in \calB(\bfb)$ then there is $\calW \in \calB(\bfb)$ such that, 
for every $\bfa \in \calW$, we have $\calV \in \calB(\bfa)$.
\end{enumerate}
Furthermore, the same is true for subsets of $\partial \calB(\bfb)$ and
$\calB(x)$. 
\end{lemma}

\begin{proof}
We prove the Lemma for $\calB(\bfb)$. The proof for  $\partial \calB(\bfb)$
is identical. The proof for $\calB(x)$ is immediate from the fact that the 
open balls in $X$ define a neighborhood basis for $X$. 

\noindent 
(i) This is immediate from the definition of $\calB(\bfb)$.

\noindent 
(ii) It is enough to show that, for $\beta_1, \dots, \beta_k \in \bfb$ and 
$r_1, \dots, r_k >0$, the intersection  
\[
\calU(\beta_{1}, r_{1}) \cap \calU(\beta_{2}, r_{2})\cap \dots \cap \calU(\beta_{k}, r_{k}),
\]
belongs to $\calB(\bfb)$.
By Lemma \ref{L:Ubeta} (4), for any $i = 1, \dots, k$ there exists $R_i$ such that 
$$\calU(\beta_i, r_i) \supseteq \calU(\beta_1, R_i).$$
Thus, if we set 
\[
r := \max_{1 \leq i \leq k} \{ R_i \}
\qquad\text{we have}\qquad
\bigcap_{i=1}^k \calU(\beta_i, r_i) \supseteq \calU(\beta_1, r)
\]
and hence the intersection belongs to $\calB(\bfb)$. 

\noindent 
(iii) Established by Lemma \ref{L:Ubeta} (2). 

\noindent 
(iv) We need to prove the following claim.
\begin{claim}\label{claim1}
For any $\calU(\beta, r)$, there exists $r'$ (usually larger than $r$) such that if $\bfa \in \calU(\beta, r')$ then there exists $r''$ (depending on $\alpha$ and $r'$ but not on $\beta$)
such that $\calU(\alpha, r'') \subseteq \calU(\beta, r)$ for some $\alpha$ representative 
of $\bfa$. 
\end{claim}

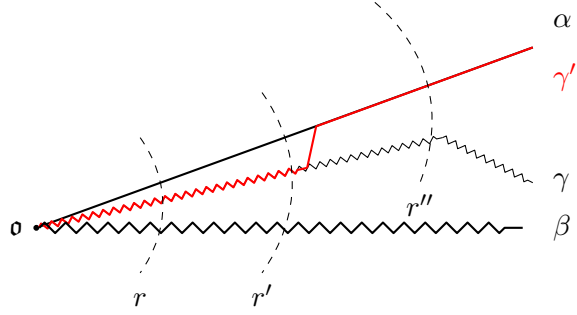
\begin{figure}
\begin{tikzpicture}[scale=0.6]
\tikzstyle{vertex} =[circle,draw,fill=black,thick, inner sep=0pt,minimum size=.5 mm]
[thick, 
    scale=1,
    vertex/.style={circle,draw,fill=black,thick,
                   inner sep=0pt,minimum size= .5 mm},
                  
      trans/.style={thick,->, shorten >=6pt,shorten <=6pt,>=stealth},
]

  \node[vertex] (o) at (0,0)[label=left:$\go$] {}; 
  \node(a) at (11,0)[label=right:$\beta$] {}; 

  \draw[thick](0,0)--(11, 4){};
  \node at (11,4) [label=above right:$\alpha$] {};
  \node at (11,4.1) [label=below right:{\color{red}$\gamma'$}] {};
   
  \draw [-,
line join=round,
decorate, decoration={
    zigzag,
    segment length=4,
    amplitude=1.2,post=lineto,
    post length=2pt
}] (o)--(9, 2){};

  \draw [-, red, thick,
line join=round,
decorate, decoration={
    zigzag,
    segment length=4,
    amplitude=1.2,post=lineto,
    post length=2pt
}] (o)--(6, 1.33){};

  \draw [thick, red] (6, 1.33) --(6.2, 2.25);
  \draw [thick, red] (6.2, 2.25)--(11, 4);
  \draw [-,
line join=round,
decorate, decoration={
    zigzag,
    segment length=4,
    amplitude=1.2,post=lineto,
    post length=2pt
}] (9,2)--(11, 1){};
  \node at (11,1) [label=right:$\gamma$] {};

  \draw [dashed](2.3, 2) to [bend left = 35] (2.3, -1){};
  \node at (2.3, -1) [label=below:$r$] {};
  \draw [dashed](5, 3) to [bend left = 35] (5, -1){};
  \node at (5, -0.8) [label=below:$r'$] {};
  \draw [dashed] (7.5, 5) to [bend left = 35] (8.5, 1){};
  \node at (8.5, 1.2) [label=below:$r''$] {};

  \draw [-,thick,
line join=round,
decorate, decoration={
    zigzag,
    segment length=8,
    amplitude=2,post=lineto,
    post length=4pt
}] (o)--(a){};

\end{tikzpicture}
\caption{The proof of Claim \ref{claim1}.}
\end{figure}

In particular, we will prove the claim for a geodesic representative $\alpha \in \bfa$, 
whose existence is established by Lemma \ref{L:Ubeta} (1). We adapt here the proof from \cite{QRT19}. 

Let us pick a $\kappa$-Morse quasi-geodesic ray $\beta$ and $r > 0$.  
Let 
$$M := \sup_{m_{\beta}(q, Q) \leq r} m_{\beta}(9q, Q)$$
and let  $r'$ be such that 
\begin{enumerate}
\item $r' > 2r$,
\item $M \leq \frac{r'}{2 \kappa(r')}$,
\item $r' > R(\beta,  r, M \kappa)$.
\end{enumerate}
Let $\alpha$ be a geodesic representative of $\bfa \in \calU(\beta, r')$. 
Choose $r''$ such that, 
\[
r'' \geq 2 r'
\qquad\text{and}
\qquad \sup_{m_\beta(q, Q) \leq r} m_{\alpha}(q, Q) \leq \frac{r''}{4\kappa(r'')}.
\]
Now consider $\bfc \in \calU(\alpha, r'')$. Let $\gamma \in \bfc$ be a
$(q, Q)$-quasi-geodesic ray, with $m_{\beta}(q, Q)$ small compared to $r$. 
By the choice of $r''$ above and by Lemma~\ref{L:symmetry},
\[
d_X(\alpha_{r''}, \gamma) \leq 2 m_{\alpha}(q, Q) \cdot \kappa(r'')
\leq \frac{r''}{2}. 
\]
We apply \lemref{Lem:surgery}, with radius being $r''$, to modify $\gamma$ 
to a $(9q, Q)$-quasi-geodesic ray $\gamma' \in \bfa$. Since $r' \leq r''/2$, 
we have $\gamma|_{r'} = \gamma'|_{r'}$. 

Also, $\gamma' \in \bfa \in \calU(\beta, r')$ and $m_{\beta}(9q, Q)$ is small compared to $r'$ (by point (2) above), therefore
\[
\gamma|_{r'} =\gamma' |_{r'} \subset \calN_\kappa\big(\beta, m_{\beta}(9q, Q)\big). 
\]
Hence, by the choice of $r'$ (point (3) above) and Definition \ref{D:k-morse}, we obtain
\[
\gamma|_r \subseteq \calN_\kappa\big(\beta, m_{\beta}(q, Q)\big). 
\]
This holds for every $\gamma \in \bfc$ with $m_\beta(q, Q)$ small compared to $r$, thus $\bfc \in \calU(\beta, r)$.
And this argument holds for every $\bfc \in \calU(\alpha, r'')$, 
therefore $\calU(\alpha, r'') \subset \calU(\beta, r)$.
\end{proof}

These properties for $\calB(x)$, $\calB(\bfb)$ and $\partial \calB(\bfb)$ are characteristic 
of the set of neighborhoods
of $\bfb$, as stated in the following proposition.

\begin{proposition}[\cite{Bourbaki}, Proposition 2]
Let $Y$ be a set. If to each element $y \in Y$ there corresponds a set $\calB(y)$ of
subsets of $Y$ such that properties (i) to (iv) from Lemma \ref{topologyconditions} above are satisfied, then there is a
unique topological structure on $Y$ such that for each $y \in Y$, 
$\calB(y)$ is the set of neighborhoods of $y$ in this topology. 
\end{proposition}

Thus we now use the sets $\partial \calB(b)$ to equip $\pka X$ with a topological 
structure and use the sets $\calB(b)$ and $\calB(x)$ to equip $X \cup \pka X$ 
with a topological structure. Note that, since neighborhoods in $\pka X$
are intersections of neighborhoods of $X \cup \pka X$ with $\pka X$, we
have that the inclusion $\pka X \subset X \cup \pka X$ is a topological 
embedding and $X \cup \pka X$ is a bordification of $X$. 

Recall that a set is open if it contains a neighborhood of each of its points. Thus, 
a set $\calW \subseteq \pka X$ is open if for every $\bfb \in \calW$ there is $\beta \in \bfb$ and  $r > 0$ such that $\partial \calU(\beta, r) \subset \calW$. A set 
$\calW \subset X \cup \pka X$ is open if its intersection with both $X$ and $\pka X$
is open. 

\subsection{Metrizability}
We now establish the metrizability of the space $\pka X$. To begin with, we need the following property of the topology:

\begin{lemma}\label{Lem:metrizable}
For each $\kappa$-Morse quasi-geodesic ray $\beta$ and $r > 0$, there exists a radius 
$r' >0$ such that for any point $\bfa \in \pka X$ there exists $r'' >0$ (depending
only on $\bfa$ and $r'$ and not on $\beta$) such that for every geodesic representative 
$\alpha_{0} \in \bfa$ 
\[
\calU(\alpha_{0}, r'') \cap \calU(\beta, r') \neq \emptyset 
\qquad \Longrightarrow\qquad \bfa \in \partial \calU(\beta, r).
\]
Similarly, for $x \in X$, let $B(x,1)$ be the ball of radius $1$ centered at $x$. Then 
\[
B(x,1) \cap \calU(\beta, r') \neq \emptyset 
\qquad \Longrightarrow\qquad x \in \calU(\beta, r).
\]
\end{lemma}

\begin{proof}
This will be done using the Surgery Lemma~\ref{Lem:surgery}. 
Pick a $\kappa$-Morse quasi-geodesic ray $\beta$ and $r > 0$. Let $\mathcal{Q} := \{ (q, Q) \ : \ m_{\beta}(q, Q) \leq \frac{r}{2\kappa(r)}\}$, 
which is bounded by properness. Set 
$$M := \sup_{(q, Q) \in \mathcal{Q}} m_{\beta}(9q, Q+1), $$
and $r' := R(\beta,  r, M \kappa)$. Let $\bfa \in \partial_{\kappa} X$. 

By Corollary ~\ref{Cor:m_beta}, there exists a constant $u > 0$ such that, 
for any geodesic ray $\star \in \bfa$ and any $(q, Q) \in \mathcal{Q}$ we have
\[
m_{\star}(1, 0) + 3 m_{\star}(q, Q) \leq u.\]
Let $R$ be the radius given in \lemref{Lem:surgery} associated to 
$q, Q$ and $2r'$, and let $r''$ be large enough so that 
\[
r'' \geq \max \{ 2u \cdot \kappa(r''), 2r', R\}.
\]
Let $\alpha_0$ be a geodesic ray in $\bfa$. By assumption, there is a point $\bfc$ 
inside the intersection
\[
\bfc \in \calU(\alpha_0, r'') \cap \calU(\beta, r').
\]
If $\bfc \in \pka X$, let $\gamma \in \bfc$ be a geodesic ray in this class and
if $\bfc \in X$, let $\gamma$ be a geodesic ray connecting $\go$ to $\bfc$. 
In either case, $\gamma_{r''}$ is well defined since, in the second case,  
$d_X(\go, \bfc) \geq r''$. Let $\alpha \in \bfa$ be a $(q, Q)$-quasi-geodesic ray with 
$m_\beta(q, Q)$ small compared 
to $r$. To conclude $\bfa \in \partial \calU(\beta, r)$ we need to show that 
$\alpha|_r \subset \calN_\kappa(\beta, m_\beta(q,Q))$.

Since $\bfc \in  \calU(\alpha_0, r'') $,
\[
d_X(\gamma_{r''}, \alpha_{0}) \leq m_{\alpha_{0}}(1, 0) \cdot \kappa(r'').
\]
Let $p  \in \pi_{\alpha_{0}}(\gamma_{r''})$. By definition of $u$ and $r''$, we have 
\[
\Norm p \leq r'' + m_{\alpha_0}(1, 0)\cdot \kappa(r'') \leq \frac{3}{2}r''.
\] 
Therefore, Lemma~\ref{L:equivalence} (ii) implies 
\[
d_X(p, \alpha) \leq 2 m_{\alpha_{0}}(q, Q)\cdot \kappa(p) 
  \leq 3 m_{\alpha_{0}}(q, Q) \cdot \kappa(r'').
\]
Hence,
\[d_X(\gamma_{r''}, \alpha) \leq d_X(\gamma_{r''}, p) + d_X(p, \alpha) \leq u \cdot \kappa(r'') \leq r''/2. \]
We can now apply the Surgery lemma (Lemma~\ref{Lem:surgery}) to $\alpha$ and $\gamma$ with radius $2r'$ to obtain a $(9q, Q)$-quasi-geodesic
ray  $\gamma'$ that 
is either in the class $\bfc$ if $\bfc \in \pka X$ or ends in $\bfc$ if $ \bfc \in X$
where $\gamma'|_{r'} = \alpha|_{r'}$.
Since $\bfc \in \calU(\beta, r')$, we have 
\[ \alpha|_{r'} = \gamma'|_{r'} \subset \calN_\kappa\big(\beta, m_{\beta}(9q, Q)\big). \]
Observing that $\alpha|_{r'}$ is $(q, Q)$-quasi-geodesic and letting 
$\kappa' = m_{\beta}(9q, Q) \cdot \kappa$, the definition of $r'$ and 
$\kappa$-Morse implies that 
\[
\alpha\vert_{r} \subset \calN_\kappa\big(\beta, m_{\beta}(q, Q)\big),
\]
hence $\bfa \in \calU(\beta, r)$.

To see the second assertion, assume $y \in B(x,1) \cap \calU(\beta, r')$. 
Let $\alpha$ be a $(q,Q)$-quasi-geodesic ray ending at $x$
where $m_\beta(q,Q)$ is small compared to $r$. 
Let $\gamma$ be a quasi-geodesic ray that is identical to $\alpha$ but at the 
last point in sent to $y$ instead of $x$. Then $\gamma$ is $(q,Q+1)$-quasi-geodesic. 
The definition of $r'$ and $\kappa$-Morse implies that 
\[
\alpha\vert_{r}= \gamma\vert_r \subset \calN_\kappa\big(\beta, m_{\beta}(q, Q)\big),
\]
hence $x \in \calU(\beta, r)$. 
\end{proof}

Our method for establishing metrizability is via the following criterion. 

\begin{theorem}[Theorem 3, \cite{Frink}] \label{Thm:Metrizable-condition} 
Assume, for every point $\bfb$ of a topological space, there exists a monotonically 
decreasing sequence $\calV_1(\bfb), \calV_2(\bfb), \cdots, \calV_i(\bfb), \cdots$
of neighborhoods whose intersection is $\bfb$ and such that the following holds: 
For every point $\bfb$ of the neighborhood space and every integer $i$, there exists 
an integer $j=j(\bfb, i) >i$ such that if $\bfa$ is a point for which $\calV_j(\bfa)$ and 
$\calV_j(\bfb)$ have a point in common then $\calV_j(\bfa) \subset \calV_i(\bfb)$. 
Then the space is homeomorphic to a metric space.
\end{theorem}

We check this condition for $\pka X$, using as neighborhoods $\calV_i(\bfb)$ the sets 
$\partial \calU(\beta, r)$ previously defined.

\begin{theorem}\label{Thm:Metrizable}
The space $\pka X$ is metrizable. 
\end{theorem}

\begin{proof}
Our goal is to construct, for any $i \in \NN$ and $\bfb \in \pka X$, neighborhoods $\calV_i(\bfb)$ which satisfy the conditions of Theorem \ref{Thm:Metrizable-condition}. 

Recall that, given a $\kappa$-Morse quasi-geodesic ray $\beta$ and $r > 0$, we can define $r'$ as 
\[ r'(\beta, r) : = R(\beta, r, M \kappa),\]
as in the proof of Lemma~\ref{Lem:metrizable}.
Note that both in Claim~\ref{claim1} and in Lemma~\ref{Lem:metrizable}, $r''$ does not depend on $\beta$ or $r$, but it depends on $\alpha$, $r'$ and on   
$$\sup_{m_\beta(q, Q) \leq r} m_\alpha(q, Q).$$
Since $q, Q \leq m_{\beta}(q, Q) \leq r \leq r'$, the maximum value of $q, Q$ can be bounded in terms of $\alpha$, $r'$, 
without referring to $\beta$ or $r$.
Hence, we can consider the function $r''(\alpha, r')$ such that both Claim~\ref{claim1} and Lemma~\ref{Lem:metrizable} hold.

For $i \in \NN$ and $\bfa \in \pka X$, pick a geodesic representative $\alpha_0 \in \bfa$ and define
\[
\calV_i(\bfa) := \partial \calU(\alpha_0, r_i(\bfa)),
\qquad\text{where}\qquad
r_i(\bfa) := \max\big(i, r''(\alpha_0, i) \big).
\]
Also, given $\bfb$ and $i$, we define $\rho_i(\bfb) := r'(\beta_0, r_i(\bfb))$, and
\[
j= j(\bfb, i) := \Big\lceil r'\big( \beta_{0}, \rho_i(\bfb) \big)\Big\rceil, 
\]
where $\beta_{0}$ is a geodesic representative in $\bfb$. Assume $\calV_j(\bfa)$ and $\calV_j(\bfb)$ have a point in common, that is, 
\[
\partial \calU \big(\alpha_0, r_j(\bfa)\big) \cap \partial \calU \big( \beta_0, r_j(\bfb)\big) 
  \not = \emptyset. 
\]
Then, since $r_j(\bfa) \geq r''(\alpha_0, j)$ and $r_j(\bfb) \geq j$, by Lemma~\ref{Lem:metrizable} we have 
\[
\bfa \in \partial \calU \Big(\beta_0,  \rho_i(\bfb)  \Big). 
\]
Now, Claim~\ref{claim1} implies 
\[
\partial \calU \Big(\alpha_{0}, r''\big(\alpha_{0},  \rho_i(\bfb) \big) \Big) 
\subset \partial \calU \big(\beta_{0}, r_i(\bfb)\big). 
\]
But $r_j(\bfa)=\max\big(j, r''(\alpha_0, j) \big)$, thus
\[ r_j(\bfa) \geq  r''(\alpha_{0}, r'(\beta_{0}, \rho_i(\bfb)) ) 
\geq r''(\alpha_0, \rho_i(\bfb)). 
\]
Therefore, 
\[
\partial \calU \big(\alpha_0, r_j(\bfa)\big) \subset \partial \calU \big(\beta_0, r_i(\bfb)\big),
\]
which is to say $\calV_j(\bfa) \subset \calV_i(\bfb)$. 
The theorem now follows from \thmref{Thm:Metrizable-condition}. 
\end{proof} 

Similarly, we have 

\begin{theorem}\label{Thm:Union-Metrizable}
The space $X \cup \pka X$ is metrizable. 
\end{theorem}

\begin{proof}
For $i \in \NN$ and $\bfa \in \pka X$, let 
$r_i(\bfa)$ be as in the proof of \thmref{Thm:Metrizable} and let 
\[
\calV_i(\bfa) := \calU(\alpha_0, r_i(\bfa)).
\]
For a point $x \in X$, we define $\calV_i(x) := B(x, \tfrac 1i)$, the ball of radius 
$\frac 1i$ centered around $x$. Since \lemref{Lem:metrizable} holds for
$\calU(\alpha_0, r_i(\bfa))$, the same proof as above works to check the conditions of 
\thmref {Thm:Metrizable-condition} for any point $\bfb \in \pka X$. 

For $x \in X$, we define $j(x,i) := 3 i$. Then, if 
\[
\calV_j(x) \cap \calV_j(y) \not = \emptyset,
\] 
there is a point $z \in B(x, \tfrac 1{3i}) \cap  B(y, \tfrac 1{3i})$ and, by the triangle 
inequality, 
\[
\calV_j(y)= B(y, \tfrac 1{3i}) \subset B(x, \tfrac 1i). 
\]

Also, if $\calV_j(x) \cap \calV_j(\bfb) \not = \emptyset$, for $x \in X$ and $\bfb \in \pka X$, 
then $B(x,1) \cap \calU(\beta, r_j(\bfb)) \neq \emptyset$. 
By the definition of $j(\bfb, i)$ and the second part of \lemref{Lem:metrizable}, 
this implies that  
\[
\calV_i(x) \subset B(x,1) \subset \calU(\beta, r_i(\bfb)).
\] 
Again, the theorem follows from \thmref{Thm:Metrizable-condition}. 
\end{proof}

We are now ready to establish the quasi-isometric invariance of $\pka X$. 

\begin{theorem}\label{invarianttopology}
Consider proper geodesic metric spaces $X$ and $Y$, let $\Phi \from X \to Y$ be a 
$(k, K)$-quasi-isometry and let $\kappa$ be a concave sublinear function. Then $\Phi$ induces a homeomorphism 
$\Phi^\star \from \pka X \to \pka Y$ 
where, for $\bfb \in \ps X$ and $\beta \in \bfb$, 
\[
\Phi^\star(\bfb) = [\Phi \circ \beta],
\]
where $[\cdot ]$ denotes the equivalence class of a quasi-geodesic ray.
\end{theorem}

The proof is identical to the proof of Theorem 5.1 in \cite{QRT19}.

\subsection{The union of $\pka X$}

We note that topologies of different sublinear boundaries are compatible.

\begin{proposition}[Proposition 4.10, \cite{QRT19}] 
Let $\kappa$ and $\kappa'$ be sublinear functions such that, for some $M > 0$
\[
\kappa'(t) \leq M \cdot \kappa(t), \qquad \forall \, t>0. 
\]
Then, $\partial_{\kappa'} X \subset \pka X$ as a subspace with the subspace topology.
\end{proposition} 

The proof is identical to the proof of Proposition 4.10 from \cite{QRT19} and
is skipped. In view of this proposition, we can define the \emph{sublinearly Morse boundary} of $X$ as
\[
\partial X := \bigcup_\kappa \pka X
\]
which is the space of equivalence classes (up to sublinear fellow traveling) of all 
sublinearly Morse quasi-geodesic rays in $X$. 

\begin{remark} 
An open neighborhood $\calV$ of 
a point $\bfb \in \partial X$ can be described as follows: assume 
$\bfb \in \pka X$ for some $\kappa$ and choose a quasi-geodesic ray $\beta \in \bfb$
and a radius $r>0$. Let $\calU_\kappa(\beta, r)$ be the neighborhood of 
$\bfb$ in $(X \cup \pka X)$ and let $\calV$ be the closure of $\calU_\kappa(\beta, r)$
in $(X \cup \partial X)$. That is, a point in $\calV \cap \partial X$ is a class of 
$\kappa'$-sublinearly Morse quasi-geodesic rays for some $\kappa'$ different from 
$\kappa$ that are eventually contained in  $\calU_\kappa(\beta, r)$. The intersection 
$\calV \cap X$ equals
$X \cap \calU_\kappa(\beta, r)$. 

Similar arguments as in \thmref{Thm:Metrizable} and \thmref{Thm:Union-Metrizable}
can be used to show that $\partial X$ and $(X \cup \partial X)$ are also metrizable. 
We skip these, for the sake of brevity. 
\end{remark} 

\section{General projections and weakly sublinearly contracting sets} \label{S:weakly-con}

In order to deal with several applications to proper geodesic spaces (in particular, our applications to mapping class groups and relatively hyperbolic groups we shall see next), the usual notion of nearest point projection may be ill-suited; 
for instance, it is well-known that nearest point projection to a closed subset of a general (e.g., not hyperbolic) metric space 
need not be, even coarsely, well-defined. 

Thus, we now introduce a more general notion of projection, which we call $\kappa$\emph{-projection}, 
where we allow an additive error which is controlled by a sublinear function $\kappa$. 

Let us denote as $\mathcal{P}(Z)$ the set of subsets of $Z$, and let us use the notation $\kappa(x) := \kappa(\Vert x \Vert)$. 

\begin{definition} \label{weakprojection}
Let $(X, d_X)$ be a proper geodesic metric space and $Z \subseteq X$ a closed subset, and let 
$\kappa$ be a concave sublinear function. A map $\pi_{Z} \from X \to \mathcal{P}(Z)$ is a $\kappa$-\emph{projection}  if there exist constants $D_{1}, D_{2}$, depending only on $Z$ and $\kappa$, such that for any points $x \in X$ and $z \in Z$, 
\[
\diam_X(\{z \} \cup \pi_{Z}(x)) \leq D_{1} \cdot d_X(x, z) + D_{2} \cdot \kappa(x).
\]
\end{definition}

A $\kappa$-projection differs from a nearest point projection by a uniform multiplicative error and a sublinear additive error.

\begin{lemma} \label{projection-property}
Given a closed set $Z$, we have for any $x \in X$
 $$\diam_X(\{x \} \cup  \pi_{Z}(x)) \leq (D_{1} + 1) \cdot d_X(x, Z) + D_{2}\cdot \kappa(x).$$
\end{lemma}

\begin{proof}
Let $z \in Z$ be a point that realizes $d_X(x, Z)$. Then, by triangle inequality and applying Definition \ref{weakprojection}, 
we obtain 
\begin{equation*}
\diam_X(\{x\} \cup \pi_{Z}(x) ) \leq d_X(x, z) + \diam_X(\{z\} \cup \pi_{Z}(x)) \leq  (D_{1}+1)\cdot d_X(x, Z) + D_{2} \cdot \kappa(x).
\qedhere
\end{equation*}
\end{proof}

We now formulate a general definition of $\kappa$\emph{-weakly contracting} with respect to a $\kappa$-projection $\pi_Z$.

\begin{definition}[$\kappa$-weakly contracting] \label{Def:generalContracting}
For a closed subspace $Z$ of a metric space $(X, d_X)$ and a $\kappa$-projection $\pi_Z$ onto $Z$, 
we say $Z$ is \emph{$\kappa$-weakly contracting} with respect to $\pi_Z$ if there are constants $C_{1}, C_{2}$, depending only on $Z$ and $\kappa$, such that, for every 
$x,y \in X$
\[
d_X(x, y) \leq C_{1} \cdot d_X( x, Z) \quad \Longrightarrow \quad
\diam_X \big( \pi_Z(x) \cup  \pi_Z(y)  \big) \leq C_{2} \cdot \kappa(x).
\]
\end{definition} 

In the special case that $\pi_Z$ is the nearest point projection and $C_1 = 1$, this property was called $\kappa$\emph{-contracting} in \cite{QRT19}. It was shown in \cite{QRT19} that, in the setting of CAT(0) spaces, this is stronger than the $\kappa$-Morse condition.

With respect to any projection we prove the following analogous statement of \cite[Theorem 3.14]{QRT19}:

\begin{theorem}[$\kappa$-weakly contracting implies sublinearly Morse] \label{Thm:W-Strong}
Let $\kappa$ be a concave sublinear function and let $Z$ be a closed subspace of $X$. 
Let $\pi_Z$ be a $\kappa$-projection onto $Z$ 
and suppose that $Z$ is $\kappa$-weakly contracting with respect to $\pi_Z$.
Then, there is a function $m_Z\from \RR^2 \to \RR$ such that, for every constant $r>0$ and every
sublinear function $\kappa'$, there is an $R= R(Z, r, \kappa')>0$ where the 
following holds: Let $\eta \from [0, \infty) \to X$ be a $(q, Q)$-quasi-geodesic ray 
so that $m_Z(q, Q)$ is small compared to $r$, let $t_r$ be the first time 
$\Norm{\eta(t_r)} = r$ and let $t_R$ be the first time $\Norm{\eta(t_R)} = R$. Then
\[
d_X\big(\eta(t_R), Z\big) \leq \kappa'(R)
\quad\Longrightarrow\quad
\eta([0, t_r]) \subset \calN_{\kappa}\big(Z, m_Z(q, Q)\big). 
\]
\end{theorem} 

The proof of this result is similar to the one in \cite{QRT19}, so we will postpone it to the appendix. 
Moreover, in the appendix we shall prove the following equivalence between $\kappa$-weakly contracting and $\kappa$-Morse 
(with a possibly different sublinear function) for any given closed set.

\begin{theorem}\label{sublinearlyequivalence}
Let $(X, \go)$ be a proper geodesic metric space with a fixed base point. Let $Z$ be a closed set and $\pi_Z$ be
 a $\kappa$-projection onto $Z$.
The  following hold:
\begin{enumerate}
\item If $Z$ is $\kappa$-weakly contracting with respect to $\pi_Z$,  then it is $\kappa$-Morse;
\item If $Z$ is $\kappa$-Morse, then it is $\kappa'$-weakly contracting with respect to $\pi_Z$ for some sublinear function $\kappa'$.
\end{enumerate}
\end{theorem}

\section{The Poisson boundary} \label{S:Poisson}

We now show a general criterion (Theorem \ref{T:poiss-general}) for the $\kappa$-Morse boundary of a group to be identified with its Poisson boundary.

\subsection*{Random walks}
Let $G$ be a locally compact, second countable group, with left Haar measure $m$, and let $\mu$ be a Borel probability measure on $G$, 
which we assume to $\emph{spread-out}$, i.e. such that there exists $n$ for which $\mu^n$ is not singular w.r.t. $m$. 
Given $\mu$, we consider the \emph{step space} $(G^\mathbb{N}, \mu^\mathbb{N})$, whose elements we denote as $(g_n)$. 
The \emph{random walk driven by }$\mu$ is the $G$-valued stochastic process $(w_n)$, where for each $n$ we define the product
$$w_n := g_1 g_2 \dots g_n.$$
We denote as $(\Omega, \mathbb{P})$ the \emph{path space}, i.e. the space of sequences $(w_n)$, where $\mathbb{P}$ is the measure induced by pushing forward the measure $\mu^\mathbb{N}$ from the step space. Elements of $\Omega$ are called \emph{sample paths} and will be also denoted as $\omega$. Finally, let $T : \Omega \to \Omega$ be the left shift on the path space. 

\subsection*{Background on boundaries}
Let us recall some fundamental definitions from the boundary theory of random walks. We refer to \cite{Kai00} for more details. 
Let $(B, \mathcal{A})$ be a measurable space on which $G$ acts by measurable isomorphisms;  a measure $\nu$ on $B$ is $\mu$-\emph{stationary} if $\nu = \int_G g_\star \nu \ d\mu(g)$, and in that case the pair $(B, \nu)$ is called a $(G, \mu)$-\emph{space}. 
Recall that a \emph{$\mu$-boundary} is a measurable $(G, \mu)$-space $(B,\nu)$ such that there exists 
a $T$-invariant, measurable map $\textbf{bnd} : (\Omega, \mathbb{P}) \to (B, \nu)$, called the \emph{boundary map}. 

Moreover, a function $f: G \to \mathbb{R}$ is $\mu$-\emph{harmonic} if $f(g) = \int_G f(gh) \ d\mu(h)$ for any $g \in G$. 
We denote by $H^\infty(G, \mu)$ the space of bounded, $\mu$-harmonic functions. 
One says a $\mu$-boundary is the \emph{Poisson boundary} of $(G, \mu)$ if the map 
\[
\Phi : H^\infty(G, \mu) \to L^\infty(B, \nu)
\]
given by $\Phi(f)(g) := \int_B f \ dg_\star \nu$ is a bijection. 
The Poisson boundary $(B, \nu)$ is the maximal $\mu$-boundary, in the sense that for any other $\mu$-boundary $(B', \nu')$ there exists a $G$-equivariant, 
measurable map $p: (B, \nu) \to (B', \nu')$.

Finally, a metric $d$ on $G$ is \emph{temperate} if there exists $C$ such that 
\[
m(\{ g \in G \ : \ d(1, g) \leq R \}) \leq C e^{CR}
\] 
for any $R > 0$.  A measure $\mu$ has \emph{finite first moment} with respect to $d$ if $\int_G d(1, g) \ d\mu(g) < +\infty$. 

We will use the ray approximation criterion from \cite{Kai00} for the Poisson boundary (for this precise version, see \cite{FT}). 

\begin{theorem} \label{T:ray}
Let $G$ be a locally compact, second countable group equipped with a temperate metric $d$, and let $\mu$ be a spread-out probability measure on 
$G$ with finite first moment with respect to $d$.
Let $(B, \lambda)$ be a $\mu$-boundary, and suppose that there exist maps $\pi_n : B \to G$ for any $n \in \mathbb{N}$ 
such that for almost every sample path $\omega = (w_n)$ we have
\begin{equation} \label{E:sublin}
\lim_{n \to \infty} \frac{d\big(w_n, \pi_n(\textup{\textbf{bnd}}(\omega))\big)}{n} = 0.
\end{equation}
Then $(B, \lambda)$ is the Poisson boundary of $(G, \mu)$. 
\end{theorem}

\subsection*{The $\kappa$-Morse boundary is the Poisson boundary}
We now apply this criterion to identify the Poisson boundary with the $\kappa$-Morse boundary. 
The following result was obtained in collaboration with Ilya Gekhtman.

\begin{theorem} \label{T:poiss-general}
Let $G$ be a finitely generated group, and let $(X, d_X)$ be a Cayley graph of $G$. 
Let $\mu$ be a probability measure on $G$ with finite first moment with respect to $d_X$, such that the semigroup generated 
by the support of $\mu$ is a non-amenable group. 
Let $\kappa$ be a concave sublinear function, and suppose that
for almost every sample path $\omega = (w_n)$, there exists a $\kappa$-Morse geodesic ray $\gamma_\omega$ such that 
\begin{equation} \label{E:sub-track}
\lim_{n \to \infty} \frac{d_X(w_n, \gamma_\omega)}{n} = 0.
\end{equation}
Then almost every sample path converges to a point in $\pka X$, and moreover the space $(\pka X, \nu)$, where $\nu$ is the hitting measure for the random walk, is a model for the Poisson boundary of $(G, \mu)$. 
\end{theorem}

\begin{proof}
By the subadditive ergodic theorem and finite first moment, the limit 
$$\ell := \lim_{n \to \infty} \frac{d_X(\go, w_n)}{n}$$
exists almost surely and is constant, and $\ell > 0$ since the group 
generated by the support of $\mu$ is non-amenable (see \cite[Theorem 8.14 and Corollary 12.5]{Woe}). 

By Lemma \ref{L:point-conv} and Eq. \eqref{E:sub-track}, almost every sequence converges to $[\gamma_\omega] \in \pka X$.
Thus, we can define $\textbf{bnd} : \Omega \to \partial_\kappa X$ as 
$$\textbf{bnd}(\omega) := \lim_{n \to \infty} w_n \in \partial_\kappa X,$$
which is $T$-invariant by definition. Moreover, $\textbf{bnd}$ is measurable, since it is a pointwise limit of the measurable functions $w_n$ 
with values in the space $X \cup \pka X$, which is metrizable by Theorem \ref{Thm:Union-Metrizable}.
Since $G$ is finitely generated, any word metric $d_X$ on it is temperate.

Finally, by Eq. \eqref{E:sub-track}, almost every sample path sublinearly tracks a $\kappa$-Morse quasi-geodesic ray. 
Hence, let us define $\pi_n : \pka X \to G$ as $\pi_n(\xi) := \alpha_{r_n}$ where 
$\alpha$ is a geodesic representative of the class of $\xi \in \pka X$, and $r_n := \lfloor \ell n \rfloor$. 

Now, let $\omega \in \Omega$ and $\gamma = \gamma_\omega$, and let $p_n$ be a nearest point projection of $w_n$ onto $\gamma$.
By \eqref{E:sub-track}, we have for almost every $\omega \in \Omega$
$$\lim_{n \to \infty} \frac{d_X(w_n, p_n)}{n} \to 0,$$ 
hence also $\frac{\Vert p_n \Vert}{n} \to \ell$.
Since $p_n$ and $\gamma_{r_n}$ lie on the same geodesic, this implies 
\begin{align} \label{E:track-2}
\frac{d_X(w_n, \gamma_{r_n})}{n} & \leq \frac{d_X(w_n, p_n)}{n} + \frac{d_X(p_n, \gamma_{r_n})}{n} \\ 
& \leq \frac{d_X(w_n, p_n)}{n} +  \frac{| \Vert p_n \Vert - r_n |}{n} \to 0 + \ell - \ell = 0
\end{align}
as $n \to \infty$. Finally, we obtain 
$$\frac{d_X(w_n, \pi_n(\textbf{bnd}(\omega)))}{n} = \frac{d_X(w_n, \alpha_{r_n})}{n} \leq \frac{d_X(w_n, \gamma_{r_n})}{n} + \frac{d_X(\gamma_{r_n}, \alpha_{r_n})}{n},$$
and the first term tends to $0$ because of \eqref{E:track-2}, while the second term tends to $0$ since $\alpha \sim \gamma$.
Hence, by Theorem \ref{T:ray}, $(\partial_\kappa X, \nu)$ is a model for the Poisson boundary of $(G, \mu)$. 
\end{proof}

\section{Boundaries of mapping class groups}

In this section, we show that for an appropriate choice of $\kappa$, the 
$\kappa$-Morse boundary of any mapping class group $G = \Map(S)$ can function as a topological model for 
the Poisson boundary of the pair $(G, \mu)$, where $\mu$ is any finitely-supported  
non-elementary measure.

We need to show that a generic sample path of such a random walk
sublinearly tracks a $\kappa$-Morse quasi-geodesic ray. We will do so but showing that, in fact, 
the limiting quasi-geodesic ray is $\kappa$-weakly contracting. 

\medskip

\subsection{Background on mapping class groups}
Let $S$ be a surface of finite hyperbolic type, let $\genus(S)$ be its genus and $\boundary(S)$ the number of its boundary components. 
Let $\Map(S)$ denote the mapping class group of $S$ equipped with a word metric $d_w$
associated to a finite generating set. That is, we are in the setting where 
$(X, d_X) = (\Map(S), d_{w})$.

\subsection*{Ending laminations}
Let $\calC(S)$ denote the curve graph of $S$ (see \cite{MM00} for definition and 
details). The curve graph is known to be $\delta$-hyperbolic
\cite{MM00}. 
By \cite{Klarreich}, the Gromov boundary of $\calC(S)$ can be identified with the space 
of ending laminations $\EL(S)$, that is, the space of minimal filling laminations after forgetting 
the measure. 

\subsection*{Subsurface projections}

By a subsurface $Y$ we always mean a connected $\pi_1$-injective subsurface
of $S$. For any subsurface $Y$ let $\calC(Y)$ denote the curve graph of $Y$.
Let $\partial Y$ denote the multi-curve 
consisting of all boundary components of $Y$. There is a projection map 
$\pi_Y \from \calC(S) \to \calC(Y)$ defined on a subset of $\calC(S)$ consisting of curves
that intersect $Y$. This is essentially a map that sends a curve $\alpha \in \calC(S)$ to 
a set of curves in $Y$ obtained from surgery between $\alpha$ and $\partial Y$. 
(Again, see \cite{MM00} for details). The set $\pi_Y(\alpha)$ has a uniformly bounded 
diameter in $\calC(Y)$, independent of $\alpha$ or $Y$. 

We can extend this projection to a map $\pi_Y \from \Map(S) \to \calC(Y)$ as follows. 
Consider a set $\theta$ of curves on $S$ that fill $S$. For example, 
following \cite{MM00}, we can assume $\theta$ is the union of a pants
decomposition and a set of \emph{dual curves}, one transverse to each curve
in the pants decomposition.  Then for $x \in \Map(S)$, define 
\[
\pi_Y(x) := \bigcup_{\alpha \in \theta} \pi_Y(x(\alpha)). 
\]
Again, the set $\pi_Y(x)$ has a uniformly bounded diameter in $\calC(Y)$. For, 
$x, y \in \Map(S)$ define 
\[
d_Y(x, y) := \diam_{\calC(Y)}(\pi_Y(x) \cup \pi_Y(y)).
\]
In particular, 
\[
d_S(x, y) := \diam_{\calC(S)} (x(\theta), y(\theta)). 
\]
Also, when $Y$ is an annulus with core curve $\alpha$, we often 
use $d_\alpha(x,y)$ instead of $d_Y(x,y)$. 

In the discussion above, $x$ and $y$ can be replaced with an ending lamination 
$\xi \in \EL(S)$ since $\xi$ has non-trivial projection to every subsurface
and $\pi_Y(\xi)$ is always well-defined. That is, we define 
\[
d_Y(x, \xi) := \diam_{\calC(Y)} (\pi_Y(x), \pi_Y(\xi)). 
\]

\subsection*{The distance formula}
In \cite{MM00}, it was shown that the word metric on $\Map(S)$ can be estimated 
up to uniform additive and multiplicative constants by these subsurface projection 
distances. To simplify the exposition, we adopt the following notation. We fix $S$ and a generating set for $\Map(S)$, 
and we say a constant $M$ is \emph{uniform} if it depends only on the topology of $S$ and the generating set. 
For two quantities $\textbf A$ and ${\textbf B}$, we write ${\textbf A} \prec {\textbf B}$ if 
there is a uniform constant $M$ such that
\[
{\textbf A} \leq M \cdot {\textbf B} + M. 
\]
We write ${\textbf A} \asymp {\textbf B}$ if ${\textbf A} \prec {\textbf B}$ and 
${\textbf B} \prec {\textbf A}$ and we use the notation $O({\textbf A})$ for a quantity 
that has an upper bound of $M \cdot {\textbf A}$.  Also, recall that, for $K>0$, 
\[
\lfloor \textbf A \rfloor_K := 
\begin{cases}
x & x \geq K\\
0 & x<K.
\end{cases}
\]
Now the Masur-Minsky distance formula can be stated as follows: 
there exists $K$ such that  
for $x,y \in \Map(S)$ we have:
\begin{equation}\label{df}
d_{w}(x, y) \asymp \sum_{Y \subseteq S} \big\lfloor d_{Y}(x, y) \big\rfloor_{K}.
\end{equation}

\subsection*{The hierarchy of geodesics}
To every pair of points $x, y \in \Map(S)$ one can associate a \emph{hierarchy
of geodesics} connecting $x(\theta)$ to $y(\theta)$ \cite[Theorem 4.6]{MM00}. The 
hierarchy $H=H(x,y)$ consists of a geodesic $[x,y]_S$ in $\calC(S)$ connecting 
$x(\theta)$ to $y(\theta)$ and other geodesics $[x,y]_Y$ in various curve graphs 
$\calC(Y)$, where $[x,y]_Y$ is essentially a geodesic connecting $\pi_Y(x)$ to 
$\pi_Y(y)$. Hence we write $H = \{ [x,y]_Y \}$. Besides $S$, other subsurfaces that appear 
in $H$ are described as follows: for every curve  $\alpha$ in $[x, y]_S$, we include every  
component of $S - \alpha$ that is not a pair of pants and the annulus $A_\alpha$ 
(the annulus whose core curve is $\alpha$). Also, if a subsurface $Y$ appears in $H$, 
then for every $\beta$ that appears in  $[x,y]_Y$, we also include
every component of $Y-\beta$ that is not a pair of pants and the annulus $A_\beta$.  
The length $|H|$ is defined to be the sum of the lengths $\big|[x,y]_Y\big|$ of geodesics 
$[x,y]_Y$. By \cite[Theorem 3.1]{MM00}, there exists $K$, depending only on the topology of $S$, such that 
for every subsurface $Y$, if $d_Y(x,y) \geq K$ then $Y$ is included in $H$. 
Furthermore,
by \cite[Theorems 6.10, 6.12 and 7.1]{MM00}  we have
\begin{equation} \label{eq:hierarchy}
d_w(x, y) \asymp |H(x,y)| \asymp \sum_{Y \textup{ in }H} \big|[x,y]_Y\big|. 
\end{equation} 
A \emph{resolution} $\calG(x,y)$ of a hierarchy $H(x,y)$ is a uniform quasi-geodesic in 
$\Map(S)$ connecting $x$ to $y$ where, for any subsurface $Y$, the projection of $\calG(x,y)$ to
$\calC(Y)$ is contained in a uniformly bounded neighborhood of the geodesic segment 
$[x,y]_Y$. 

We can also replace $x$ or $y$ with a point $\xi \in \EL$. 
We start with a (tight) geodesic $[x, \xi)_S$ in $\calC(S)$ and build $H(x, \xi)$ 
the same as before replacing, for every subsurface $Y$, $\pi_Y(y(\theta))$ with 
$\pi_Y(\xi)$. The resolution $\calG(x,\xi)$ of $H(x, \xi)$ is then a uniform quasi-geodesic 
in $\Map(S)$ starting from $x$ such that the \emph{shadow} of $\calG(x,\xi)$ in $\calC(S)$ 
(namely, the set $\{ z(\theta) : z \in \calG(x, \xi) \}$) converges to $\xi$. 

We use the hierarchy paths to show: 

\begin{proposition} \label{P:hierarchy}
Let $p := 3\genus(S) - 3 + \boundary(S)$ be the \emph{complexity} of $S$. For any $x, y \in X$, assume that 
$d_Y(x, y) \leq E$ for all $Y \subsetneq S$ and some $E > 1$. Then we have 
$$d_w(x, y) \prec d_S(x, y) \cdot E^p.$$
\end{proposition}

\begin{proof}
This is essentially contained in \cite{MM00}. We sketch the proof here and refer
the reader to \cite{MM00} for definitions and details. 
In view of \eqnref{eq:hierarchy}, we need to show 
\[
|H(x, y)| \prec d_S(x, y) \cdot E^p. 
\]
The restriction of $H(x,y)$ to a subsurface $Y$ is again a hierarchy which we denote 
with $H_Y(x,y)$. We check the Proposition inductively. When $S$ is $S_{1,1}$ or 
$S_{0, 4}$, we have $p=1$ and for every curve $\alpha$ in $S$, $S-\alpha$ does not 
have any complementary component that is not a pair of pants. Also, by assumption, 
for every $\alpha \in [x,y]_S$, we 
have $\big| [x,y]_\alpha \big| \prec d_\alpha(x,y) \leq E$. Therefore, 
$|H(x,y)| \prec E \cdot \big|[x,y]_S \big| \prec E \cdot d_S(x,y) \leq E$, as required. 

Now let $S$ be a larger surface and assume, by induction, that for every subsurface 
$Y$, the hierarchy $H_Y(x,y)$ satisfies $|H_Y(x,y)| \prec d_Y(x, y) \cdot E^{p-1}$. We have
\[
|H(x,y)| \prec \sum_{\alpha \in [x,y]_S} \left( \big| [x, y]_\alpha \big| 
 + \sum_{Y\subset S-\alpha} |H_Y(x,y)| \right)
\prec \big| [x,y]_S \big| \cdot (E + 2 d_Y(x, y) \cdot E^{p-1}). 
\]
But $d_Y(x, y) \leq E$ and $\big|[x,y]_S\big| \prec d_S(x, y)$, thus $|H(x, y)| \prec d_S(x, y) \cdot E^p$. 
\end{proof}

\subsection*{Projections in mapping class groups}
Here, we recall the construction of the center of a triangle in $\Map(S)$ according to 
Eskin-Masur-Rafi \cite{Projection}. For $x \in \Map(S)$ and a subsurface $Y$, we denote 
$\pi_Y(x(\theta))$ simply by $x_Y$.  Also, as before, for $x, y \in \Map(S)$, the geodesic segment
in $\calC(Y)$ connecting $x_Y$ and $y_Y$ is denoted by $[x,y]_Y$. For any subsurface 
$Y$, the curve graph $\calC(Y)$ is 
$\delta$-hyperbolic for some uniform constant $\delta$. Thus, for any three points 
$x, y, z \in \Map(S)$ and every subsurface $Y$, there exists a point 
$\cent_Y(x,y,z)$ in $\calC(Y)$ that is $\delta$-close to all three geodesic segments 
$[x,y]_Y$, $[x, z]_Y$ and $[y, z]_Y$. We refer to $\cent_Y(x,y,z)$  as the \emph{center} 
of the triple $x_Y, y_Y, z_Y \in \calC(Y)$. 

It was shown in \cite{Projection} that there is an element $\eta \in \Map(S)$
that projects near the center of $x_Y$, $y_Y$, $z_Y$ for every subsurface $Y$. 
More precisely: 

\begin{lemma}[{\cite[Lemma 4.11]{Projection}}] \label{L:Projection}
There exists a constant $D$ such that the following holds. For any $x, y, z \in \Map(S)$, 
there exists a point $\eta \in \Map(S)$ such that, for any subsurface $Y \subseteq S$, 
we have 
\[
d_Y\big(\eta_Y, \cent_Y(x,y,z)\big) \leq D. 
\]
We call $\eta$ the \emph{center} of $x$, $y$ and $z$ and we denote 
it by $\cent(x, y, z)$. 
\end{lemma}

Note that, as always, we can replace each of $x,y,z$ with an ending lamination 
$\xi \in \PML$. That is, $\cent(x,y, \xi)$ is a well-defined element of $\Map(S)$. 
From now on, we will denote as $\go$ the identity element in $\Map(S)$, which will 
function as base point.

\begin{definition}
Let $D$ be given from Lemma \ref{L:Projection}, and let $\xi \in \EL$ be an ending lamination on $S$. 
We define a \emph{$D$-cloud of a ray in the direction of $\xi$} to be
\[ 
\calZ(\go, \xi)  := \Big\{ z \in \Map(S) \ \big| \  d_{\calC(Y)} \big( z_Y,  [\go, \xi)_Y\big) \leq D \quad  \forall \, Y \Big\}. 
\]
\end{definition}
By construction, the resolution $\calG(\go, \xi)$ of the hierarchy $H(\go, \xi)$
is contained in $\calZ(\go, \xi)$. 
Fixing $\xi \in \EL$, we define a projection map
\begin{align*}
\Pi_\xi \from \Map(S) \to \calZ(\go, \xi)   \qquad\text{where}\qquad   \Pi_\xi(x) := \cent(\go, x, \xi), \qquad x \in \Map(S). 
\end{align*}

We now check that $\Pi_\xi$ satisfies the usual properties of a projection; in particular, it is a $\kappa$-projection 
according to Definition \ref{weakprojection}.

\begin{lemma}
For any $\xi \in \EL$, the map $\Pi_\xi$ is coarsely Lipschitz with respect to $d_w$. 
Furthermore, if 
$x \in \calZ(\go, \xi)$, then $d_w(x, \Pi_\xi(x))$ is uniformly bounded. As a consequence, $\Pi_\xi$ is a $\kappa$-projection. 
\end{lemma}

\begin{proof}
Consider points $x, x' \in \Map(S)$ where $d_{w}(x, x') \leq 1$.
Then $x(\theta)$ and $x'(\theta)$ have a uniformly bounded intersection 
number, which implies that there exists a uniform constant $C_1>0$ such that
\[
\forall \, Y, \qquad  d_Y(x_Y, x'_Y) \leq C_1. 
\]
Let $\eta := \cent(\go, x, \xi)$ and $\eta' := \cent(\go, x', \xi)$. 
Since $\calC(Y)$ is hyperbolic, the dependence of $\eta_Y$ on $x_Y$ is Lipschitz, 
that is, there exists a uniform constant $C_2>0$ such that 
\[
 \forall \, Y \subseteq S, \qquad  d_Y (\eta_Y, \eta'_Y) \leq C_2. 
\]
Now, Proposition \ref{P:hierarchy} implies that 
\[
d_{w} (\Pi_\xi(x), \Pi_\xi(x')) \prec (C_2)^{p+1}
\]
which means $\Pi_\xi$ is coarsely Lipschitz. Similarly, if $x \in \calZ(\go, \xi)$
then for $\eta = \cent(\go, x, \xi)$ we have $ d_Y (x_Y, \eta_Y) \leq C_2$
for all subsurfaces $Y$ and hence, $d_Y(x, \Pi_\xi(x))\prec  (C_2)^{p+1}$. 
\end{proof}

This projection has the following desirable property as shown by Duchin-Rafi \cite{Duchin-Rafi}.

\begin{theorem}[\cite{Duchin-Rafi}, Theorem 4.2]\label{duchin-rafi}
There exist constants $B_{1}, B_{2}$ depending on the topology of the surface $S$ 
and $D$ such that, for $x, y \in \Map(S)$,
\[
d_{w}(x, \calZ(\go, \xi)) \geq B_{1} \cdot d_{w}(x, y) \qquad \Longrightarrow \qquad d_{S}(\Pi_\xi(x), \Pi_\xi(y)) \leq B_{2}.
\] 
\end{theorem} 

In \cite{Duchin-Rafi}, the theorem is proven under the assumption that the geodesic (or cloud) is cobounded. However, 
the result holds in general: we will see in Proposition \ref{P:bounded-proj} a detailed proof for relatively hyperbolic groups, 
which can be easily adapted to mapping class groups.

\subsection*{Logarithmic projections}

We now consider the set of points in $\EL$ that have \emph{logarithmically bounded 
projection} to all subsurfaces. Given a proper subsurface $Y \subsetneq S$, let
$\partial Y$ denote the multi-curve of boundary components of $Y$ and define 
\[
\Norm{Y}_S :=d_{S}(\theta, \partial Y). 
\] 
Similarly, for $x \in \Map(S)$, define
\[
\Norm{x}_S :=d_{S}\big(\theta, x(\theta) \big).  
\]

\begin{definition}
For a constant $c > 0$, let $\calL_c$ be the set of points $\xi \in \EL$ such that 
\begin{equation} \label{E:log-proj}
 d_Y(\go, \xi) \leq c \cdot  \log \Norm{Y}_S
\end{equation}
for every subsurface $Y \subsetneq S$.
\end{definition}

\begin{proposition} \label{P:contracting-morse}
For any $\xi \in \calL_c$, the set $\calZ(\go, \xi)$ is $\kappa$-weakly contracting, where 
$\kappa(r) = \log^{p}(r)$. Furthermore, any resolution $\calG(\go, \xi)$ of
the hierarchy $H(\go, \xi)$ is also $\kappa$-Morse.
\end{proposition}

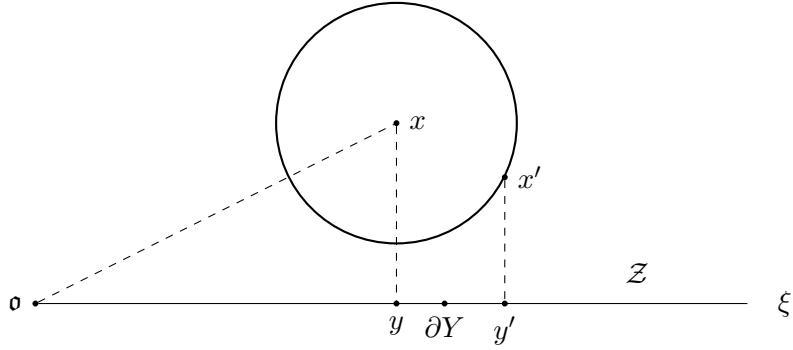
\begin{figure}[h!]
\begin{tikzpicture}[scale=0.8]
\tikzstyle{vertex} =[circle,draw,fill=black,thick, inner sep=0pt,minimum size=.5 mm]
[thick, 
    scale=1,
    vertex/.style={circle,draw,fill=black,thick,
                   inner sep=0pt,minimum size= .5 mm},
                  
      trans/.style={thick,->, shorten >=6pt,shorten <=6pt,>=stealth},
   ]

  \node[vertex] (o) at (-1,0)[label=left:$\go$] {}; 
  \node(a) at (11,0)[label=right:$\xi$] {}; 
  \node at (9,0)[label=above:$\calZ$] {}; 

  \draw (o)--(a){};
  \node [vertex] (b) at (5,3) [label=right:$x$] {};
  \node [vertex] (c) at (6.8,2.1) [label=right:$x'$] {};
  \draw[dashed] (o) to (b){};
  \draw[dashed] (b) to (5, 0){};
  \draw[dashed] (c) to (6.8, 0){};

  \node [vertex] (b) at (5,0) [label=below:$y$] {};
  \node [vertex] (b) at (5.8,0) [label=below:$\partial Y$] {};
  \node [vertex] (b) at (6.8,0) [label=below:$y'$] {};
  \draw[thick] (5,3) circle (2cm);
\end{tikzpicture}
\caption{The projections of $x$ and $x'$ to $\calZ$ are $\kappa$-close.}
\end{figure}

\begin{proof}
In this proof, we use the notations $\prec_c$ and $O_c$ to mean that the implicit constants additionally depend on $c$. 
Let $\calZ= \calZ(\go, \xi)$. Given $x, x' \in X$ where  
\[
B_1 \cdot d_w(x, x') < d_w(x, \calZ),
\]
let $y = \Pi_\xi(x)$, $y' = \Pi_\xi(x')$. We claim that, for every proper subsurface $Y$,
\[
d_Y(y, y') \prec_c \log \Norm{x}_S. 
\]

Since $\calC(S)$ is hyperbolic, nearest point projection in $\calC(S)$ is coarsely distance decreasing, hence $\Norm{y}_S \prec \Norm{x}_S$. Also, by 
Theorem \ref{duchin-rafi} 
\begin{equation} \label{eq:DR}
d_{S}(y, y') \leq B_2
\end{equation}
therefore, $\Norm{y'}_S \prec \Norm{x}_S$. Which means, for every curve $\alpha$ in the 
geodesic segment $[y, y']_S$ in $\calC(S)$ we have 
$d_S(\go, \alpha) \prec \Norm{x}_S$. The Bounded Geodesic Image Theorem \cite[Theorem 3.1]{MM00}
implies that if $d_Y(y, y')$ is large then $d_{\calC(S)}([y, y']_S, \partial Y) \prec 1$, hence 
\[
\Norm{Y}_S \prec \Norm{x}_S. 
\]

By the definition of $\Pi_\xi$, $y_Y$ and $y'_Y$ are $D$-close to the geodesic
segment $[\go, \xi]_Y$ in $\calC(Y)$ and, by assumption, the length of this 
segment is at most a uniform multiple of $\log \Norm{Y}_S$. 
Therefore, 
\[
d_Y(y, y') \prec \big| [\go, \xi]_Y\big| \prec_c \log \Norm{Y}_S \prec \log \Norm{x}_S. 
\]
In view of \eqnref{eq:DR} and Proposition \ref{P:hierarchy}, we get 
\[
d_w(y,y') \prec_c  \log^p \Norm{x}_S. 
\]
Now, by Theorem \ref{sublinearlyequivalence}, $\calZ(\go, \xi)$ is 
$\kappa$-Morse. Let $m_\calZ$ be the associated Morse gauge for $\calZ(\go, \xi)$. 

Now we show $\calG(\go, \xi)$ is also $\kappa$-Morse. Assume $\kappa'$ and 
$r>0$ be given (see \defref{D:k-morse}) and, using the fact that 
$\calZ(\go, \xi)$ is $\kappa$-Morse, let $R$ be a radius such that, for any 
$(q,Q)$-quasi-geodesic ray 
$\beta$ in $\Map(S)$ with $m_\calZ(q,Q)$ small compared to $r$, we have 
\[
d_w(\beta_R, \calZ(\go, \xi) ) \leq \kappa'(R) 
\qquad \Longrightarrow\qquad
\beta|_r \subset \calN(\calZ(\go, \xi), m_\calZ(q,Q)). 
\]
Also, assume 
\[
d_w(\beta_R, \calG(\go, \xi) ) \leq \kappa'(R). 
\]
We need to show that every $x \in \beta|_r$ is close to $\calG(\go, \xi)$. 

Since $\calG(\go, \xi) \subset \calZ(\go, \xi)$ we can still conclude that 
there is a point $y \in \calZ(\go, \xi)$ with
\[
d_w(x,y) \leq m_\calZ(q, Q) \cdot \kappa(x). 
\]
In fact $y$ can be taken to be $\Pi_\xi(x)$ and hence $\Norm{y}_S \prec \Norm{x}_S$. 
Let $z$ be a point in $\calG(\go, \xi)$ where $d_S(z_S , y_S) \leq D$
(such a point exists since the shadow of $\calG(\go,\xi)$ to $\calC(S)$ is the
geodesic ray $[\go, \xi)_S$). Since $y,z \in \calZ(\go, \xi)$, we have for every 
subsurface $Y$ that 
\[
d_Y(y,z) \prec_c \log \max(\Norm{y}_S, \Norm{z}_S) \prec \log (\Norm{x}_S + D) 
 \prec \log \Norm{x}_S.
\]
Therefore, by \propref{P:hierarchy}, we have 
\[
d_w(y,z) \prec_c  \log^p \Norm{x}_S \prec \kappa(x).
\]
And hence, 
\[
d_w(x,z) \leq d_w(x,y) + d_w(y,z) \prec_c m_\calZ(q,Q) \cdot \kappa(x). 
\]
We have shown 
\[
\beta|_r \subset \calN\Big(\calG(\go, \xi), O_c\big(m_\calZ(q,Q)\big)\Big). 
\]
That is, $\calG(\go, \xi)$ is $\kappa$-Morse with a Morse gauge 
$m_\calG = O_c(m_\calZ)$. 
\end{proof}

\subsection{Convergence to the $\kappa$-Morse boundary}

Let $\mu$ be a probability measure on $\Map(S)$. We say that $\mu$ is \emph{non-elementary} if the semigroup generated by its support 
contains two pseudo-Anosov elements with disjoint fixed sets in $PMF$. 

Let us recall some useful facts on random walks on the mapping class group.

\begin{theorem} \label{T:RW}
Let $\mu$ be a finitely supported, non-elementary probability measure on $\Map(S)$. Then: 
\begin{enumerate}
\item
For almost every sample path $\omega = (w_n)$, the sequence $(w_n)$ converges to a point $\xi_\omega$ in the Gromov boundary of $\calC(S)$, which is $\EL(S)$. 
\item
Moreover, there exists $\ell > 0, c < 1$ such that 
$$\mathbb{P}\left( d_S(\go, w_n) \geq \ell n  \right) \geq 1 - c^n$$
for any $n$.
\item
Further, for any $k > 0$ there exists $C > 0$ such that 
$$\mathbb{P}\left( d_S(w_n, \gamma_\omega) \geq C \log n \right) \leq n^{-k}$$
for any $n$, where $\gamma_\omega = [\go, \xi_\omega)_S$. 
\end{enumerate}
\end{theorem}

Claim (2) is proven by Maher (\cite{Maher}, \cite{MaherExp}), while (1) and (3) are proven in \cite{MaherTiozzo}. 
Also, exactly the same proof as in \cite[Theorem A.17]{QRT19} (inspired by \cite{ST} and \cite[Lemma 4.4]{sisto-track}), yields for the mapping class group:

\begin{theorem}  \label{T:log-exc}
Let $\mu$ be a finitely supported, non-elementary probability measure on $\Map(S)$.
Then for any $k > 0$ there exists $C > 0$ such that for all $n$ we have
$$\mathbb{P}\left(\sup_Y d_{Y}(\go, w_n) \geq C \log n \right) \leq C n^{-k},$$
where the supremum is taken over all (proper) subsurfaces $Y$ of $S$. 
As a consequence, for almost every sample path there exists $C > 0$ such that for all $n$
$$\sup_Y d_{Y}(\go, w_n ) \leq C \log n.$$
\end{theorem}

We now prove show that almost every sample path converges to a point in the $\kappa$-Morse boundary of the mapping class group, 
where $\kappa(r) = \log^p(r)$.

\begin{theorem} \label{T:converge}
Let $\mu$ be a finitely supported, non-elementary probability measure on $\Map(S)$, and let $\kappa(r) := \log^p(r)$. Then:
\begin{enumerate}
\item
Almost every sample path $\omega = (w_n)$ converges to a point in the $\kappa$-Morse boundary with respect to the topology of $X \cup \pka X$;
\item
Moreover, for almost every sample path there exists a $\kappa$-Morse geodesic ray $\calG_\omega$ in $\Map(S)$ such 
that  
$$\limsup_{n \to \infty} \frac{d(w_n, \calG_\omega)}{\log^{p+1}(n)} < +\infty.$$
\end{enumerate}
\end{theorem}

\begin{proof}
By Theorem \ref{T:RW}, for almost every sample path, $w_n \theta$ converges to 
a point $\xi_\omega$ on the boundary of $\calC(S)$ (with respect to the topology on $\calC(S) \cup \EL(S)$). 

Let $\calG = \calG(\go, \xi_\omega)$ be 
a resolution of a hierarchy towards $\xi_\omega$ and let $\gamma_\omega := [\go, \xi_\omega)_S$ be 
the shadow of $\calG$, which is a geodesic ray in $\calC(S)$ starting from  $\theta$ and limiting to $\xi_\omega$. 

By Proposition \ref{P:contracting-morse}, in order to prove that $\mathcal{G}$ is $\kappa$-Morse, 
it is sufficient to prove that 
there is a constant $c$ such that  
\begin{equation} \label{E:supY}
\sup_Y d_Y(\go, \xi_\omega) \leq c \log d_S(\go, \partial Y).
\end{equation}

In the next few steps, we will show that \eqref{E:supY} holds for almost every $\omega$.

\medskip
Let $p_n$ be a nearest point projection (in $\calC(S)$) of $w_n \theta$ to $\gamma_\omega$ 
and $c_n := \cent(\go, w_n , \xi_\omega)$.  
By the definition of center and the hyperbolicity of $\calC(S)$, there is $D'>0$ depending on $D$ and $\delta$
such that $d_S(c_n, p_n) \leq D'$ for any $n$. 

\medskip
\textbf{Step 1.}
We claim that there exists $C > 0$ such that 
\begin{equation} \label{E:supY1}
\mathbb{P}(\sup_Y d_Y(w_n, c_n) \geq  C \log n) \leq C n^{-2}
\end{equation}
for all $n$.

\begin{proof}
Since the drift of the random walk is positive with exponential decay (Theorem \ref{T:RW} (2)), we have by the Markov property
that there exists $0 < C_0 < 1$ such that
\begin{equation} \label{E:C0}
\mathbb{P}(d_S(w_n, w_{2n}) \leq \ell  n) = \mathbb{P}(d_S(\go, w_{n} ) \leq \ell  n) \leq (C_0)^n\qquad \forall n
\end{equation}
where $\ell > 0$ is the drift of the random walk.
By Theorem \ref{T:RW} (3), there exists $C_1 > 0$ so that 
\begin{equation} \label{E:C1}
\mathbb{P}( d_S(w_n, p_n) \geq C_1 \log n) \leq C_1 n^{-2} \qquad \forall n
\end{equation}
and, by Theorem \ref{T:log-exc}, there exists $C_2 >0$ such that
\begin{equation} \label{E:C2}
\mathbb{P}(\sup_Y d_Y(w_n, w_{2n}) \geq C_2 \log n) = \mathbb{P}(\sup_Y d_Y(\go, w_{n})  \geq C_2 \log n) \leq C_2 n^{-2} \qquad \forall n.
\end{equation}
Now, since projection in a $\delta$-hyperbolic space is coarsely distance-decreasing, 
$$d_Y(w_n, w_{2n}) \leq C_2 \log n$$ 
implies 
$$d_Y(c_n, c_{2n}) \leq C_2 \log n + 2 D,$$
for some $D$ which depends only on $\delta$, hence we also have for some new constant $C_3 >0$
\begin{equation} \label{E:C3}
\mathbb{P}(\sup_Y d_Y(c_n, c_{2n}) \geq C_3 \log n) \leq C_3 n^{-2} \qquad \forall n.
\end{equation}
Now, the complement of the union of all events expressed by \eqref{E:C0}, \eqref{E:C1},\eqref{E:C2},\eqref{E:C3} 
has measure at least $1 - C_5 n^{-2}$ for some new $C_5$. We will consider from now on a sample path in such a set.

Let us now pick a subsurface $Y$. By the bounded geodesic image theorem, there exists $C_4$ (independent of $Y$) such that if $\partial Y$ is far from the geodesic segment $[w_n, p_n]$ in $\calC(S)$, namely 
$$d_{\calC(S)}(\partial Y, [w_n, p_n]) \geq C_4,$$ 
then $d_Y(w_n, c_n) \leq C_4$ is uniformly bounded.

On the other hand, if $d_{\calC(S)}(\partial Y, [w_n, p_n]) \leq C_4$, let us denote by $q_1$ a nearest point projection of $\partial Y$ onto 
$[w_{n}, p_{n}]$ and by $q_2$ a nearest point projection of $\partial Y$ onto $[w_{2n}, p_{2n}]$. Then
\begin{align*}
d_S(\partial Y, [w_{2n}, p_{2n}]) = d_S(\partial Y, q_2) &  \geq d_S(w_n, w_{2n}) - d_S(\partial Y, q_1) - d_S(q_1, w_n) - d_S(w_{2n}, q_2) \\
						 &  \geq d_S(w_n, w_{2n}) - d_S(\partial Y, q_1) - d_S(w_n, p_n) - d_S(w_{2n}, p_{2n}) \\
& \geq \ell n - C_4 - C_1 \log n - C_1 \log (2 n) > C_4
\end{align*}
for $n$ sufficiently large (independently of $Y$ and $\omega$). Hence, by the bounded geodesic image theorem, 
$$d_Y(w_{2n}, c_{2n}) \leq C_4.$$
By putting these estimates together and by the triangle inequality we obtain 
\begin{align*}
d_Y(w_n, c_n) & \leq d_Y(w_n , w_{2n}) + d_Y(w_{2n} , c_{2n}) + d_Y(c_{2n}, c_n) \\
& \leq C_2 \log n + C_4 + C_3 \log n \leq C \log n 
\end{align*}
by setting $C$ appropriately, which proves the claim.
\end{proof}

\textbf{Step 2.}
We show that there exists a constant $C_7$ such that for any $n$ we have 
\begin{equation} \label{E:sup2}
\mathbb{P}(\sup_Y d_Y(\go, c_n) \geq C_7 \log d_S(\go, p_n) + C_7) \leq C_7 n^{-2}.
\end{equation}

\begin{proof}
First note that by the triangle inequality if $d_S(\go, w_n) \geq \ell n$ and $d_S(w_n, p_n) \leq C_1 \log n$ then
\begin{align*}
d_S(\go, p_n) & \geq d_S(\go, w_n) - d_S(w_n, p_n) \\
& \geq \ell n - C_1 \log n \geq \ell n/2
\end{align*}
for $n$ sufficiently large, hence by \eqref{E:C0} and \eqref{E:C1} there exists $C_6$ for which
$$\mathbb{P}(d_S(\go, p_n) \geq \ell n/2) \geq 1 - C_6 n^{-2}\qquad \forall n.$$
Then, by the triangle inequality
\begin{align*}
d_Y(\go, c_n) & \leq d_Y(\go, w_n) + d_Y(w_n, c_n) \\
\intertext{and by Theorem \ref{T:log-exc} and Step 1.}
& \leq C_2 \log n + C \log n \\
& \leq (C_2 + C) \log d_S(\go, p_n) + (C_2 + C) \log(2/\ell) \\ 
& \leq C_7 \log d_S(\go, p_n) + C_7
\end{align*}
for the appropriate choice of $C_7$.
\end{proof}

\textbf{Step 3.}
For almost every $\omega$, there exists $C_8(\omega)$ such that,  
\begin{equation} \label{E:sup3}
d_Y(\go, \xi_\omega) \leq C_8(\omega) \log d_S(\go, \partial Y)
\end{equation} 
for every subsurface $Y \neq S$.

\begin{proof}
By Step 2. and Borel-Cantelli, for almost every $\omega$ there exists $n_0 = n_0(\omega)$ such that 
\begin{equation}\label{E:sup-last}
\sup_Y d_Y(\go, c_n) \leq C_7 \log d_S(\go, p_n) + C_7
\end{equation}
for any $n \geq n_0$.

Now, observe that the random walk is finitely supported, so $d_S(w_n, w_{n+1})$ is uniformly bounded, hence, since nearest point 
projection is coarsely distance decreasing, there exists $D_1 > 0$ such that 
$$d_S(p_n, p_{n+1}) \leq D_1$$ 
is also uniformly bounded.

Now, let $Y$ be a subsurface. By the bounded geodesic image theorem, $d_Y(\go, \xi_\omega) \leq C_4$ unless $\partial Y$ lies in a $C_4$-neighborhood of $[\go, \xi_\omega)_S$. 
Let us suppose that $\partial Y$ lies in a $C_4$-neighborhood of $[\go, \xi_\omega)_S$
and let $n = n(Y, \omega)$ be the smallest integer such that $d_S(\go, p_{n}) \geq d_S(\go, \partial Y) + C_4$. 

Note that by minimality, $d_S(\go, p_n) - D_1 \leq d_S(\go, p_{n-1}) \leq d_S(\go, \partial Y) + C_4 \leq d_S(\go, p_{n})$, hence $d_S(\go, p_n) \leq d_S(\go, \partial Y) +D_1 + C_4$.

Moreover, by the bounded geodesic image theorem, 
$$|d_Y(\go, \xi_\omega) - d_Y(\go, c_n)| \leq C_4$$
since the projection of $[p_n, \xi)_S$ to $\calC{(Y)}$ is uniformly bounded.
Hence by \eqref{E:sup-last}, if $n \geq n_0(\omega)$ then 
\begin{align*}
d_Y(\go, \xi_\omega) & \leq d_Y(\go, c_n) + C_4 \\
& \leq C_7 \log d_S(\go, p_n) +  C_7 + C_4 \\ 
& \leq C_7 \log (d_S(\go, \partial Y) + D_1 + C_4) + C_7 + C_4
\end{align*}
On the other hand, if $n(Y, \omega) \leq n_0(\omega)$ then $d_S(\go, \partial Y) \leq d_S(\go, p_n) \leq d_S(\go, p_{n_0(\omega)})$, 
and there are at most finitely many such subsurfaces $Y$ for which the projection of $[\go, \xi_\omega)_S$ onto $\calC(Y)$ is large. 
Hence 
$$\sup_{Y : n(Y, \omega) \leq n_0} d_Y(\go, \xi_\omega) < + \infty,$$
thus the claim follows by adjusting the constant to a constant $C_8(\omega)$, which depends on $\omega$,
to take into account the initial part.
\end{proof}

Now, from \eqref{E:supY} and Proposition \ref{P:contracting-morse}, we obtain that the quasi-geodesic ray $\mathcal{G}$ in $\Map(S)$ is $\kappa$-Morse. 

\bigskip
Let us now prove the second claim, namely the tracking estimate between the sample path and the geodesic ray. 
From that and Lemma \ref{L:point-conv}, it follows that almost every sample path converges in the topology of $X \cup \pka X$.

For a given sample path $\omega = (w_n)$, let $\xi \in \EL(S)$ be the ending lamination the path $(w_n)$ is converging to, 
and let $c_n := \cent(\go, w_n, \xi)$. 
Note that by construction the projection of $c_n$ onto $\calC(S)$ is within uniformly bounded distance of $p_n$. 
Then, by equation \eqref{E:supY1}, there exists $C > 0$ such that 
$$\mathbb{P}(\sup_Y d_Y(w_n, c_n) \geq  C \log n) \leq C n^{-2}$$
for any $n$. 
Applying Proposition \ref{P:hierarchy} with $D = \log n$, there exists $C > 0$ such that 
\begin{equation} \label{E:sub-tracking}
\mathbb{P}(d_w(w_n, c_n) \geq C \log^{p+1}(n) ) \leq C n^{-2}
\end{equation}
for any $n$. 
As a consequence, the Borel-Cantelli lemma implies that 
for almost every $\omega \in \Omega$ there exists $n_0$ such that 
\begin{equation} \label{E:sub-tr2}
d_w(w_n, c_n) \leq C \log^{p+1}(n)
\end{equation}
for all $n \geq n_0$.  The claim follows by noting that $d_w(w_n, \mathcal{G}) \leq d_w(w_n, c_n) + O(1)$.

Finally, by Lemmas \ref{L:Ubeta} and \ref{L:equivalence}, we can replace the quasi-geodesic $\calG$ by a geodesic in the same equivalence class. 
\end{proof}

We now complete the proof of Theorem \ref{T:PB-intro} by identifying the $\kappa$-Morse boundary with the Poisson boundary. 

\begin{theorem} \label{T:mcg-poiss}
Let $\mu$ be a non-elementary, finitely supported measure on $G = \Map(S)$.
Then for $\kappa(r) := \log^p(r)$, the $\kappa$-Morse boundary is a model for the Poisson boundary of $(G, \mu)$. 
\end{theorem}

\begin{proof}
By Theorem \ref{T:converge}, almost every sample path sublinearly tracks a $\kappa$-Morse geodesic ray, with $\kappa(r) = \log^p(r)$, 
so we can apply Theorem \ref{T:poiss-general}.
\end{proof}

\section{Relatively hyperbolic groups}

\subsection{Background}

Let $G$ be a finitely generated group, and let $P_{1}, P_{2}, \dots, P_{k}$ be a set of subgroups of $G$ which we call \emph{peripheral subgroups}. 
Fix a finite generating set $S$.  Let  $(G, d_{G})$ denote the Cayley graph of $G$ with respect to $S$, equipped 
with the word metric. Following \cite{sistopaper}, we denote by $(\widehat{G}, d_{\widehat{G}})$
the metric graph obtained from the Cayley graph of $G$ by adding an edge 
between every pair of distinct vertices contained in a left coset $P$ of a peripheral
subgroup. Setting up this way, $G$ and $\hG$ have the same
vertex set, so any set in $G$ can also be considered as a set in $\hG$. However, for $x,y \in G$, $d_{\hG}(x,y) \leq d_G(x,y)$. 
In this Section, $[x, y]$ will denote the geodesic segment between $x, y$ in the metric $d_{\hG}$
and $\llbracket x, y \rrbracket$ the geodesic segment in $d_G$.

\begin{definition}
A group $G$ is \emph{relatively hyperbolic}, relative to peripheral subgroups $P_{1}, P_{2}, \dots, P_{k}$,  if the graph $\hG$ has the properties:
\begin{itemize}
\item It is $\delta$-hyperbolic;
\item It is \emph{fine}: for each integer $n$, every edge belongs to only finitely many simple cycles of length $n$.
\end{itemize}
The collection of all cosets of peripheral subgroups is denoted by $\calP$. We shall denote as $\calN_D(P)$ the $D$-neighborhood of $P$ in the metric $d_G$.
\end{definition}

Moreover, we say a relatively hyperbolic group is \emph{non-elementary} if it is infinite, not virtually cyclic, and each $P_i$ is infinite and with infinite index in $G$. We now fix a non-elementary relatively hyperbolic group $G$ and a generating set. We shall use the symbols $\prec$ and $\asymp$ as before, where the implicit constants depend only on $G$ and the generating set; we shall write $\asymp_K$ if the implicit constants additionally depend on $K$. By definition of relative hyperbolicity, the graph $\widehat{G}$ is $\delta$-hyperbolic for some $\delta > 0$.  Moreover, for $P \in \mathcal{P}$ we let  $\pi_P : G \to P$ be a nearest point projection to $P$, and denote 
\[
d_{P}(x, y): = d_{G}(\pi_{P}(x), \pi_{P}(y)).
\] 

We need some properties of relatively hyperbolic groups from \cite{sistopaper}, \cite{sisto-track}. A \emph{lift} of a geodesic ray in $\hG$ is a path in $G$ obtained by substituting edges labeled by an element of some $P_{i}$, and possibly their endpoints, with a geodesic in the corresponding left coset.
Given $D, R$ and a geodesic segment $\gamma$ in $G$, we define the set $\textup{deep}_{D, R}(\gamma)$ as the set of points $p$ of $\gamma$ 
that belong to some subgeodesic $\llbracket x_1, y_1 \rrbracket$ of $\gamma$ with endpoints in $\calN_{D}(P)$ for some $P \in \calP$ and with $d_G(x_1, p) > R, d_G(y_1, p) > R$. 
A point which does not belong to $\textup{deep}_{D, R}(\gamma)$  is called a \emph{transition point}.

\begin{proposition} \label{P:RH-recall}
Let $G$ be a relatively hyperbolic group, and fix a generating set. Then we have: 
\begin{enumerate}
\item \textup{(Coarse lifting property, \cite[Proposition 1.14]{sistopaper})} There exist uniform constants $q_0, Q_0>0$ such that if $\alpha$ is a geodesic in $\widehat{G}$, then its lifts are $(q_0, Q_0)$-quasi-geodesics in $G$. 
\item \textup{(Distance formula, \cite[Theorem 0.1]{sistopaper})} There exists $K_0$ such that, for any $K \geq K_0$ and every pair of points $x, y \in G$,
\[
d_{G}(x, y) \asymp_K \sum_{P \in \calP} \lfloor d_{P}(x, y) \rfloor_{K} + d_{\widehat{G}}(x, y).
\]
\item \textup{(Bounded geodesic image theorem, \cite[Lemma 1.15]{sistopaper})}
There exists $L_0 \geq 0$ such that, if $d_P(x, y) \geq L_0$ for some $P \in \mathcal{P}$, 
all geodesics in $\widehat{G}$ connecting $x$ to $y$ contain an edge in $P$. 
Moreover, for every $q, Q>0$ there is $D = D(q, Q)$ such that every 
$(q, Q)$-quasi-geodesics in $G$ connecting $x$ and $y$ intersect the balls
of radius $D$, $B_D(\pi_P(x))$ and $B_D(\pi_P(y))$. Let $D_0 = D(1,0)$ be the 
constant associated to geodesics. 
\item \textup{(Deep components, \cite[Lemma 3.3]{sisto-track})}
There exist $D, t, R$ such that for any $x, y \in G$, the set
$\textup{deep}_{D, R}(\llbracket x, y \rrbracket)$ is contained in a disjoint union of subgeodesics of $\llbracket x, y \rrbracket$, each contained in $\calN_{t D}(P)$
for some $P \in \calP$. We call each subgeodesic a \emph{deep component of $\llbracket x, y \rrbracket$ along $P$}. 
Moreover, if $d_P(x, y) \geq L_0$, then $\llbracket x, y \rrbracket$ contains a deep component along $P$. 
\end{enumerate}
\end{proposition}

Let $\go \in G$ be the vertex representing the identity element and consider an infinite geodesic ray  $\gamma$ in  $(G, d_{G})$ starting at $\go$.
For $x \in G$, define $\Norm{x}_{\hG} := d_{\hG}(\go, x)$. Also,  for $P \in \calP$, define $\Norm{P}_{\hG} := d_{\hG}(\go, P)$. 
We will include in our $\kappa$-Morse boundary the geodesic rays which have excursion in each peripheral set bounded by a multiple of $\kappa$ of its $\hG$-norm. 
To be precise, we have the following. 

\begin{definition}
Let $D_{0}$ be given by Proposition~\ref{P:RH-recall} (3). We say that $\gamma$ has \emph{$\kappa$-excursion} with respect to $\calP$ if there exists a constant $E_{\gamma}$ such that, for each $P \in \calP$, we have
\[
\diam_{G}(\gamma \cap \calN_{D_0} (P)) 
  \leq E_{\gamma} \cdot \kappa(\Norm{P}_{\hG}),
\]
where $ \calN_{D_0} (P)$ is the $D_0$-neighborhood of $P$ in $(G, d_G)$. 
That is, the amount of time $\gamma$ stays near $P$ grows sublinearly with 
$\Norm{P}_{\hG}$.
\end{definition}

Our goal is to prove that if $\gamma$ has $\kappa$-excursion, then it 
is $\kappa$-Morse. In fact, we will first show that $\gamma$ is $\kappa$-weakly 
contracting (see \defref{Def:generalContracting}) and then use \thmref{Thm:W-Strong} to conclude that 
$\gamma$ is $\kappa$-Morse. 

Let $\gamma$ be a geodesic ray in $G$, and let us define $\pi_\gamma \from G \to \gamma$ to be a nearest point projection onto $\gamma$
in the metric $d_{\hG}$. 
Note that this is well-defined up to bounded distance in $\hG$, but we will fix one such choice for the remainder of this Section. 

By hyperbolicity of $\hG$, there exist $L_1, R_0$ (depending on $\delta$) such that the following holds
(see e.g. \cite[Proposition 3.4]{Maher}): 
for any $x, y \in G$, if $d_{\hG}(\pi_\gamma(x), \pi_\gamma(y)) \geq L_1$, the geodesic $[x, y]$ in $\hG$ and the broken geodesic 
$\gamma' = [x, \pi_\gamma(x)] \cup [\pi_\gamma(x), \pi_\gamma(y)] \cup [\pi_\gamma(y), y]$ lie within a $R_0$-neighborhood of each other in the metric $d_{\hG}$.
Moreover, any geodesic segment $[x, \overline{x}]$ and $[y, \overline{y}]$, with $[\overline{x}, \overline{y}]$ a subsegment of $[\pi_\gamma(x), \pi_\gamma(y)]$, belongs 
to an $R_0$-neighborhood of $\gamma'$. 

We denote as $\calP_{\gamma, x}$ the set of $P \in \calP$ such that 
$P$ (which has diameter 1 in $\hG$) lies within $d_{\hG}$-distance $R_1 := 1 + 4 R_0$ of $\pi_\gamma(x)$. Now, define 
a projection $\Pi_\gamma \from G \to \gamma$ by 
\[
\Pi_\gamma(x) := \bigcup_{P \in \calP_{\gamma, x}} ( \gamma \cap \calN_{D_0} (P)) 
\]
if $ \calP_{\gamma, x} \neq \emptyset$, and $\Pi_\gamma(x) := \pi_\gamma(x)$ otherwise.
As usual, the image of $\Pi_\gamma$ is a subset of $\gamma$ that could have a large 
diameter. 

\begin{figure}[h!]
\begin{tikzpicture}[scale=1]
\tikzstyle{vertex} =[circle,draw,fill=black,thick, inner sep=0pt,minimum size=.5 mm]
[thick, 
    scale=1,
    vertex/.style={circle,draw,fill=black,thick,
                   inner sep=0pt,minimum size= .5 mm},
                  
      trans/.style={thick,->, shorten >=6pt,shorten <=6pt,>=stealth},
   ]

\draw [thick](-3,0) to (9,0){};

\node at (9,0)[label=right:$\gamma$] {}; 
 \node at (4, -0.25) [label=right:$\tiny D_{0}$] {}; 
  \node at (0, 0.05) [label=above:$\tiny D_{0}$] {}; 
    \node at (-1.9, -0.05) [label=below:$\tiny D_{0}$] {}; 
    \node at (1,1.1)[label=right:$\pi_{\gamma}(x) \text{ in } \hG$] {}; 

\tikzstyle{every node}=[trapezium, draw, thick, minimum width=2cm,
trapezium left angle=120, trapezium right angle=60]

\node[trapezium] (p1)at (-1,0)[label=below:$P_{1}$] {};
    
    \node[trapezium stretches body]  (p2) at (3,-0.5)[label=below:$P_{2}$]{};
    
     \draw [decorate,decoration={brace,amplitude=2pt},xshift=0pt,yshift=0pt]
  (4,0) -- (4,-0.5) {};
  
    \draw [decorate,decoration={brace,amplitude=2.5pt},xshift=0pt,yshift=0pt]
  (-1.7,-0.05)-- (-2.1, -0.05) {};
   \draw [decorate,decoration={brace,amplitude=2pt},xshift=0pt,yshift=0pt]
 (-0.2, 0.05) -- (0.2, 0.05){};
\node[trapezium]at (7,0)  (p3) [label=below:$P_{3}$]{};
\node [vertex] (a) at (0,2) [label=above:$x$]{};
\draw [dashed] (a) to (p1){};
\draw [dashed] (a) to (p2){};
\draw [red, very thick] (-2.1, 0) to (0.2,0){};
\draw [red, very thick] (2,0) to (4,0){};
\end{tikzpicture}
\caption{Definition of $\Pi_\gamma(x)$. Suppose $P_{1}, P_{2}$ and $P_{3}$ are peripheral sets that are within distance $D_{0}$ of $\gamma$ in $G$, and that, out of the three, $P_{1}, P_{2} \in \calP_{\gamma, x}$; then we consider the union of intersections
of $\gamma$ with the $D_{0}$-neighborhoods of $P_{1}$ and $P_{2}$, which are illustrated in red.} 
\end{figure}
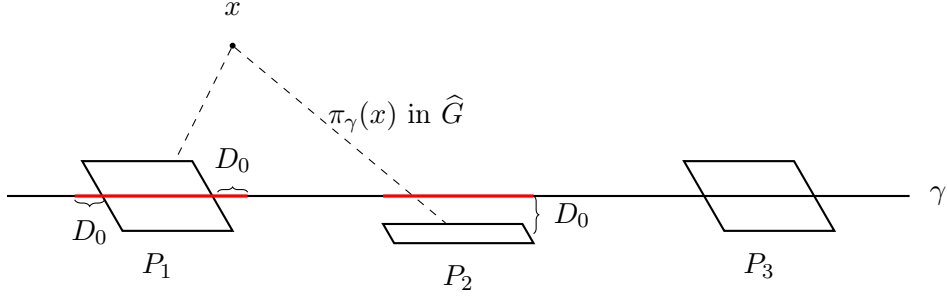

Finally, given $x, y \in G$, we define $c_{\gamma, y}(x)$ to be the point of $\gamma$ closest to $\pi_\gamma(x)$ such that 
the segment $\llbracket \pi_\gamma(x), c_{\gamma, y}(x) \rrbracket$ contains all deep components of all $P$ in $\calP_{\gamma, x}$. 

\begin{lemma}\label{L:P-bound}
There exists $L > 0$ such that the following holds. Let $x, y \in G$, and let $\gamma$ be a geodesic ray based at $\go$, and suppose that $d_{\hG}(\pi_\gamma(x), \pi_\gamma(y)) \geq L$. Let $x_\gamma := c_{\gamma, y}(x)$.
\begin{enumerate}
\item
Then we have
$$d_{\hG}(x, x_\gamma) \leq d_{\hG}(x, y) + L.$$
\item For any $P \in \calP$, if $d_P(x, x_\gamma) \geq L$, we have 
$$d_P(x,x_\gamma) \leq d_P(x, y) +L.$$
\end{enumerate}
\end{lemma}

\begin{proof}
Let $y_\gamma := \pi_\gamma(y)$.

(1) By the choice of $L_1, R_0$, if $d_{\widehat{G}}(\pi_\gamma(x), \pi_\gamma(y)) \geq L_1$, the geodesic $\gamma_1 := [x, y]$ in $\hG$ and the broken geodesic  
\[
\gamma_2 
  := [x, \pi_\gamma(x)] \cup [\pi_\gamma(x), \pi_\gamma(y)] \cup [\pi_\gamma(y), y]
\] 
lie in a $R_0$-neighborhood of each other for the metric $d_{\hG}$. 
Let $p_1, p_2$ be nearest point projections (in $\hG$), respectively, of $x_\gamma$ and $y_\gamma$ onto $[x, y]$.
This implies 
\begin{align}\label{deltathin}
d_{\widehat{G}}(x, y) & \geq d_{\widehat{G}}(x, p_1)  \geq d_{\widehat{G}}(x, x_\gamma) - R_0 
\end{align}
which proves (1), as long as $L \geq R_0$.

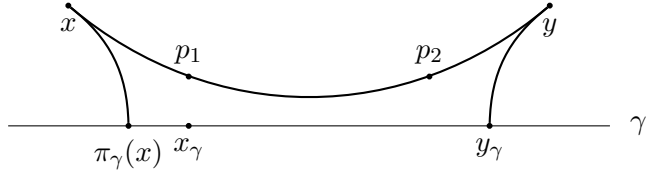
\begin{figure}[h!]
\begin{tikzpicture}[scale=0.8]
\tikzstyle{vertex} =[circle,draw,fill=black,thick, inner sep=0pt,minimum size=.5 mm]
[thick, 
    scale=1,
    vertex/.style={circle,draw,fill=black,thick,
                   inner sep=0pt,minimum size= .5 mm},
                  
      trans/.style={thick,->, shorten >=6pt,shorten <=6pt,>=stealth},
   ]

  \node[vertex] (a) at (0,0)[label=below:$\pi_\gamma(x)$] {}; 
  \node[vertex] (b) at (6,0)[label=below:$y_\gamma$] {}; 
  \node[vertex] (c) at (-1,2)[label=below:$x$] {}; 
  \node[vertex] (d) at (7,2)[label=below:$y$] {}; 
  \node[vertex] (e) at (1,0)[label=below:$x_{\gamma}$] {}; 
%   \node[vertex] (f) at (5,0)[label=below:$y_{\gamma}$] {}; 
  \node[vertex]  at (1, 0.82)[label= above:$p_1$] {}; 
  \node[vertex]  at (5,0.82)[label=above:$p_2$] {};

  \draw [-, thick] (c) to [bend left=25] (a);
  \draw [-, thick] (d) to [bend right=25] (b);
  \draw [-, thick] (c) to [bend right=40] (d);
  \draw (-2, 0) to (8, 0);
  \node at (8,0)[label=right:$\gamma$] {}; 
\end{tikzpicture}
\caption{A thin quadrilateral in $\hG$.}
\end{figure}

\medskip

(2) To prove (2), suppose that $d_P(x, x_\gamma) \geq 2 L_0$ is large, where $L_0$ is given by Proposition~\ref{P:RH-recall} (3). 
First, we claim that $d_P(x_\gamma, y_\gamma) \leq L_0$. By the triangle inequality, 
$$d_P(x, x_\gamma) \leq d_P(x, \pi_\gamma(x)) + d_P(\pi_\gamma(x), x_\gamma)$$
hence there are two cases: either $d_P(x, \pi_\gamma(x)) \geq L_0$ or $d_P(\pi_\gamma(x), x_\gamma) \geq L_0$.

If $d_P(\pi_\gamma(x), x_\gamma) \geq L_0$, then  by Proposition~\ref{P:RH-recall} (4), 
the geodesic segment $\llbracket \pi_\gamma(x), x_\gamma \rrbracket$ in $G$ 
contains a deep component along $P$. 
Then, since $x_\gamma$ is a transition point and deep components are disjoint, the segment $\llbracket x_\gamma, y_\gamma \rrbracket$ in $G$
has no deep component along $P$. 
This implies by Proposition~\ref{P:RH-recall} (4) that $d_P(x_\gamma, y_\gamma) \leq L_0$, as claimed. 

Otherwise, we can assume that $d_P(\pi_\gamma(x), x_\gamma) \leq L_0$ and $d_P(x, \pi_\gamma(x)) \geq L_0$. 
First, by the bounded geodesic image theorem, the geodesic segment $[x, \pi_\gamma(x)]$ in $\hG$ contains an edge in $P$; 
let $p \in P$ be a vertex of this edge. 
Now, by contradiction, suppose that $d_P(x_\gamma, y_\gamma) \geq L_0$; then, by the bounded geodesic image theorem, 
$[x_\gamma, y_\gamma]$ also contains an edge in $P$. Thus, there exists $p' \in [x_\gamma, y_\gamma] \cap P$ with 
$d_{\hG}(p, p') \leq 1$. 
Thus, using that $\gamma_1$ and $\gamma_2$ lie in a $R_0$-neighborhood of each other,
$$d_{\hG}(p, \pi_\gamma(x)) \leq d_{\hG}(p, p') + 4 R_0 \leq 1 +  4 R_0 = R_1.$$
Hence, $P$ belongs to $\calP_{\gamma, x}$, which, as above, implies $d_P(x_\gamma, y_\gamma) \leq L_0$.

We now claim that 
\[ d_{\hG}([x, x_\gamma], [y, y_\gamma]) \geq 2, \] 
which again by Proposition~\ref{P:RH-recall} (3) implies $d_P(y, y_\gamma) \leq L_0$. 
Indeed, let $q_1$ be a point in $[x, x_\gamma]$ and $q_2$ be a point in $[y, y_\gamma]$; by hyperbolicity, $d_{\hG}(q_1, [x, p_1]) \leq R_0$ and also
$d_{\hG}(q_2, [y, p_2]) \leq R_0$ ; 
hence, 
$$d_{\hG}(q_1, q_2) \geq d_{\hG}(p_1, p_2) - 2 R_0 \geq d_{\hG}(\pi_\gamma(x), \pi_\gamma(y)) - 2 R_2 \geq L - 2 R_2 \geq 2$$
where $R_2  = 2R_0 + D_0 + R_1$, provided that we choose $L \geq 2 R_2 + 2$. 

Finally, by the triangle inequality 
\begin{align*}
d_P(x, y) & \geq d_P(x, x_\gamma) - d_P(x_\gamma, y_\gamma) - d_P(y_\gamma, y) \\
& \geq d_P(x, x_\gamma) - 2 L_0.
\end{align*}
Thus, if we choose $L := \max\{ L_1, R_0, 2 R_2 + 2, 2 L_0\}$, both (1) and (2) hold. 
\end{proof}

\begin{proposition} \label{P:bounded-proj}
Let $\gamma$ be a geodesic ray with $\kappa$-excursion. Then, the map 
$\Pi_\gamma$ defined above is a $\kappa$-projection map. 
Furthermore, there exist $D_{1} < 1, D_{2}>1$ such that for any two points 
$x, y \in G$ we have
\[
d_{G}(x, y) \leq D_{1} \cdot d_{G}(x, \gamma) \qquad \Longrightarrow \qquad d_{\widehat{G}}(\pi_{\gamma}(x), \pi_{\gamma}(y)) \leq D_{2}.
\] 
\end{proposition}

\begin{proof}
We now fix $L$ as given by Lemma \ref{L:P-bound}, and we start by contradiction, by assuming that $d_{\widehat{G}}(\pi_\gamma(x), \pi_\gamma(y)) \geq L$.
By Lemma \ref{L:P-bound}, $d_{\hG}(x, x_\gamma) \prec d_{\hG}(x,y)$, and $d_P(x, x_\gamma) \prec d_P(x, y)$ 
whenever $d_P(x, x_\gamma)$ is large enough.
Now, applying the Distance formula (Proposition~\ref{P:RH-recall} (2)) to the pair of points $(x, y)$ we have:
\begin{align*}
d_{G}(x, y) &\asymp \sum_{P \in \calP} \lfloor d_{P}(x, y) \rfloor_{L} + d_{\widehat{G}}(x, y)\\
	& \succ \sum_{P \in \calP} \lfloor d_{P}(x, x_\gamma) \rfloor_{L} + d_{\widehat{G}}(x, x_\gamma) - O(\delta) \qquad \textup{ by Eq.}~\eqref{deltathin} \\
	& \asymp d_{G}(x, x_\gamma). 
\end{align*}
That is to say, there exists $D_{1} = D_{1}(L, \delta) $ such that 
\[
d_{G}(x, y) \geq D_{1} \cdot d_{G}(x, x_\gamma),
\]
which is a contradiction since $x_\gamma \in \gamma$. Therefore, setting $D_2 = L$ yields
\begin{equation*}
d_{G}(x, y) \leq D_{1} \cdot d_{G}(x, \gamma) \qquad  \Longrightarrow \qquad d_{\widehat{G}}((\pi_\gamma(x), \pi_\gamma(y)) \leq D_{2}.
\qedhere
\end{equation*}

\end{proof}

We now show that every $\kappa$-excursion geodesic ray is $\kappa$-weakly contracting. 

\begin{proposition}\label{excursionimpliescontracting}
Every $\kappa$-excursion geodesic ray $\gamma \in (G, d_{G})$ is $\kappa$-weakly contracting.  That is to say, there exists $D_{1} < 1, D_{2}>1$ such that for any two points $x, y \in G$ we have
\[
d_{G}(x, y) \leq D_{1} \cdot d_{G}(x, \gamma) \qquad  \Longrightarrow \qquad \diam_{G}(\Pi_{\gamma}(x) \cup \Pi_{\gamma}(y)) \leq D_{2} \cdot \kappa(x).
\] 
As a consequence, every $\kappa$-excursion geodesic ray is $\kappa$-Morse.
\end{proposition}

\begin{proof}
By Proposition \ref{P:bounded-proj},  $d_{\widehat{G}}(x_\gamma, y_\gamma) \leq D_2$, so there are only boundedly many $P$ 
intersecting $[x_\gamma, y_\gamma]$ in $\widehat{G}$; let $\calP_0$ denote the set of such $P$. By definition of $\kappa$-excursion, for each $P \in \calP_0$, $d_P(x_\gamma, y_\gamma)$ is bounded above by 
$\kappa(\Vert P \Vert_{\hG})$. 

We claim that for all $P \in \calP_0$ we have $\Vert P \Vert_{\hG} \prec \Vert x \Vert$, hence also $\kappa(\Vert P \Vert_{\hG}) \prec \kappa(x)$.
Indeed, since $P$ intersects $[x_\gamma, y_\gamma]$ and $d_{\widehat{G}}(x_\gamma, y_\gamma)$ is bounded, 
\[
\Vert P \Vert_{\hG} = d_{\widehat{G}}(\go, P) \prec d_{\widehat{G}}(\go, x_\gamma);
\] 
then, since nearest point projection in the $\delta$-hyperbolic space $\hG$ is coarsely 
distance decreasing, 
$d_{\widehat{G}}(\go, x_\gamma) \prec d_{\widehat{G}}(\go, x)$, and finally
\[
d_{\widehat{G}}(\go, x) \prec d_G(\go, x)
\]
since the inclusion $G \to \hG$ is Lipschitz. 

Thus, the claim together with the previous estimates and the distance formula yields
$$d_G(x_\gamma, y_\gamma) \asymp \sum_{P \in \calP} \lfloor d_P(x_\gamma, y_\gamma)  \rfloor_L \prec \sum_{P \in \calP_0} \kappa( \Vert P \Vert_{\hG})  \prec D_2 \cdot  \kappa(x).$$

Finally, by Theorem~\ref{Thm:W-Strong}, a $\kappa$-weakly contracting geodesic ray is $\kappa$-Morse.
\end{proof}

We now show that the map $\Pi_\gamma$ is a $\kappa$-projection.

\begin{proposition} 
Let $\gamma$ be a geodesic ray with $\kappa$-excursion. Then, the map 
$\Pi_\gamma$ defined above is a $\kappa$-projection map. 
\end{proposition}

\begin{proof}
Let $x \in X$ and $z \in \gamma$,  and let $x_\gamma := c_{\gamma, z}(x)$. Then, if $d_{\hG}(x_\gamma, z) \geq L$, we have, 
as in the proof of Lemma \ref{L:P-bound}, that $d_{\hG}(x_\gamma, z) \prec d_{\hG}(x, z)$ and $d_{P}(x_\gamma, z) \prec d_{P}(x, z)$ for any $P \in \calP$, hence 
by the Distance formula  (Proposition~\ref{P:RH-recall} (2))
$$d_G(x_\gamma, z) \prec d_G(x, z).$$
On the other hand, if $d_{\hG}(x_\gamma, z) \leq L$, then, as in the proof of Proposition \ref{excursionimpliescontracting},
$$d_G(x_\gamma, z) \prec \kappa(\Norm{x_\gamma}_{\hG}) \prec \kappa(\Norm{x}_{\hG}).$$
Moreover, since $\gamma$ has $\kappa$-excursion, we have $\textup{diam}_G(\Pi_\gamma(x)) \prec E_\gamma \cdot \kappa(\Norm{x_\gamma}_{\hG}) \prec E_\gamma \cdot \kappa(\Norm{x}_{\hG})$.
Thus, we obtain 
$$\textup{diam}_G(\Pi_\gamma(x) \cup \{z\}) \leq C_1 \cdot d_G(x, z) + C_2 \cdot \kappa(x),$$
where $C_1, C_2$ depend only on $\gamma$, completing the proof.
\end{proof}

The following is our main result on relatively hyperbolic groups. 

\begin{theorem} \label{T:RH}
Let $\mu$ be a finitely supported probability measure on a relatively hyperbolic group $G$, 
and let $S$ be a finite generating set. 
Let $\kappa(r) := \log r$, and let $\pka X$ be the $\kappa$-Morse boundary of the Cayley graph $X$ of $G$ with respect to $S$. 
Then: 
\begin{enumerate}
\item Almost every sample path $(w_n)$ converges to a point in $\pka X$; 
\item The pair $(\pka X, \nu)$, where $\nu$ is the hitting measure of the random walk on $\pka X$, is a model for the Poisson boundary of $(G, \mu)$. 
\end{enumerate}
\end{theorem}

\begin{proof}
The proof is fairly similar to the proof of Theorem \ref{T:converge} for the mapping class group, 
after replacing subsurfaces $Y$ by peripherals $P \in \calP$, $d_Y$ by $d_P$, and the curve complex $\calC(S)$ by $\hG$.
Note, however, that one difference is that there need not be a hyperbolic space analogous to $\calC(Y)$ for each $P \in \calP$; 
moreover, the concept of \emph{center} is not well-defined.
Essentially, the only step where the proof as written does not immediately generalize is the proof of Eq. \eqref{E:C3}, where we do \emph{not} know that nearest point projections are distance decreasing (not even coarsely).
We shall give an alternative proof of this point, not using the hyperbolicity of $\calC(Y)$.

By Theorem \ref{T:RW}, almost every sample path $\omega = (w_n)$ converges to 
a point $\xi_\omega$ in the Gromov boundary of  $\widehat{G}$. 
Consider a geodesic ray in $\hG$ joining the base point $\go$ and $\xi_\omega$, 
and let $\gamma = \gamma_\omega$ be a lift to $G$.

Let $c_n := c_{\gamma, w_{2n}}(w_n)$, following the notation before Lemma \ref{L:P-bound}. 
We replace Step 1 in the proof of Theorem \ref{T:converge} with the following.

\medskip
\textbf{Step 1.}
We claim that there exists $C > 0$ such that 
\begin{equation} \label{E:supP}
\mathbb{P}(\sup_P d_P(w_n, c_n) \geq  C \log n) \leq C n^{-2}
\end{equation}
for all $n$.

\begin{proof}
Since the drift of the random walk is positive with exponential decay (Theorem \ref{T:RW} (2)), we have by the Markov property
that there exists $0 < C_0 < 1$ such that
\begin{equation} \label{E:C0-RH}
\mathbb{P}(d_{\hG}(w_n, w_{2n}) \leq \ell  n) = \mathbb{P}(d_{\hG}(\go, w_{n} ) \leq \ell  n) \leq (C_0)^n\qquad \forall n
\end{equation}
where $\ell > 0$ is the drift of the random walk.
By Theorem \ref{T:RW} (3), and recalling that $c_n$ projects close to, there exists $C_1 > 0$ so that 
\begin{equation} \label{E:C1-RH}
\mathbb{P}( d_{\hG}(w_n, c_n) \geq C_1 \log n) \leq C_1 n^{-2} \qquad \forall n.
\end{equation}
If a sample path lies in the complement of the union of the events expressed by \eqref{E:C0-RH}, \eqref{E:C1-RH}, we have 
\begin{align*}
d_{\hG}(c_n, c_{2n}) &  \geq d_{\hG}(w_n, w_{2n}) - d_{\hG}(w_n, c_n) - d_{\hG}(w_{2n}, c_{2n}) \\
& \geq \ell n - C_1 \log n - C_2 \log (2n) \geq L
\end{align*}
where $L$ is given by Lemma \ref{L:P-bound}, hence, by Lemma \ref{L:P-bound}, we have 
\begin{equation} \label{E:small-proj}
d_P(w_n, c_n) \leq d_P(w_n, w_{2n}) +L
\end{equation}
for any $P \in \calP$. 
Moreover, by \cite[Lemma 4.4]{sisto-track}, there exists $C_2 > 0$ for which 
\begin{equation} \label{E:C2-RH}
\mathbb{P}(\sup_P d_P(w_n, w_{2n}) \geq C_2 \log n) = \mathbb{P}(\sup_P d_Y(\go, w_{n})  \geq C_2 \log n) \leq C_2 n^{-2} \qquad \forall n.
\end{equation}
hence, by combining Eq. \eqref{E:small-proj} and \eqref{E:C2-RH} we obtain \eqref{E:supP}. 
\end{proof}
Then, we proceed exactly as in Theorem \ref{T:converge} (Steps 2 and 3), 
proving that for almost every $\omega \in \Omega$, there is a constant $c$ such that  
\begin{equation} \label{E:supY-RH}
\sup_{P \in \mathcal{P}} \diam_G(\gamma \cap \calN_{D_0}(P)) \leq c \log d_{\hG}(\go, P).
\end{equation}

Hence, by Proposition \ref{excursionimpliescontracting}, the geodesic ray $\gamma_\omega$ has 
$\kappa$-excursion, hence it is $\kappa$-Morse. This shows (1). 

Finally, (2) follows by Theorem \ref{T:ray} using Theorem \ref{T:poiss-general}. 
\end{proof}

\begin{remark}
Continuing as in the proof of Theorem \ref{T:converge}, the above argument also shows the following tracking result:
for almost every sample path there exists a geodesic ray $\gamma$ in $G$ such that 
$$\limsup_{n \to \infty} \frac{d(w_n, \gamma)}{\log^2(n)} < + \infty.$$
However, \cite{sisto-track} already shows the stronger tracking result with $\log(n)$ instead of $\log^2(n)$, hence we do not 
write out the details.
\end{remark}

%%%%%%%%%%%%%%%%%%%%%%%%%%%%%%%%%%%%%%%%%%%%%%%%%%%%%%%%
\begin{appendix}

\section{General projection and Weakly $\kappa$-Contracting Property}

We begin this appendix by proving the following, announced in Section \ref{S:weakly-con}.

\begin{theorem}[$\kappa$-weakly contracting implies sublinearly Morse] \label{Thm:W-Strong-app}
Let $\kappa$ be a concave sublinear function and let $Z$ be a closed subspace of $X$. 
Let $\pi_Z$ be a $\kappa$-projection onto $Z$ and suppose that $Z$ is $\kappa$-weakly contracting with respect to $\pi_Z$.
Then, there is a function $m_Z\from \RR^2 \to \RR$ such that, for every constant $r>0$ and every
sublinear function $\kappa'$, there is an $R= R(Z, r, \kappa')>0$ where the 
following holds: Let $\eta \from [0, \infty) \to X$ be a $(q, Q)$-quasi-geodesic ray 
so that $m_Z(q, Q)$ is small compared to $r$, let $t_r$ be the first time 
$\Norm{\eta(t_r)} = r$ and let $t_R$ be the first time $\Norm{\eta(t_R)} = R$. Then
\[
d_X\big(\eta(t_R), Z\big) \leq \kappa'(R)
\quad\Longrightarrow\quad
\eta([0, t_r]) \subset \calN_{\kappa}\big(Z, m_Z(q, Q)\big). 
\]
\end{theorem} 

\begin{proof}
Let $C_{1}, C_{2}, D_{1}, D_{2}$ be the constants which appear in the definitions of  $\kappa$-projection and $\kappa$-weakly contracting (Definition~\ref{weakprojection} and Definition~\ref{Def:generalContracting}). 
Note that the condition of being $\kappa$-weakly contracting becomes weaker as $C_1$ gets smaller, hence we can assume that $C_1 \leq 1/2$.
We first set 
\begin{equation} \label{Eq:Conditions}
m_0 := \max \left \{ \frac{q(q C_{2} + q +1) + Q}{C_{1}}, \frac{2 C_{2}(D_{1}+ 1)}{(q-1)},  Q \right\},
\qquad 
m_1 := q( C_{2}+1 )(D_{1} + 1). 
\end{equation} 
\begin{claim}\label{claimA5}
Consider a time interval $[s,s']$ during which $\eta$ is outside of 
$\calN_{\kappa}(Z, m_0)$. Then there exists a constant $\mathfrak{A}$ depending only on $\{C_{1}, C_{2}, D_{1}, D_{2}, q, Q \}$, such that 
\begin{equation}  \label{Eq:End-Point} 
|s'-s| \leq m_1 \big( d_X\big(\eta(s), Z\big)+ d_X\big(\eta(s'), Z\big) \big) + \mathfrak{A} \cdot \kappa(\eta(s')).
\end{equation} 
\end{claim} 

\begin{proof}[Proof of  Claim~\ref{claimA5}] \renewcommand{\qedsymbol}{$\blacksquare$}
Let 
\[
s = t_{0} < t_{1} < t_{2}< \dots < t_{\ell} = s'
\]
be a sequence of times such that, for $i=0, \dots, {\ell-2}$, we have $t_{i+1}$ is a first time after $t_i$ where 
\begin{equation} \label{E:def-eta}
d_X\big(\eta(t_i), \eta(t_{i+1}) \big) = C_{1} d_X(\eta(t_i), Z)
\quad\text{and}\quad 
d_X\big(\eta(t_{\ell-1}), \eta(t_\ell) \big) \leq C_{1} d_X(\eta(t_{\ell-1}), Z).
\end{equation}
To simplify the notation, we define
\[
\eta_i := \eta(t_i), \qquad r_i := \Norm{\eta(t_i)}
\]
and moreover, we pick some $\pi_i \in \pi_{Z}(\eta_i)$ and let 
\[
d^{\pi}_i := d_X(\eta_i, \pi_{i}), \qquad d_i := d_X(\eta_i, Z).
\]
Note that, by assumption
\begin{equation}\label{assumption1}
d^{\pi}_i \geq d_i = d_{X}(\eta_{i}, Z) \geq m_0 \cdot \kappa(r_i ).
\end{equation}

\begin{claim}\label{claimA6} We have the inequality
$d^{\pi}_{\ell-1} \leq 2( D_{1}+ 1) \,d^{\pi}_{\ell} +D_{2}\cdot \kappa(\eta_{\ell-1}) $.
\end{claim}
\begin{proof} [Proof of Claim~\ref{claimA6}]
By the triangle inequality and Eq. \eqref{E:def-eta}, 
\begin{align*}
d_X(\eta_{\ell-1}, Z) & \leq d_X(\eta_{\ell-1}, \eta_{\ell}) + d_X(\eta_\ell, Z)  \\
			       & \leq C_1 d_X( \eta_{\ell-1}, Z) + d_X(\eta_\ell, Z)
\end{align*}
hence, using $C_1 \leq 1/2$, 
\begin{align*}
d_X(\eta_{\ell-1}, Z) & \leq \frac{1}{1-C_1} d_X(\eta_\ell, Z) \leq 2 d_X(\eta_\ell, Z). 
\end{align*}
Thus, by Lemma \ref{projection-property}, 
\begin{align*}
d^{\pi}_{\ell-1} &\leq (D_{1} + 1) d_X(\eta_{\ell-1}, Z) + D_{2}\cdot \kappa(\eta_{\ell-1})\\
\intertext{and by the above equation}
                       & \leq  2 (D_{1} + 1) d_X(\eta_{\ell}, Z)+D_{2}\cdot \kappa(\eta_{\ell-1}) \\
                       &\leq 2 (D_{1} + 1) d^{\pi}_{\ell}+D_{2}\cdot \kappa(\eta_{\ell-1}).
                       \qedhere
\end{align*}
\end{proof}

Now, since $Z$ is $\kappa$-weakly contracting, by Definition \ref{Def:generalContracting} we get
\[
d_X\big( \pi_0 , \pi_\ell \big) \leq 
\sum_{i=0}^{\ell-1} d_X \big( \pi_i  , \pi_{i+1}  \big) 
\leq \sum_{i=0}^{\ell-1} C_{2} \cdot \kappa(r_i). 
\]
But $\eta$ is $(q, Q)$-quasi-geodesic, hence, 
\begin{align} 
|s'-s| & \leq q \, d_X(\eta_0, \eta_\ell) + Q \notag \\
  & \leq q \left( d^{\pi}_0 + d_X\big( \pi_0 , \pi_\ell \big)  + d^{\pi}_\ell \right) + Q 
  \label{Eq:Upper} \\
  & \leq  q \, C_{2}  \left( \sum_{i=0}^{\ell-1}  \kappa(r_i) \right) +  
    q \, (d^{\pi}_0 + d^{\pi}_\ell) + Q.  \notag
\end{align}
On the other hand, 
\[
|s'-s| = \sum_{i=0}^{\ell-1} |t_{i+1}- t_i| \ge 
\frac{1}{q} \sum_{i=0}^{\ell-1} \left(   d_X(\eta_i, \eta_{i+1}) -Q \right). 
\]
Meanwhile, for $i=0, \dots, \ell-2$ we have $d_X(\eta_i, \eta_{i+1}) =C_{1} d_X(\eta_{i}, Z)$. 
Furthermore, we have by triangle inequality,
\[
d_X(\eta_{\ell-1}, \eta_\ell) + d^{\pi}_{\ell} + d_X(\pi_{\ell-1}, \pi_{\ell}) \geq  d^{\pi}_{\ell-1} \geq d_X(\eta_{\ell-1}, Z), 
\]
which gives 
\[
d_X(\eta_{\ell-1}, \eta_\ell) \geq d_X(\eta_{\ell-1}, Z) - d^{\pi}_{\ell} - C_{2}\cdot \kappa(r_{\ell-1}).
\]
Hence, together with Equation~\eqref{assumption1} and using $C_1 \leq 1$ we have  
\begin{align*}
|s'-s| 
 & \geq \frac{1}{q} \sum_{i=0}^{\ell-1} \left( C_{1} d_X(\eta_{i}, Z) -Q \right) 
   -  \frac{d^{\pi}_\ell + C_{2}\cdot \kappa(r_{\ell-1}) }{q} \\
& \geq  \frac{1}{q}  \sum_{i=0}^{\ell-1} \left( C_{1} m_0 \cdot \kappa(r_i) -Q \right) 
   -  \frac{d^{\pi}_\ell}{q}-C_{2} \frac{\kappa(r_{\ell-1}) }{q}  &  \text{By Equation~\eqref{assumption1}}\\
   & \geq \frac{1}{q} \sum_{i=0}^{\ell-1} \left( C_{1} m_0 \cdot \kappa(r_i) -Q \cdot \kappa(r_i)  \right) 
   -  \frac{d^{\pi}_\ell}{q}-C_{2} \frac{\kappa(r_{\ell-1}) }{q} & \text{$\kappa(t) \geq 1$ for all $t$} \\
    & \geq \left(  \frac{C_{1} m_0 - Q}{q} \right) \sum_{i=0}^{\ell-1}  \kappa(r_i)  
   -  \frac{d^{\pi}_\ell}{q}-C_{2} \frac{\kappa(r_{\ell-1}) }{q}. 
\end{align*}
Combining the above inequality with \eqnref{Eq:Upper} we get 
\begin{align}\label{Morseequation}
q \, (d^{\pi}_0 + d^{\pi}_\ell) + Q + \frac{d^{\pi}_\ell}{q} + C_{2} \frac{\kappa(r_{\ell-1})}{q} & \geq  
  \left(  \frac{C_{1} m_0 - Q}{q}  - q \, C_{2}   \right)  \sum_{i=0}^{\ell-1} \kappa(r_i) \\
  &  \geq  (q+ 1)  \sum_{i=0}^{\ell-1} \kappa(r_i) \notag,
\end{align}
where in the last step we plugged in the definition of $m_{0}$ from \eqref{Eq:Conditions}.

By \eqref{assumption1} we also have
\begin{align*}
q \, (d^{\pi}_0 + d^{\pi}_\ell) + Q + \frac{d^{\pi}_\ell}{q} + C_{2} \frac{\kappa(r_{\ell-1})}{q}  &\leq q \, (d^{\pi}_0 + d^{\pi}_\ell) + Q + \frac{d^{\pi}_\ell}{q} + C_{2} \frac{d^{\pi}_{\ell-1}}{m_{0}\,q} \\
\intertext{By the expression of $m_{0}$ in \eqref{Eq:Conditions}, $Q \leq m_{0}$ and by Equation~\eqref{assumption1}, $m_{0} \leq d^{\pi}_{0}$. Thus we have $Q \leq d_{0}^{\pi}$ and again by plugging in  $m_{0}$  and using Claim \ref{claimA6}, we obtain}
 q \, (d^{\pi}_0 + d^{\pi}_\ell) + Q + \frac{d^{\pi}_\ell}{q} + C_{2} \frac{d^{\pi}_{\ell-1}}{m_{0}\,q}& \leq(q + 1)(d^{\pi}_0 + d^{\pi}_\ell) + \frac{C_{2} D_{2}}{m_{0}\, q}\cdot \kappa(\eta_{\ell-1}).
\end{align*}
Plugging this inequality into Equation~\eqref{Morseequation}, we get
\begin{align*}
\sum_{i=0}^{\ell-1} \kappa(r_i) &\leq d^{\pi}_0 + d^{\pi}_\ell + \frac{C_{2} D_{2}}{m_{0}\, q (q+1)}\cdot \kappa(\eta_{\ell-1}) \\
                                                   &\leq (D_{1} + 1)( d_0 + d_\ell ) + D_{2} (\kappa(\eta_{0}) + \kappa(\eta_{\ell}) ) +\frac{C_{2} D_{2}}{m_{0}\, q (q+1)}\cdot \kappa(\eta_{\ell-1})
                                                   \end{align*}
where we recall $d_{i} = d_{X}(\eta_{i}, Z)$, and the last inequality comes from Lemma~\ref{projection-property} (note the difference between $d_i$ and $d_i^{\pi} = d_X(\eta_i, \pi_i)$). 

Lastly, we claim that since $\eta$ is a quasi-geodesic ray, there is a constant $C_3$, related to $q, Q$, such that 
$\eta_{i} \leq C_3 \cdot \eta(s') + C_3$ for $i = 0, \dots, \ell$, hence also $\kappa(\eta_{i}) \leq 2 C_3 \cdot \kappa(\eta(s')) + \kappa(2 C_3) $; thus, to shorten the preceding expression,
let $\mathfrak{A}$ be a constant, depending on $\{ C_1, C_{2}, D_{1}, D_{2}, q, Q,  \kappa \}$, such that
\[
q (C_2+ 1) D_{2} (\kappa(\eta_{0}) + \kappa(\eta_{\ell}) ) + Q + \frac{ C_{2}^2 D_{2}}{m_{0} (q+1)}\cdot \kappa(\eta_{\ell-1}) \leq \mathfrak{A} \cdot \kappa(\eta(s')).
\]
By \eqnref{Eq:Upper} and the definition $m_1= q( C_{2}+1 )(D_{1} + 1)$ from \eqref{Eq:Conditions},
\[
 |s'-s| \leq m_1 (d_0 + d_\ell) + \mathfrak{A} \cdot \kappa(\eta(s')). 
\]
This proves Claim~\ref{claimA5}.
\end{proof} 

\begin{figure}[h!]
\begin{tikzpicture}[scale=0.7]
 \tikzstyle{vertex} =[circle,draw,fill=black,thick, inner sep=0pt,minimum size=.5 mm]
[thick, 
    scale=1,
    vertex/.style={circle,draw,fill=black,thick,
                   inner sep=0pt,minimum size= .5 mm},
                  
      trans/.style={thick,->, shorten >=6pt,shorten <=6pt,>=stealth},
   ]

  \node[vertex] (o) at (0,0)[label=left:$\go$] {}; 
  \node(a) at (13,0)[label=right:$Z$] {}; 
  
  \draw (o)--(a){};
  \draw [dashed] (0, 0.3) to [bend left = 8] (12,1.5){};
  \node at (8.5,1.3){$m_{0} \cdot \kappa(R)$};
  \draw [dashed] (0, 0.6) to [bend left = 10] (5,2){};
  \node at (3.2,2.2){\small $m_{Z}(q, Q) \cdot \kappa(r)$};
     
 \draw [dashed] (12, 4.5) to (12,0) {};
 \node at (12,0)[label=below:$R$] {};
 \draw [dashed] (5, 4.5) to (5,0){};
 \node at (5,0)[label=below:$r$] {};
        
 \draw [decorate,decoration={brace,amplitude=10pt},xshift=0pt,yshift=0pt]
  (12,3.7) -- (12,0)  node [thick, black,midway,xshift=0pt,yshift=0pt] {};       
 \node at (13.3,1.7) {$\kappa'(R)$};
        
 \node[vertex] at (5.4,1.35)[label=below right:$t_{\rm last}$] {}; 
 \node[vertex] at (1.65, 0.68)[label=below right:$s$] {}; 
 \node[vertex] at (4.34,1.18)[label=above right:$s'$] {}; 
 
  \pgfsetlinewidth{1pt}
  \pgfsetplottension{.75}
  \pgfplothandlercurveto
  \pgfplotstreamstart
  \pgfplotstreampoint{\pgfpoint{0cm}{0cm}}  
  \pgfplotstreampoint{\pgfpoint{1cm}{-.6cm}}   
  \pgfplotstreampoint{\pgfpoint{2cm}{1.3cm}}
  \pgfplotstreampoint{\pgfpoint{3cm}{1.0cm}}
  \pgfplotstreampoint{\pgfpoint{4cm}{1.5cm}}
  \pgfplotstreampoint{\pgfpoint{5cm}{0.5cm}}
  \pgfplotstreampoint{\pgfpoint{6cm}{3cm}}
  \pgfplotstreampoint{\pgfpoint{7cm}{2.2cm}}
  \pgfplotstreampoint{\pgfpoint{8cm}{3.6cm}}
  \pgfplotstreampoint{\pgfpoint{9cm}{3.1cm}}
  \pgfplotstreampoint{\pgfpoint{10cm}{3.6cm}}
  \pgfplotstreampoint{\pgfpoint{11cm}{3.1cm}}
  \pgfplotstreampoint{\pgfpoint{12cm}{3.7cm}}
  \pgfplotstreamend
  \pgfusepath{stroke} 
  \end{tikzpicture}
  
\label{Fig:Strong2} 
\caption{The proof of Theorem \ref{Thm:W-Strong-app}.}
\end{figure}
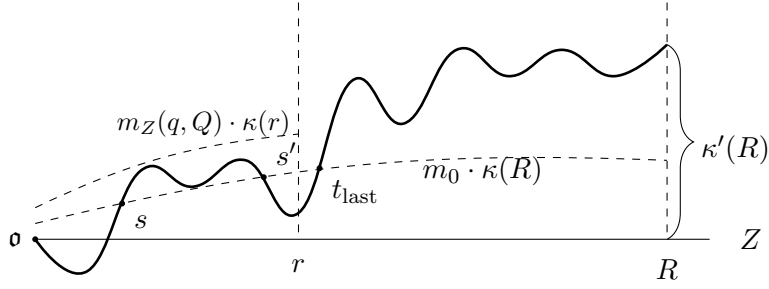

Now let $t_{\rm last}$ be the last time $\eta$ is in $\calN_{\kappa}(Z, m_0)$ and 
consider the quasi-geodesic path $\eta([t_{\rm last}, t_R])$. Since this path is 
outside of $\calN_\kappa(Z, m_0)$, we can use \eqnref{Eq:End-Point} to get 
\[
| t_R-t_{\rm last} | \leq m_1 \big( d_X(\eta(t_{\rm last}), Z)+ d_X(\eta(t_R), Z) \big) + \mathfrak{A} \cdot \kappa(R). 
\]
But 
\begin{align*}
d_X(\eta(t_{\rm last}), Z) &\leq m_0 \cdot \kappa(\eta(t_{\rm last})) & \text{ by the choice of $t_{\rm last}$}\\
  &\leq m_0 \cdot \kappa(R) & \text{since $\kappa$ is monotone}
\end{align*}
and we have by assumption $d_X(\eta(t_R), Z) \leq  \kappa'(R)$. Therefore,
\[
|t_R-t_{\rm last} | \leq m_0 \, m_1\cdot \kappa(R)+  m_1 \cdot  \kappa'(R)+ \mathfrak{A} \cdot \kappa(R).
\]
Since $\eta$ is $(q, Q)$-quasi-geodesic, we obtain $R = d_X(\eta(0), \eta(t_{R})) \leq q t_{R} + Q$, hence
$$t_{R} \geq \frac{R - Q}{q}.$$
Since $m_0$ and $m_1$ are given and $\kappa$ and $\kappa'$ are 
sublinear, there is a value of $R$ depending on $m_0$, $m_1$,  $r$, $\mathfrak{A}$,
$\kappa$ and $\kappa'$ such that 
\[
 m_0 \cdot m_1 \cdot \kappa(R)+ m_1 \cdot \kappa'(R) + \mathfrak{A} \cdot \kappa(R) \leq  
 \frac{R - Q}{q} - r.
\]
For any such $R$, we then have 
\[  t_{\rm last} \geq t_{R} - 
 \frac{R - Q}{q} + r \geq r. \] 
We show that $\eta([0, t_{\rm last}])$ stays in a larger $\kappa$-neighborhood
of $Z$. Consider any other subinterval $[s,s'] \subset [0, t_{\rm last}]$ where 
$\eta$ exits $\calN_\kappa(Z, m_0)$. By taking $[s,s']$ as large as possible, 
we can assume $\eta(s), \eta(s') \in \calN_\kappa(Z, m_0)$. 
In this case, 
\[
d_X(\eta(s), Z) \leq m_0 \cdot \kappa(\eta(s)) 
\qquad\text{and}\qquad
d_X(\eta(s'), Z) \leq m_0 \cdot \kappa(\eta(s')).  
\]
Again applying \eqnref{Eq:End-Point}, we get 
\begin{equation} \label{E:s-s}
|s'-s| \leq  m_0 \, m_1 \cdot \big(\kappa(\eta(s)) + \kappa(\eta(s'))\big) + \mathfrak{A} \cdot \kappa(\eta(s'))
\end{equation}
and thus 
\begin{align*}
d_X(\eta(s'), \eta(s)) & \leq q \, m_0 \, m_1 \cdot \big(\kappa(\eta(s)) + \kappa(\eta(s'))\big) + q\, \mathfrak{A} \cdot \kappa(\eta(s')) + Q \\
& \leq ( 2 q \, m_0 \, m_1 + q \,  \mathfrak{A} + Q )  \cdot \max \big(\kappa(\eta(s)), \kappa(\eta(s'))\big). 
\end{align*}
Applying the Sublinear Estimation Lemma (\cite[Lemma 3.2]{QRT19}), we obtain 
\[
\kappa(\eta(s')) \leq  m_2 \cdot \kappa(\eta(s))
\]
for some $m_2$ depending on $q$, $Q$ and $\kappa$. 
Therefore, by plugging this inequality back into \eqref{E:s-s}, we have for any $t \in [s, s']$
\begin{equation} \label{Eq:t-s}
|t-s| \leq \left( m_0 \, m_1 (1+ m_2)  + \mathfrak{A} \, m_2 \right) \cdot \kappa(\eta(s)) = m_3 \cdot \kappa(\eta(s)).
\end{equation}
with $m_3 = \left( m_0 \, m_1 (1+ m_2)  + \mathfrak{A} \, m_2 \right)$. As before, this implies, 
\[
d_X(\eta(t), \eta(s)) 
\leq q \, m_3 \cdot \kappa(\eta(s)) +  Q \leq ( q \, m_3 + Q) \cdot \kappa(\eta(s)).  
\]
Applying \cite[Lemma 3.2]{QRT19} again, we have 
\begin{equation} \label{Eq:m3}
\kappa(\eta(s)) \leq m_4 \cdot \kappa(\eta(t)),
\end{equation}
for some $m_4$ depending on $q$, $Q$ and $\kappa$. 

Now, for any $t \in [s, s']$ we have
\begin{align*}
d_X(\eta(t), Z)& \leq d_X(\eta(t), \eta(s)) + r_0 \\
& \leq q \, |t-s| + Q + m_0 \cdot \kappa(\eta(s))\\
& \leq  \left( q \, m_3  + Q + m_0\right) \cdot \kappa(\eta(s))
  \tag{\eqnref{Eq:t-s}}\\
&  \leq   \left( q \, m_3  + Q + m_0\right) \, m_4 \cdot \kappa(\eta(t)).  
  \tag{\eqnref{Eq:m3}}
\end{align*}
Now setting 
\begin{equation}\label{mz}
m_Z(q, Q) =  \left( q \, m_3  + Q + m_0\right) \, m_4 
\end{equation}
we have the inclusion
\[
\eta([s, s']) \subset \calN_{\kappa}\big(Z, m_Z(q, Q)\big)
\qquad\text{and hence}\qquad 
\eta([0, t_{\rm last}]) \subset \calN_{\kappa}\big(Z, m_Z(q, Q)\big).
\]
The $R$ we have chosen depends on the value of $q$ and $Q$. However, 
the assumption that $m_Z(q, Q)$ is small compared to $r$ (see \eqnref{Eq:Small}) 
gives an upper bound for the values of $q$ and $Q$. Hence, we can choose $R$ to be the radius associated to the largest possible value 
for $q$ and the largest possible value for $Q$. This finishes the proof. 

Note that, the assumption that $m_Z(q, Q)$ is small compared to $r$ is not 
really needed here and any upper bound on the values of $q$ and $Q$ would
suffice. But this is the assumption we will have later on and hence it is natural to state the
theorem this way. 
\end{proof}

Next, we show that being $\kappa$-Morse implies being $\kappa'$-weakly contracting, with respect to $\kappa$-projection maps, for a sublinear function $\kappa'$. Note that $\kappa'$ is not assumed to be the same function as $\kappa$. 
This is parallel to Theorem 3.8 in \cite{QRT19}, where we show that in a CAT(0) space being $\kappa$-contracting is equivalent to being $\kappa$-Morse  (with the same $\kappa$); an identical statement cannot hold in general for proper geodesic spaces, as evidenced by the following example. 

\begin{example}
We give here a folklore example of a geodesic ray in proper metric space which is $1$-Morse but not $1$-contracting
(see also the related \cite[Example 3.4]{ACGH}). The points $x_{i}$ form loops with the geodesic ray $\gamma$ such that each path going through $x_{i}$ represents a detour that is locally
an $(i, 0)$ quasi-geodesic segment. This geodesic has the following properties:
\begin{itemize}
\item $\gamma$ is 1-Morse, as any $(q, Q)$-quasi-geodesic ray must lie in a $q$-neighborhood of $\gamma$ (in particular, it only 
goes through finitely many loops);
\item $\gamma$ is not 1-contracting, but it is $\sqrt{t}$-contracting, as $\Norm{x_{i}} \succcurlyeq i^{2}$ and, if $\pi_\gamma$ is the nearest point projection to $\gamma$, we have $\diam(\pi_\gamma(x_{i})) = i$. 
\end{itemize}
Thus it makes sense for us to prove in general that a $\kappa$-Morse set is $\kappa'$-weakly contracting for some sublinear function $\kappa'$.

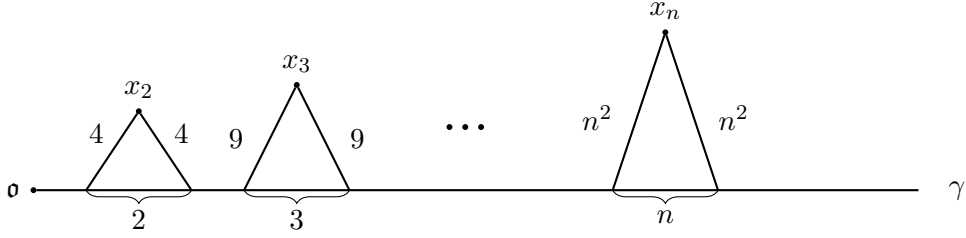
\begin{figure}[h!]
\begin{tikzpicture}[scale=0.7]
 \tikzstyle{vertex} =[circle,draw,fill=black,thick, inner sep=0pt,minimum size=.5 mm]
[thick, 
    scale=1,
    vertex/.style={circle,draw,fill=black,thick,
                   inner sep=0pt,minimum size= .5 mm},
                  
      trans/.style={thick,->, shorten >=6pt,shorten <=6pt,>=stealth},
   ]

   \node[vertex] (o) at (0,0)[label=left:$\go$] {}; 
   \node(a) at (17,0)[label=right:$\gamma$] {}; 
  
   \node[vertex](x1) at (2, 1.5)[label=above:$x_{2}$]{};
   \node[vertex](x2) at (5, 2)[label=above:$x_{3}$]{};
   \node[vertex](x4) at (12, 3)[label=above:$x_{n}$]{};
   
   \node[vertex] at (7.9, 1.2){};
   \node[vertex] at (8.2, 1.2){};
   \node[vertex] at (8.5, 1.2){};
  
   \draw [thick](x1)--(1,0) node[pos=.8, above, label=above:$4$]{};
   \draw [thick](x1)--(3,0)node[pos=.8, above, label=above:$4$]{};{};
   
   \draw [thick](x2)--(4,0)node[pos=.6, above, label=left:$9$]{};
   \draw  [thick](x2)--(6,0)node[pos=.6, above, label=right:$9$]{};
   
   \draw [thick](x4)--(11,0)node[pos=.6, above, label=left:$n^{2}$]{};
   \draw  [thick](x4)--(13,0)node[pos=.6, above, label=right:$n^{2}$]{};

    \draw [decorate,decoration={brace,amplitude=5pt},xshift=0pt,yshift=0pt]
  (3,0) -- (1,0)  node [black,midway,xshift=0pt,yshift=0pt, label=below: $2$] {};
  
    \draw [decorate,decoration={brace,amplitude=5pt},xshift=0pt,yshift=0pt]
  (6,0) -- (4,0)  node [black,midway,xshift=0pt,yshift=0pt, label=below: $3$] {};
  
    \draw [decorate,decoration={brace,amplitude=5pt},xshift=0pt,yshift=0pt]
  (13,0) -- (11,0)  node [black,midway,xshift=0pt,yshift=0pt, label=below: $n$] {};

    \draw[thick](o)--(a)[label=right:$\gamma$] {}; 
\end{tikzpicture}
\caption{A $1$-Morse geodesic ray which is $\sqrt{t}$-contracting.}
\end{figure}
\end{example}

To begin with, a set $Z$ is Morse if it is $\kappa$-Morse for $\kappa = 1$, and its Morse gauge is denoted by $ m_{Z}(q, Q) $. We say that a closed set $Z$
 is \emph{$\rho$-radius-contracting} if there exists a function $\rho$ such that, for each ball $B$ of radius $r$ that is disjoint from $Z$, the nearest point projection of $B$ to $Z$ is a set whose diameter is bounded above by $\rho(r)$. We present Proposition 4.2 in \cite{ACGH} here with notation adapted to that of this paper. Furthermore, we take into account that in the setting of this paper, nearest point projection is nonempty:

\begin{proposition}[Proposition 4.2 \cite{ACGH}]\label{Prop:Hume}
Let $Z$ be a closed subspace of a geodesic metric space $X$. Suppose $Z$ is Morse with Morse gauge $ m_{Z}(q, Q)$. Then there is a  sublinear function $\rho$, 
depending on $m_{Z}(q, Q)$, such that the nearest point projection of balls of radius $r$ (disjoint from $Z$) onto $Z$ is bounded above by $\rho(r)$.
 Specifically, $\rho$ is obtained as follows: 
\[ 
\rho(r) : = \sup_{s} \bigg\{  s \leq 4r \text{ and } s \leq 18 \, m_{Z}\left( \frac {12 r}{s}, 0  \right)\bigg \}.
\]
\end{proposition}

Given a point $x \in X$
 and a sublinearly Morse set $Z \subseteq X$, first we note that if a ball $B(x, r)$ is disjoint from $Z$, then  
 \[ r \leq d_X(x, Z) \leq d_X(x, \go) = \Norm x.\]
We call a set $Z$ \emph{sublinearly weakly contracting} if it is $\kappa$-weakly contracting for some sublinear function $\kappa$. 

\begin{proposition}\label{Morseimpliescontracting}
Let $Z$ be a $\kappa$-Morse set in a proper geodesic space $X$, with Morse gauge $m_{Z}(q, Q)$. 
Then $Z$ is sublinearly weakly contracting with respect to any $\kappa$-projection.
\end{proposition} 

\begin{proof}
It suffices to prove the statement for nearest point projections, since by Lemma~\ref{projection-property}, the uniform multiplicative error and the sublinear additive error will not contradict the conclusion. 
Consider a ball $B(x, r)$, disjoint from $Z$, and centered at 
$x$ with radius $r$. Observe first of all that $B(x, r)$ is inside the ball of radius $\Norm x+r \leq 2 \Norm x$.  Thus the distance between any point $y \in B(x, r)$ and a point in its nearest point projection $y_{1} \in \pi^{(near)}_{Z}(y)$ is also bounded above by $2 \Norm x$. We have by triangle inequality 
\[\Norm{y_{1}} \leq 4\Norm x.\]
That is to say, given a ball  $B(x, r)$, we only have to consider its projection to $Z \cap B(\go, 4 \Norm x)$.  
Since $Z$ is $\kappa$-Morse, the set $Z \cap B(\go, 4 \Norm x)$ is Morse with its Morse gauge $m_{Z}(q, Q) \cdot \kappa(4x)$. Now let $s$ denote the function that measures the diameter of the projection of disjoint ball $B(x, r)$ to $Z$. By  Proposition~\ref{Prop:Hume}, 
\[
s \leq 18 \, m_{Z}  \left(  \frac {12r}{s} , 0\right)\cdot \kappa(4x) \quad \text{by definition of $\kappa$-Morse.}
\]
Since $r \leq \Norm x$, we have 
\[
s \leq 18 \, m_{Z}\left(  \frac {12 \Norm x}{s} , 0 \right) \cdot \kappa(4 x).
\]
Suppose $s$ is not a sublinear function of $\Norm x$, that is to say,  as $\Norm x \to \infty$, there exists a sequence of disjoint balls $\{B_{i} \}$ with projections $\{ s_{i} \}$ such that there exists a positive number $c$ such that 
\[ \lim_{i \to \infty} \frac{s_{i}}{\Norm x_{i}} \geq c \]
then for every $\epsilon$, there exists $N$ such that for all $j > N$, 
\begin{align*} 
s_{j} &\leq 18 \, m_{Z}\left(  \frac {12 \Norm x_{j}}{s_{j}} , 0\right) \cdot \kappa(4x_{j})\\
   &\leq 18 \, m_{Z}\left(  12 \left(\frac 1c -\epsilon\right) , 0\right) \cdot \kappa(4  x_{j}).
\end{align*}
That is to say, $s$ is bounded above by a sublinear function of $\Norm x$, which means $s$ itself is a sublinear function of $\Norm x$, which is contrary to our assumption. Therefore, there does not exists such a sequence of balls 
and thus $s$ is a sublinear function of $\Norm x$, which we denote as $\kappa'(x)$.

Lastly, by Lemma~\ref{projection-property}, the same claim holds for all $\kappa$-projection maps.
\end{proof}

To summarize, we prove the equivalence between $\kappa$-weakly contracting
and $\kappa$-Morse (for a possibly different $\kappa$) for any given closed set:
\begin{theorem}\label{sublinearlyequivalence}
Let $(X, \go)$ be a proper geodesic metric space with a fixed base point. Let $Z$ be a closed set and $\pi$ be a 
$\kappa$-projection in the sense of Definition~\ref{weakprojection}.
The  following hold:
\begin{enumerate}
\item If $Z$ is $\kappa$-weakly contracting with respect to $\pi$,  then it is $\kappa$-Morse;
\item If $Z$ is $\kappa$-Morse, then it is $\kappa'$-weakly contracting with respect to $\pi$ for some sublinear function $\kappa'$.
\end{enumerate}
\end{theorem}
\begin{proof}
If $Z$ is $\kappa$ weakly contracting, then by Theorem~\ref{Def:generalContracting} it is 
$\kappa$-Morse. On the other hand, if $Z$ is $\kappa$-Morse, by Proposition~\ref{Morseimpliescontracting}  there exists a sublinear function $\kappa'$ for which $Z$ is $\kappa'$-weakly contracting. 
\end{proof}

\end{appendix}

\bibliographystyle{alpha}

\begin{thebibliography}{AKB12}

\bibitem[AK11]{AK11}
Y. Algom-Kfir, 
\newblock{\em  Strongly contracting geodesics in outer space}, 
\newblock{Geom. Topol. 15 (2011), no. 4, 2181--2233.}

\bibitem[ACGH17]{ACGH}
G. Arzhantseva, C. Cashen, D. Gruber, and D. Hume,
\newblock{\em Characterizations of Morse quasi-geodesics via superlinear divergence and sublinear contraction},
\newblock{Doc. Math. 22 (2017), 1193--1224.}

\bibitem[ACT15]{ACT}
G. Arzhantseva, C. Cashen and J. Tao, 
\newblock{\em Growth tight actions}, 
\newblock{Pacific J. Math. 278 (2015), no. 1, 1--49.}

\bibitem[Beh06]{Be06} 
J. Behrstock, 
\newblock{\em Asymptotic geometry of the mapping class group and Teichm\"uller space}, 
\newblock{Geom. Topol. 10 (2006), 1523--1578.}

\bibitem[BHS17]{HHG}
J. Behrstock, M. Hagen and A. Sisto, 
\newblock{\em Hierarchically hyperbolic spaces, I: Curve complexes for cubical groups}, 
\newblock{Geom. Topol. 21 (2017), no. 3, 1731--1804.}

\bibitem[BF09]{BF09}
M. Bestvina and K. Fujiwara, 
\newblock{ \em A characterization of higher rank symmetric spaces via bounded cohomology}, 
\newblock{Geom. Funct. Anal. 19 (2009), no. 1, 11--40.}

\bibitem[Bou98]{Bourbaki}
N. Bourbaki,
\newblock{\em General topology. Chapters 1-4},
\newblock{Translated from the French. Reprint of the 1989 English translation. Elements of Mathematics (Berlin)}, 
Springer-Verlag, Berlin, 1998. vii+437. 

\bibitem[Cas16]{cashen}
C. Cashen,
\newblock{\em Quasi-isometries need not induce homeomorphisms of contracting boundaries with the Gromov product topology},
\newblock{Anal. Geom. Metr. Spaces 4 (2016), no. 1, 278--281.}

\bibitem[CM19]{cashenmackay}
C. Cashen and J. Mackay,
\newblock{\em A metrizable topology on the contracting boundary of a group},
\newblock{Trans. Amer. Math. Soc. 372 (2019), no. 3, 1555--1600.}

\bibitem[CS15]{CS15}
R. Charney and H. Sultan, 
\newblock{\em Contracting boundaries of CAT(0) spaces},
\newblock{J. Topol. 8 (2015), no. 1, 93--117.}

\bibitem[CD60]{CD60}
G. Choquet and J. Deny,
 \newblock{\em Sur l' {\'e}quation de convolution $\mu = \mu * \sigma$},
 \newblock{C. R. Math. Acad. Sci. Paris 250 (1960), 799--801.}  

\bibitem[Cor18]{Morse}
M. Cordes,
\newblock {\em Morse boundaries of proper geodesic spaces},
\newblock {Groups Geom. Dyn. 11 (2017), no. 4, 1281--1306.}

\bibitem[CDG20]{CDG20}
M. Cordes, M. Dussaule, and I. Gekhtman, 
\newblock{\em An embedding of the Morse boundary in the Martin boundary}, 
\newblock{preprint arXiv:2004.14624}.

\bibitem[CK00]{CK00}
C. B. Croke and B. Kleiner, 
\newblock{\em Spaces with nonpositive curvature and their ideal boundaries},
\newblock{Topology 39 (2000), no. 3, 549--556.}

\bibitem[Dru00]{Drutu}
C. Drutu, 
\newblock{\em Quasi-isometric classification of non-uniform latticed in semisimple groups of higher rank}, 
\newblock{Geom. Funct. Anal. 10 (2000), 327--388.}

\bibitem[DR09]{Duchin-Rafi}
M. Duchin and K. Rafi,
\newblock{\em Divergence of Geodesics in Teichm\"uller Space and the Mapping Class Group}, 
\newblock{Geom. Funct. Anal. 19 (2009), 722--742.} 

\bibitem[DM60]{DM60}
 E. B. Dynkin and M. B. Maljutov, 
 \newblock{\em Random walk on groups with a finite number of generators},
 \newblock {Dokl. Akad. Nauk SSSR 137 (1961), 1042--1045.}

\bibitem[EFW12]{EFW12}
A. Eskin, D. Fisher and K. Whyte, 
\newblock{\em Coarse differentiation of quasi-isometries I: Spaces not quasi-isometric to Cayley graphs},
\newblock {Ann. of Math. (2) 176 (2012), no. 1, 221--260. }

\bibitem[EFW13]{EFW13}
A. Eskin, D. Fisher and K. Whyte, 
\newblock{\em Coarse differentiation of quasi-isometries II: Rigidity for Sol and lamplighter groups}, 
\newblock {Ann. of Math. (2) 177 (2013), no. 3, 869--910. }

\bibitem[EMR18]{EMR18}
A. Eskin, H. Masur and K. Rafi,
\newblock{\em Rigidity of Teichm\"uller space},
\newblock {Geom. Topol. 22 (2018), no. 7, 4259--4306. }

\bibitem[EMR17]{Projection}
A. Eskin, H. Masur and K. Rafi,
\newblock{\em Large scale rank of Teichm\"uller space}, 
\newblock{Duke Math. J. 166 (2017), no. 8, 1517--1572.}

\bibitem[Far98]{Far}
B. Farb, 
\newblock{\em Relatively hyperbolic groups}, 
\newblock{Geom. Funct. Anal. 8 (1998), 810--840.}

\bibitem[FM98]{braids}
B. Farb and H. Masur,
\newblock{\em  Superrigidity and mapping class groups},
\newblock{Topology 37 (1998), 1169--1176.}

\bibitem[Flo80]{Floyd}
W. Floyd, 
\newblock{\em Group completions and limit sets of Kleinian groups},
\newblock{Invent. Math. 57 (1980), 205--218.}

\bibitem[FT18]{FT}
B. Forghani and G. Tiozzo, 
\newblock{\em Shannon's theorem for locally compact groups}, 
\newblock{preprint arXiv:1812.07292.}

\bibitem[Fri37]{Frink}
A. Frink,
\newblock{\em Distance functions and the metrization problem},
\newblock{Bull. Amer. Math. Soc. 43 (1937), no. 2, 133--142.}

\bibitem[Fur63]{Furstenberg}
H. Furstenberg,
\newblock{\em A Poisson Formula for Semi-Simple Lie Groups},
\newblock {Ann. of Math. (2) 77 (1963), no. 2, 335--386}.

\bibitem[GM12]{GM}
F. Gautero and F. Math\'eus,
\newblock{\em Poisson boundary of groups acting on $\RR$-trees},
\newblock {Israel J. Math. 191 (2012), no. 2, 585--646.}

\bibitem[GQR20]{GQR}
I. Gekhtman, Y. Qing and K. Rafi, 
\newblock{\em QI-invariant model of Poisson boundaries of CAT(0) groups}, 
in preparation. 

\bibitem[Gro87]{Gro87}
M. Gromov, 
\newblock{{\em Hyperbolic groups}, in S. Gersten (ed.), {\em Essays in group theory}}, 
\newblock{Mathematical Sciences Research Institute Publications 8, Springer, New York, 1987, 75--263.}

\bibitem[Kai94]{Kai1}
V. Kaimanovich, 
\newblock {\em The Poisson boundary of hyperbolic groups},
\newblock {C. R. Acad. Sci. Paris 318 (1994), no. 1, 59--64.}

\bibitem[Kai00]{Kai00}
V. Kaimanovich, 
\newblock{\em The Poisson formula for groups with hyperbolic properties},
\newblock{Ann. of Math. (2) 152 (2000), no. 3, 659--692.}

\bibitem[KM96]{KM2}
V. Kaimanovich and H. Masur, 
\newblock{\em The Poisson boundary of the mapping class group},
\newblock {Invent. Math. 125 (1996), no. 2, 221--264. }

\bibitem[KM98]{KM}
V. Kaimanovich and H. Masur, 
\newblock{\em The Poisson boundary of Teichm{\"u}ller space},
\newblock {J. Funct. Anal. 156 (1998), no. 2, 301--332. }

\bibitem[Kar11]{asympCAT(0)}
A. Kar,
\newblock{\em Asymptotically CAT(0) groups},
\newblock {Publ. Mat. 55 (2011), no. 1, 67--91.}

\bibitem[KM99]{KarlMarg}
A. Karlsson and G. Margulis,
\newblock{\em A multiplicative ergodic theorem and nonpositively curved spaces},
\newblock{Comm. Math. Phys. 208 (1999), no. 1, 107--123. }

\bibitem[KN04]{KN}
A. Karlsson and G. Noskov, 
\newblock{\em Some groups having only elementary actions on metric spaces with hyperbolic boundaries}, 
\newblock{Geom. Dedicata 104 (2004), 119--137.}


\bibitem[Kla98]{Klarreich}
E. Klarreich, 
\newblock{\em The Boundary at Infinity of the Curve Complex and the Relative Teichm\"uller Space}, 
\newblock{preprint arXiv:1803.10339.}

\bibitem[Mah10]{Maher}
J. Maher,  
\newblock{\em Linear progress in the complex of curves},
\newblock{Trans. Amer. Math. Soc. 362 (2010), no. 6, 2963--2991.}

\bibitem[Mah12]{MaherExp}
J. Maher,
\newblock{\em Exponential decay in the mapping class group},
\newblock{J. Lond. Math. Soc. (2) 86 (2012), no. 2, 366--386.}

\bibitem[MT18]{MaherTiozzo}
J. Maher and G. Tiozzo, 
\newblock{Random walks on weakly hyperbolic groups}, 
\newblock{J. Reine Angew. Math. 742, (2018), 187--239.} 

\bibitem[MM00]{MM00}
H. Masur  and Y. Minsky,  
\newblock{\em Geometry of the complex of curves. II. Hierarchical structure},
\newblock{Geom. Funct. Anal. 10 (2000), no. 4, 902--974.}

\bibitem[Mor24]{Mor24}
H. M. Morse, 
\newblock{\em A fundamental class of geodesics on any closed surface of genus greater than one}, 
\newblock{Trans. Amer. Math. Soc. 26 (1924), 25--60.} 

\bibitem[Qin16]{qing1}
Y. Qing,
\newblock{\em Geometry of Right-Angled Coxeter Groups on the Croke-Kleiner Spaces},
\newblock{Geom. Dedicata 183 (2016), no. 1, 113--122.}

\bibitem[QRT19]{QRT19}
Y. Qing, K. Rafi and G. Tiozzo,
\newblock{\em Sublinearly Morse boundary I: CAT(0) spaces},
\newblock {preprint arXiv:1909.02096} 

\bibitem[QT19]{QT19}
Y. Qing and G. Tiozzo, 
\newblock{\em Excursions of generic geodesics in right-angled Artin groups and graph products}, 
\newblock{Int. Math. Res. Not. IMRN (2019), rnz294.}

\bibitem[RV18]{RV}
K. Rafi and Y. Verberne, 
\newblock{\em Geodesics in the mapping class groups}, 
\newblock{preprint arXiv:1810.12489.}

\bibitem[Ser83]{Series}
C. Series, 
\newblock{\em Martin boundaries of random walks on Fuchsian groups}, 
\newblock{Israel J. Math. 44  (1983), 221--242.}

\bibitem[Sis13]{sistopaper}
A. Sisto, 
 \newblock{\em Projections and relative hyperbolicity},
 \newblock{Enseign. Math. (2) 59 (2013), no. 1-2, 165--181.}

\bibitem[Sis17]{sisto-track}
A. Sisto,
\newblock{\em Tracking rates of random walks}, 
\newblock{Israel J. Math. 220 (2017), 1--28.}

\bibitem[Sis18]{Si18}
A. Sisto, 
\newblock{\em Contracting elements and random walks}, 
\newblock{J. Reine Angew. Math. 742 (2018), 79--114.}

\bibitem[ST19]{ST}
A. Sisto and S. Taylor,
\newblock {\em Largest projections for random walks and shortest curves for random mapping tori},
\newblock {Math. Res. Lett. 26 (2019) no. 1, 293--321.}

\bibitem[Tio15]{Ti15}
G. Tiozzo, 
\newblock{\em Sublinear deviation between geodesics and sample paths},
\newblock{Duke Math. J. 164 (2015), no. 3, 511--539.}

\bibitem[Woe00]{Woe}
W. Woess, 
\newblock{\em Random Walks on Infinite Graphs and Groups}, 
\newblock{Cambridge Tracts in Mathematics, Cambridge University Press, Cambridge, 2000.} 

\bibitem[Yan20]{Ya20}
W. Yang, 
\newblock{\em Genericity of contracting elements in groups}, 
\newblock{Math. Ann. 376 (2020), 823--861.}

\end{thebibliography}

\end{document}